\documentclass[a4paper]{amsart}
\usepackage{amsmath,amssymb}
\usepackage{latexsym}
\usepackage[all]{xy}
\usepackage[latin1]{inputenc}
\usepackage{graphicx}
\usepackage{array}
\usepackage{url}

\usepackage{color}

\usepackage[colorlinks=true]{hyperref}

\newcommand{\ie}{i.e. }
\newcommand{\loccit}{{loc. cit.}}

\newcommand{\cA}{{\mathcal A}}
\newcommand{\cB}{{\mathcal B}}
\newcommand{\cC}{{\mathcal C}}
\newcommand{\cD}{{\mathcal D}}
\newcommand{\cE}{{\mathcal E}}
\newcommand{\cK}{{\mathcal K}}

\newcommand{\cX}{{\mathcal X}}

\newcommand{\bbZ}{{\mathbb Z}}

\newcommand{\arr}[1]{\stackrel{#1}{\longrightarrow}}

\newcommand{\HHom}[2]{[#1,#2]} 
\newcommand{\TTens}{\otimes} 

\newcommand{\trafo}{\rightarrow} 

\newcommand{\Hom}{{\rm Hom}} 

\newcommand{\unit}{\eta} 
\newcommand{\counit}{\epsilon} 

\newcommand{\W}{{\rm W}} 

\newcommand{\diag}[1]{{\rm \fbox{$\scriptstyle #1$}}} 

\newcommand{\pair}[1]{\{ #1\}} 

\newcommand{\bid}{{\varpi}} 
\newcommand{\ev}{{\rm ev}} 
\newcommand{\coev}{{\rm coev}} 

\newcommand{\tpp}{{tp}} 
\newcommand{\thh}{{th}} 
\newcommand{\cc}{{s}} 
\newcommand{\q}{{\pi}} 
\newcommand{\qh}{{\kappa}} 
\newcommand{\fp}{{\alpha}} 
\newcommand{\fh}{{\beta}} 
\newcommand{\fg}{{\lambda}} 
\newcommand{\ff}{{\mu}} 
\newcommand{\ssp}{{\theta}} 
\newcommand{\sh}{{\nu}} %
\newcommand{\ea}{{a}} 
\newcommand{\eb}{{b}} 
\newcommand{\ec}{{c}} 
\newcommand{\rr}{{\zeta}} 
\newcommand{\exch}{{\varrho}} 
\newcommand{\dd}{{\tau}} 
\newcommand{\eps}{{\varepsilon}} 
\newcommand{\gam}{{\gamma}} 

\newcommand{\rel}{{\omega}} 
\newcommand{\abs}{{\omega}} 
\newcommand{\one}{{\mathbf 1}} 
\newcommand{\pt}{{\rm Pt}} 

\newcommand{\func}[1]{{|#1|}} 

\newcommand{\ttens}{\bullet} 
\newcommand{\hhom}{h} 
\newcommand{\ath}{ath} 
\newcommand{\asso}{asso} 
\newcommand{\ep}{\epsilon} 

\newcounter{mydiagram}
\newcommand{\diagram}{\refstepcounter{mydiagram} {\rm \fbox{$\scriptstyle \themydiagram$}}}

\newcounter{myassump}
\renewcommand{\themyassump}{\Alph{myassump}}

\newcommand{\assumption}[2]{{\bf(#1}$_{#2}${\bf)}}

\newcommand{\assum}[2]{\refstepcounter{myassump}{\rm\assumption{\themyassump}{#1} }#2\newline}

\theoremstyle{definition}
\newtheorem{defi}{Definition}[subsection]

\theoremstyle{plain}
\newtheorem{theo}[defi]{Theorem}

\newtheorem{coro}[defi]{Corollary}
\newtheorem{prop}[defi]{Proposition}
\newtheorem{lemm}[defi]{Lemma}
\newtheorem{nota}[defi]{Notation}

\theoremstyle{remark}
\newtheorem{rema}[defi]{Remark}
\newtheorem{exam}[defi]{Example}

\newcommand{\mypar}[1]{{\smallskip \noindent \bf #1.}}

\title{{\bf Tensor-triangulated categories and dualities}}
\author{Baptiste Calm{\`e}s and Jens Hornbostel}

\begin{document}
\bibliographystyle{amsplain}

\maketitle

\begin{abstract}
In a triangulated symmetric monoidal closed category, there are natural
dualities induced by the internal Hom. Given a monoidal exact functor $f^*$
between two such categories and adjoint couples $(f^*,f_*)$, $(f_*,f^!)$, we
establish the commutative diagrams necessary for $f^*$ and $f_*$ to respect
certain dualities, for a projection formula to hold between them (as duality
preserving exact functors) and for classical base change and composition
formulas to hold when such duality preserving functors are composed. This
framework allows us to define push-forwards for Witt groups, for example. 
\end{abstract}

\tableofcontents

\section*{Introduction}

Several cohomology theories on schemes use the concept of a duality in their
definition. In fact, these cohomology theories are usually defined at a
categorical level; a group is associated to a category with some additional
structure, including a duality which is then a contravariant endofunctor of
order two (up to a natural isomorphism) on the category. The main example 
that we have in mind is the Witt
group, defined for a triangulated category with duality, the (coherent or
locally free) Witt groups of a scheme then being defined by using one of the
various bounded derived categories of a scheme endowed with a duality coming from the
derived functor RHom. Of course, most of the structural properties of these
cohomology theories should be proved at a categorical level. For example, Paul
Balmer has proved localization for Witt groups directly at the level of
triangulated categories by using localization properties of such categories
\cite{Balmer00}. In the rest of the article, 
we only discuss the example of Witt groups, but
everything works exactly the same way for Grothendieck-Witt
groups. The reader should also have in mind that similar considerations 
should apply to hermitian K-theory and Chow-Witt theory, etc. 
 
We are interested in the functoriality of Witt groups along morphisms of schemes. At the categorical level, this means that we are given functors between categories (such as the derived functors of $f^*$ or $f_*$) and that we would like to use them to induce morphisms between Witt groups. Let us elaborate on this theme. The basic ideas are:
\begin{itemize}
\item a functor between categories respecting the duality induces a 
morphism of groups and 
\item two such morphisms can be compared if there is a morphism of functors between two such functors, again respecting the duality. 
\end{itemize}
By ``respecting the duality'', we actually mean the following. Let
$(\cC,D,\bid)$ be a category with duality, which means that $D: \cC \to \cC$
is a contravariant functor and that $\bid: Id_{\cC} \to D^2$ is a morphism of
functors such that for any object $A$, we have
\begin{equation} \label{bidualFormulaEq}
D(\bid_A) \circ \bid_{DA}=Id_{DA}.
\end{equation} 
A duality preserving functor from $(\cC_1,D_1,\bid_1)$ to $(\cC_2,D_2,\bid_2)$ is a pair $\pair{F,\phi}$ where $F: \cC_1 \to \cC_2$ is a functor and $\phi:FD \to DF$ is a morphism of functors, such that the diagram of morphism of functors
\begin{equation} \label{dualPresDiag}
\xymatrix{
F \ar[r]^{F\bid_1} \ar[d]_{\bid_2 F} & FD_1 D_1 \ar[d]^{\phi D_1} \\
D_2 D_2 F \ar[r]^{D_2 \phi} & D_2 F D_1
}
\end{equation}
commutes. In practice, the functor $F$ is usually given. For example, in the
context of schemes, it can be the derived functor of the pull-back or of the
push-forward along a morphism of schemes. It then remains to find an interesting $\phi$ and to check the commutativity of diagram \eqref{dualPresDiag}, which can be very intricate. 

A morphism of duality preserving functors $\pair{F,\phi} \to \pair{G,\psi}$ is a morphism of functors $\rho:F \to G$ such that the diagram
\begin{equation} \label{dualPresMorphDiag}
\xymatrix{
FD_1 \ar[r]^{\rho D_1} \ar[d]_{\phi} & GD_1 \ar[d]^{\psi} \\ 
D_2F & D_2 G \ar[l]^{D_2 \rho}
}
\end{equation}
commutes. This situation with two functors arises when we want to compare what
two different functors (respecting the dualities) yield as Witt group
morphisms. For example, we might want to compare the composition of two
pull-backs and the pull-back of the composition, in the case
of morphisms of schemes.
In this kind of situation, usually $\pair{F,\phi}$ and $\pair{G,\psi}$ are
already given, $\rho$ has to be found and the commutative diagram
\eqref{dualPresMorphDiag} has to be proved.

It turns out that in the context of symmetric monoidal closed categories, this type of diagram commutes for certain dualities, functors and morphisms of functors arising in a natural way. The main object of this article is to prove it.

Our original motivation for dealing with the questions of this paper
was to define push-forward
morphisms for Witt groups with respect
to proper maps of schemes and study their properties,
which in turn should be useful 
both for general theorems and concrete computations of Witt groups
of schemes. The article \cite{Calmes08b_pre}
uses the results of this article to establish these push-forwards. 
We refer the reader to \loccit{}
for details. As Witt groups are the main application that we have in mind, 
we have stated in corollaries of the main theorems what they imply for Witt
groups. These corollaries are trivial, and just rely on the classical
propositions \ref{dualPresIsTransfer_prop} and
\ref{samemorphismWitt_prop}. This should be seen as some help for the reader
familiar with cohomology theories and geometric intuition. The reader not
interested in Witt groups or other Witt-like cohomology theories might just
skip the corollaries. 
In all of them, the triangulated categories are assumed to be $\bbZ[1/2]$-linear (to be able to define Witt groups). 

\medskip

We now give a more precise description of the functors and morphisms of 
functors we consider. 

\mypar{Dualities}
Assume we are given a symmetric monoidal closed category $\cC$ were the tensor product is denoted by $\otimes$ and its ``right adjoint'', the internal Hom, is denoted by $[-,-]$. Then, as recalled in Section \ref{consClosedMon}, fixing an object $K$ and setting $D_K=[-,K]$, a natural morphism of functors $\bid_K:Id \to D_K D_K$ can be defined using the symmetry and the adjunction of the tensor product and the internal Hom and the formula \eqref{bidualFormulaEq} is satisfied. This is well known and recalled here just for the sake of completeness.  

\mypar{Pull-back}
Assume now that we are given a monoidal functor denoted by $f^*$ (by analogy with the algebro-geometric case) from $\cC_1$ to $\cC_2$, then, for any object $K$, as explained above, we can consider the categories with duality $(\cC_1,D_K,\bid_K)$ and $(\cC_2,D_{f^*K},\bid_{f^*K})$. There is a natural morphism of functors $\fh_K: f^*D_K \to D_{f^*K}f^*$ such that diagram \eqref{dualPresDiag} is commutative (see Proposition \ref{defifh} and Theorem \ref{PullBack0_theo}). In other words, $\pair{f^*,\fh_K}$ is a duality preserving functor from $(\cC_1,D_K,\bid_K)$ to $(\cC_2,D_{f^*K},\bid_{f^*K})$. 

\mypar{Push-forward}
Assume furthermore that that we are given a right adjoint $f_*$, then there is a natural morphism of projection $\q:f_*(-)\otimes * \to f_*(- \otimes f^*(*))$ (see Proposition \ref{projFormMorphExists}). When $f_*$ also has a right adjoint $f^!$, we can consider the categories with duality $(\cC_1,D_K,\bid_K)$ and $(\cC_2,D_{f^!K},\bid_{f^!K})$ and there is a natural isomorphism $\rr_K: f_* D_{f^!K} \to D_K f_*$ such that diagram \eqref{dualPresDiag} is commutative. In other words, $\pair{f_*,\rr_K}$ is a duality preserving functor from $(\cC_2,D_{f^!K},\bid_{f^!K})$ to $(\cC_1,D_K,\bid_K)$.

\mypar{Product}
Theorems involving products of dualities are stated. Given a pair of categories with duality $(\cC_1,D_1,\bid_1)$ and $(\cC_2,D_2,\bid_2)$, there is an obvious structure of a category with duality $(\cC_1\times \cC_2, D_1 \times D_2, \bid_1 \times \bid_2)$. If $\cC$ is monoidal symmetric closed, there is a natural morphism of (bi)functors $\dd_{K,M}:D_K \otimes D_M \to D(- \otimes -)_{K \otimes M}$ (see Definition \ref{defidd}). Proposition \ref{existProduct} recalls that $\pair{-\otimes-,\dd_{K,M}}$ is a duality preserving functor from $(\cC \times \cC, D_K \times D_M, \bid_K \times \bid_M)$ to $(\cC,D_{K \otimes M}, \bid_{K \otimes M})$. This gives a product on Witt groups.

\medskip
We now explain relations between the functors $\pair{f^*,\fh}$, $\pair{f_*,\rr}$ and $\pair{\TTens,\dd}$ in different contexts. Before going any further, let us remark that a morphism $\iota: K\to M$ induces a morphism of functors $\tilde{\iota}:D_K \to D_M$ which is easily shown to yield a duality preserving functor $I_\iota=\pair{Id,\tilde{\iota}}$ from $(\cC,D_K,\bid_K)$ to $\cC,D_M,\bid_M)$.

\mypar{Composition}
There is a natural way to compose duality preserving functors: $\pair{F,f}\circ \pair{G,g}:= \pair{FG,gF\circ fG}$. Suppose that we are given a pseudo contravariant functor $(-)^*$ from a category $\cB$ to a category where the objects are symmetric monoidal closed categories and the morphisms are monoidal functors. As usual, ``pseudo'' means that we have almost a functor, except that we only have a natural isomorphism $\ea_{g,f}:f^* g^* \simeq (gf)^*$ instead of an equality. Then an obvious question arises: can we compare $\pair{(gf)^*,\fh_K}$ with $\pair{f^*,\fh_{g^*K}} \circ \pair{g^*,\fh_{K}}$ for composable morphisms $f$ and $g$ in $\cB$. The answer is that the natural isomorphism $\ea$ is a morphism of duality preserving functors from $I_{\ea_{g,f,K}} \circ \pair{f^*,\fh_{g^*K}} \circ \pair{g^*,\fh_{K}}$ to $\pair{(gf)^*,\fh_K}$, \ie the diagram \eqref{dualPresMorphDiag} commutes (see \ref{compof^*_theo}) with $\rho=\ea$. The correction by $I_{\ea_{g,f,K}}$ is necessary for otherwise, strictly speaking, the target categories of the two duality preserving functors are not equipped with the same duality.

A similar composition question arises when some $f \in \cB$ are such that there are adjunctions $(f^*,f_*)$ and $(f_*,f^!)$. On the subcategory $\cB'$ of such $f$, there is a way of defining natural (with respect to the adjunctions) pseudo functor structures $\eb:(gf)_* \simeq g_* f_*$ and $\ec:f^!g^! \simeq (gf)^!$, and we can compare $\pair{(gf)_*,\rr_K}$ and $\pair{g_*,\rr_{g,f^!K}} \circ \pair{f_*,\rr_{f,K}}$ using the natural morphism of functors $\eb$ up to a small correction of the duality using $\ec$ as above (see Theorem \ref{compof_*0_theo}). 

\mypar{Base change}
A ``base change" question arises when we are given a commutative diagram  
$$\xymatrix{
V \ar[r]^{\bar{g}} \ar[d]_{\bar{f}} & Y \ar[d]^{f}  \\
X \ar[r]_{g} & Z 
}$$
in $\cB$ such that $g$ and $\bar{g}$ are in $\cB'$. In this situation, there is a natural morphism $\eps:f^* g_* \to \bar{g}_* \bar{f}^*$ (see Section \ref{BaseChange}). When $\eps$ is an isomorphism, there is a natural morphism $\gam:\bar{f}^*g^! \to \bar{g}^!f^*$. Starting from an object $K \in \cC_Z$, the morphism of functors $\gam$ is used to define a duality preserving functor $I_{\gam_K}$ from $(\cC_V,D_{\bar{f}^*g^!K},\bid_{\bar{f}^*g^!K})$ to $(\cC_V,D_{\bar{g}^!f^*K},\bid_{\bar{g}^!f^*K})$. How can we compare $\pair{f^*,\fh_K} \circ \pair{g_*,\rr_{K}}$ and $\pair{\bar{g}_*,\rr_{f^*K}}\circ I_{\gam_K} \circ \pair{\bar{f}^*,\fh_{g^!K}}$? Theorem \ref{basechange0_theo} proves that $\eps$ defines a duality preserving functor from the first to the second. 

\mypar{Pull-back and product}
Given a monoidal functor $f^*$ the isomorphism $\fp: f^*(-)\otimes f^*(-) \to f^*(-\otimes -)$ (defining $f^*$ as a monoidal functor) defines a duality preserving functor $I_{\fp_{K,M}}$ from $(\cC,D_{f^*K \otimes f^*M},\bid_{f^*K \otimes f^*M})$ to $(\cC,D_{f^*(K \otimes M)},\bid_{f^*(K \otimes M)})$. Proposition \ref{fpDualPres} shows that the morphism of functors $\fp$ then defines a duality preserving morphism of functors from $I_{\fp_{K,M}} \circ \pair{-\otimes-,\dd_{f^*K,f^*M}}\circ \pair{f^* \times f^*,\fh_K \times \fh_M}$ to $\pair{f^*,\fh_{K \otimes M}} \circ \pair{-\otimes-,\dd_{K,M}}$. It essentially means that the pull-back is a ring morphism on Witt groups. 

\mypar{Projection formula}
Given a monoidal functor $f^*$ with adjunctions $(f^*,f_*)$, $(f_*,f^!)$ such that $\q$ is an isomorphism, there is a natural morphism of functors $\ssp: f^!(-) \otimes f^*(-) \to f^!(-\otimes -)$. This defines a duality preserving functor $I_{\ssp_{K,M}}$ from $(\cC_2,D_{f^!K \otimes f^*M},\bid_{f^!K \otimes f^*M})$ to $(\cC_2,D_{f^!(K \otimes M)},\bid_{f^!(K \otimes M)})$. Theorem \ref{projform0_theo} proves that $\q$ is a duality preserving morphism of functors from $\pair{-\otimes-, \dd_{K,M}} \circ \pair{f_*\times Id,\rr_K \times id}$ to $\pair{f_*,\rr_{K \otimes M}} \circ I_{\ssp_{K,M}} \circ \pair{-\TTens -, \dd_{f^!K,f^*M}} \circ \pair{Id \times f^*,id \times \fh_M}$.

\mypar{Suspensions and triangulated structures}
As it is required for the example of Witt groups, all the results are proved as well in the setting of {\it triangulated} symmetric monoidal categories, functors and morphisms of functors respecting the triangulated structure. This means essentially two things. First, the categories are endowed with an autoequivalence $T$ (called suspension in the article) and the tensor products defining monoidal structures respect this suspension, \ie such a tensor product comes equipped with a pair of isomorphism of functors such that a certain diagram anticommutes (see Definition \ref{defiSuspBif}). As proved in the article, it follows that all the other functors (or bifunctors) mentioned above can be endowed by construction (see Point 1 of Propositions \ref{suspAdjExists} and \ref{suspACRBExists}) with morphisms of functors to commute with the suspensions in a suitable way involving commutative diagrams (similar to \eqref{dualPresDiag} with $T$ instead of $D$). All the morphisms of functors involved then also respect the suspension, \ie satisfy commutative diagrams similar to \eqref{dualPresMorphDiag}. This is important, because in the main applications, checking this kind of things by hand amounts to checking signs involved in the definitions of complexes or maps of complexes, and it has been the source of errors in the literature. Here, we avoid such potential errors by construction (there are no signs and no complexes). Second, using these morphisms of functors to deal with the suspensions, the functors should respect the collection of distinguished triangles. This is easy and has nothing to do with commutative diagrams. It is again obtained by construction (see Point 2 of Propositions \ref{suspAdjExists} and \ref{suspACRBExists}), except for the first variable of the internal Hom, for which it has to be assumed. The reader who is not interested in triangulated categories, but rather only in monoidal structures can just forget about all the parts of the statements involving $T$. What remains is the monoidal structure. 

These functors and morphisms between them are discussed in the first five
sections of this article. Section \ref{variousReform_sec} is devoted to reformulations of the main results when the base category $\cB$ is enriched in order that the objects also specify the duality: they are pairs $(X,K)$ where $K \in X^*$ is used to form the duality $D_K$. The morphisms can then be chosen in different ways (see Sections \ref{fofUnitObj_sec}, \ref{reformcorrectdual_sec} and \ref{reformFinalObject_sec}). These sections are directed towards geometric applications. In particular, the object $\rel_f$ introduced at the beginning of \ref{fofUnitObj_sec} has a clear geometric interpretation in the case of morphisms of schemes: it is the relative canonical sheaf. 
Finally, we have included an appendix to recall symmetric monoidal closed
structures on categories involving chain complexes and to relate the sign
choices involved with the ones made by various authors working with 
Witt groups. 

\medskip

\mypar{Remarks on the proofs}
Nearly every construction of a morphism in the article is based on variations on a single lemma on adjunctions, namely Lemma \ref{adjab}. Its refinements are Theorem \ref{adjsquare}, Lemma \ref{adjabbif} and Theorem
\ref{theoGenAdj}. Similarly, to establish commutativities of natural
transformations arising from by Lemma \ref{adjab} the main tool
is Lemma \ref{cube}. The whole Section \ref{DualFormAdj} is devoted to 
these formal results about adjunctions.

In some applications, certain natural morphisms considered in this article
have to be isomorphisms. For example, to define Witt groups, one requires that
the morphism of functors $\bid_K$ is an isomorphism. This is not important to
prove the commutativity of the diagrams, therefore it is not part of our
assumptions. Nevertheless, in some cases, it might be useful to know that if a
particular morphism is an isomorphism, then another one is also automatically
an isomorphism. When possible, we have included such statements, for example
in Propositions \ref{projFormIso} and \ref{adjpairsf_*_prop}. The attribute
``strong" in Definitions \ref{catwithduality},
\ref{DualPresFunc_defi} and \ref{MorphDualPresFunc_defi} is part of this
philosophy. By contrast, two isomorphism assumptions are important to the
abstract setting regardless of applications. This is the case of the
assumption on $\q$, since the essential morphism $\ssp$ is {\it defined} using
the inverse of $\q$, as well as the assumption on $\eps$ whose inverse is
involved in the definition of $\gam$. These isomorphism assumptions are 
listed at the beginning of Section \ref{FunctClosedMon} and are recalled in every theorem where they are used.

\medskip

Considering possible future applications (e.g. motivic homotopy 
categories or the stable homotopy category), we have presented some aspects 
in as much generality as possible.
Conversely, we do not attempt to provide a complete
list of articles where parts of the general framework we study has
already been considered. See however \cite{Fausk03} for a recent reference
written by homotopy theorists which contains many further references. 

Finally, let us mention a question a category theorist might ask when reading such a paper: are there coherence theorems, in the spirit of \cite{Kelly71}, that would prove systematically the needed commutative diagrams? As far as the authors know, there are no such coherence theorems available for the moment. Although it is certainly an interesting question, it is unclear (to the authors) how to formulate coherence statements. One problem is that, as mentioned above, part of the interesting commutative diagrams involve morphisms that are defined using inverses of natural morphisms. Without those inverted natural morphisms, the commutativity of the diagrams become meaningless. So, the interested reader might consider this article as a source of inspiration for future coherence theorems.

\section{Adjunctions and consequences} \label{DualFormAdj}

\subsection{Notations and conventions} \label{NotConv}

The opposite category of a category $\cC$ is denoted by $\cC^o$.

When $F$ and $G$ are functors with same source and target, we denote a morphism of functors between them as $t: F \trafo G$. 
When $s: G \trafo H$ is another one, their composition is denoted by $s \circ t$. When $F_1, F_2 : \cC \to \cD$, $G_1, G_2: \cD \to \cE$, $f:F_1 \trafo F_2$ and $g: G_1 \trafo G_2$, we denote by $gf$ the morphism of functors defined by $(fg)_A=G_2(f_A)\circ g_{F_1(A)}=g_{F_2(A)}\circ G_1(f_A)$ on any object $A$. When $F_1=F_2=F$ and $f=id_F$ (resp. $G_1=G_2=G$ and $g=id_G$), we usually use the notation $gF$ (resp. $Gf$). With this convention, $gf=G_2 f \circ g F_1= g F_2 \circ G_1 f$. 
When a commutative diagram is obtained by this equality or other properties immediate from the definition of a morphism of functors, we just put an $\diag{mf}$ label on it and avoid further justification. To save space, it may happen that when labeling maps in diagrams, we drop the functors from the notation, and just keep the important part, that is the morphism of functors (thus $FgH$ might be reduced to $g$). Many of the commutative diagrams in the article will be labeled by a symbol in a box (letters or numbers, such as in $\diag{H}$ or $\diag{3}$). When they are used in another commutative diagram, eventually after applying a functor to them, we just label them with the same symbol, so that the reader recognizes them, but without further comment. 

\subsection{Useful properties of adjunctions} \label{UsefulAdj}

This section is devoted to easy facts and theorems about adjunctions, that are repeatedly used throughout this article. All these facts are obvious, and we only prove the ones that are not completely classical. A good reference for the background on categories and adjunctions as discussed here is \cite{MacLane98}. 

\begin{defi}\label{defadjcouple}
An adjoint couple $(L,R)$ is the data consisting of two functors $L: \cC \to \cD$ and $R:\cD \to \cC$ and equivalently:
\begin{itemize}
\item  a bijection $\Hom(LA,B)\simeq \Hom(A,RB)$, functorial in $A\in \cC$ and $B \in \cD$, or 
\item two morphism of functors $\unit:Id_{\cC} \to RL$ and $\counit:LR \to Id_{\cD}$, called respectively unit and counit, such that the resulting compositions $R
\stackrel{\unit R}{\to} RLR \stackrel{R \counit}{\to} R$ and 
$L \stackrel{L \unit}{\to} LRL \stackrel{\counit L}{\to} L$ are identities.
\end{itemize}
In the couple, $L$ is called the left adjoint and $R$ the right adjoint. When we want to specify the unit and counit of the couple and the categories involved, we say $(L,R,\unit,\counit)$ is an adjoint couple from $\cC$ to $\cD$.
\end{defi}

When the commutativity of a diagram follows by one of the above compositions giving the identity, we label it $\diag{adj}$.

\begin{rema}
Adjunctions between functors that are contravariant can be considered in two different ways, by taking the opposite category of the source of $L$ or $R$. This does {\it not} lead to the same notion, essentially because if $(L,R)$ is an adjoint couple, then $(R^o,L^o)$ is an adjoint couple (instead of $(L^o,R^o)$). For this reason, we only use covariant functors in adjoint couples.
\end{rema}

\begin{lemm} \label{isoadjoints}
Let $(L,R,\unit,\counit)$ and $(L',R',\unit',\counit')$ be two adjoint couples
between the same categories $\cC$ and $\cD$, and let $l: L \trafo L'$ (resp. $r: R \trafo R'$) be an isomorphism. Then, there is a unique isomorphism $r: R \trafo R'$ (resp. $l: L \trafo L'$) such that $\unit'=r l \circ \unit$ and $\counit'= \counit \circ l^{-1} r^{-1}$. In particular, a right (resp. left) adjoint is unique up to unique isomoprhism.
\end{lemm}
\begin{proof}
The morphism $r$ is given by the composition $R' \counit \circ R' l^{-1} R \circ \unit' R$ and its inverse by the composition $R \counit' \circ R l R' \circ \unit R'$.
\end{proof}

\begin{lemm}
An equivalence of categories is an adjoint couple $(F,G,a,b)$ for which the unit and counit are isomorphisms. In particular, $(G,F,b,a)$ is also an adjoint couple.
\end{lemm}

\begin{lemm} \label{composeadj}
Let $(L,R,\unit,\counit)$ (resp. $(L',R',\unit',\counit')$) be an adjoint couple from $\cC$ to $\cD$ (resp. from $\cD$ to $\cE$). Then $(L'L,RR',R\unit'L\circ \unit,\counit'\circ L'\counit R')$ is an adjoint couple from $\cC$ to $\cE$.
\end{lemm}

We now turn to a series of less standard results, nevertheless very easy.

\begin{lemm} (mates)\label{adjab}
Let $H$, $H'$, $J_1$, $K_1$, $J_2$ and $K_2$ be functors with sources and targets as on the following diagram.
$$\xymatrix{
\cC_1 \ar@<0.5ex>[d]^{K_1} \ar[r]^{H} & \cC_2 \ar@<0.5ex>[d]^{K_2} \\
\cC_1' \ar@<0.5ex>[u]^{J_1} \ar[r]_{H'} & \cC_2' \ar@<0.5ex>[u]^{J_2} \\
}$$
Assume $(J_i,K_i,\unit_i,\counit_i)$, $i=1,2$ are adjoint couples. Let $a: J_2 H' \trafo H J_1$ (resp. $b: H' K_1 \trafo K_2 H$) be a morphism of functors. Then there exists a unique morphism of functors $b: H' K_1 \trafo K_2 H$ (resp. $a: J_2 H' \trafo H J_1$) such that the diagrams
$$\begin{array}{c}
\xymatrix{
J_2 H' K_1  \ar[d]_{aK_1} \ar[r]^{J_2 b} \ar@{}[dr]|{\diag{H}} & J_2 K_2 H \ar[d]^{\eta_2 H} \\
H J_1 K_1 \ar[r]_{H \counit_1} & H 
}\end{array}\hfill
\hspace{5ex}
and
\hspace{5ex}\begin{array}{c} 
\xymatrix{
H'  \ar[d]_{\unit_2 H'} \ar[r]^{H'\unit_1} \ar@{}[dr]|{\diag{H'}} & H' K_1 J_1 \ar[d]^{b J_1}  \\
K_2 J_2 H' \ar[r]_{K_2 a} & K_2 H J_1
}
\end{array}$$ 
are commutative. Furthermore, given two morphisms of functors $a$ and $b$,
the commutativity of one diagram is equivalent to the commutativity of the other one. In this situation, we say that $a$ and $b$ are {\it mates} with respect to the rest of the data.
\end{lemm}
\begin{proof}
We only prove that the existence of $a$ implies the uniqueness and existence of $b$, the proof of the other case is similar. Assume that $b$ exists and makes the diagrams
commutative. The commutative diagram
$$
\xymatrix{
H'K_1  \ar[d]_{\unit_2 H'K_1} \ar[r]^{H'\unit_1 K_1} \ar@{}[dr]|{\diag{H'}} & H'K_1 J_1 K_1 \ar[r]^-{H'K_1 \counit_1} \ar[d] \ar@{}[dr]|{\diag{mf}} & H'K_1 \ar[d]^b  \\
K_2 J_2 H' K_1 \ar[r]_{K_2 a K_1} & K_2 H J_1 K_1 \ar[r]_-{K_2 H \counit_1} & K_2 H 
}
$$
in which the upper horizontal composition is the identity of $H'K_1$ (by adjunction) shows that $b$ has to be given by the composition
$$\xymatrix@C+2ex{H'K_1 \ar[r]^-{\unit_2 H'K_1} & K_2 J_2 H'K_1 \ar[r]^{K_2 a K_1} & K_2 H J_1 K_1 \ar[r]^-{K_2 H \counit_1} & K_2 H}.$$
This proves uniqueness. Now let $b$ be given by the above composition. The commutative diagram
$$
\xymatrix{
J_2 H' K_1 \ar[r] \ar`u[r]-/d5ex/`[rrr]_{J_2 b}[rrr] \ar@{=}[dr] & J_2 K_2 J_2 H' K_1 \ar[d] \ar[r] \ar@{}[dr]|{\diag{mf}} & J_2 K_2 H J_1 K_1 \ar[d] \ar[r] \ar@{}[dr]|{\diag{mf}} & J_2 K_2 H \ar[d]^{\counit_2 H} \\
\ar@{}@<1ex>[ur]_(.75){\diag{adj}} & J_2 H' K_1 \ar[r]_{aK_1} & H J_1 K_1 \ar[r]_{H \counit_1} & H \\
}
$$
proves $\diag{H}$ and the commutative diagram
$$
\xymatrix{
H'K_1 J_1 \ar[r] \ar`u[r]-/d5ex/`[rrr]_{bJ_1}[rrr] \ar@{}[dr]|{\diag{mf}} & K_2 J_2 H' K_1 J_1 \ar[r] \ar@{}[dr]|{\diag{mf}} & K_2 H J_1 K_1 J_1 \ar[r] \ar@{}@<1ex>[dr]_(.25){\diag{adj}} & K_2 H J_1 \\
H'  \ar[u]^{H'\unit_1} \ar[r]_{\unit_2 H'} & K_2 J_2 H' \ar[u] \ar[r]_{K_2 a} & K_2 H J_1 \ar[u] \ar@{=}[ur] & \\
}
$$
proves $\diag{H'}$. The fact that the commutativity of one of the diagrams implies 
commutativity to the other is left to the reader.
\end{proof}

\begin{lemm} \label{cube}
Let us consider a cube of functors and morphisms of functors
$$\begin{array}{c}
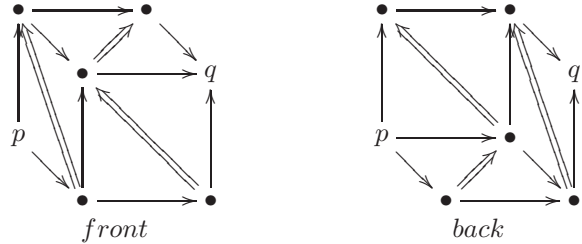
\xymatrix@!0{
\bullet \ar[rr] \ar[dr] & & \bullet \ar[dr] & \\
 & \bullet \ar@2[ur] \ar[rr] & & q \\
p \ar[uu] \ar[dr] & & & \\
 & \bullet \ar@2[luuu] \ar[uu] \ar[rr] & & \bullet \ar@2[lluu] \ar[uu]
} \\
front
\end{array}
\hspace{10ex}
\begin{array}{c}
\xymatrix@!0{
\bullet \ar[rr] & & \bullet \ar[dr] & \\
 & & & q \\
p \ar[uu] \ar[dr] \ar[rr] & & \bullet \ar@2[lluu] \ar[uu] \ar[dr] & \\
 & \bullet \ar@2[ur] \ar[rr] & & \bullet \ar@2[luuu] \ar[uu]
} \\
back
\end{array}$$
that is commutative in the following sense: The morphism between the
two outer compositions of functors from $p$ to $q$ given
by the composition of the three morphisms of functors of the front is equal to the composition of the three morphism of functors of the back. Assume that the vertical maps have right adjoints. Then, by Lemma \ref{adjab} applied to the vertical squares, we obtain the following cube (the top and bottom squares have not changed).
$$\begin{array}{c}
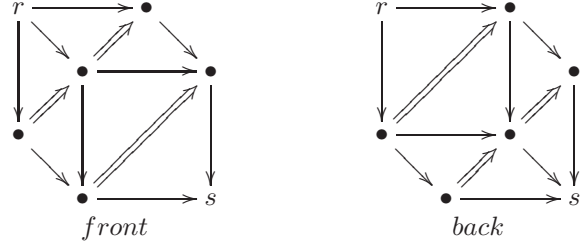
\xymatrix@!0{
r \ar[dd] \ar[rr] \ar[dr] & & \bullet \ar[dr] & \\
 & \bullet \ar[dd] \ar@2[ur] \ar[rr] & & \bullet \ar[dd] \\
\bullet \ar@2[ur] \ar[dr] & & & \\
 & \bullet \ar@2[uurr] \ar[rr] & & s 
} \\
front
\end{array}
\hspace{10ex}
\begin{array}{c}
\xymatrix@!0{
r \ar[dd] \ar[rr] & & \bullet \ar[dd] \ar[dr] & \\
 & & & \bullet \ar[dd] \\
\bullet \ar[dr] \ar@2[uurr] \ar[rr] & & \bullet \ar[dr] \ar@2[ur] & \\
 & \bullet \ar@2[ur] \ar[rr] & & s 
} \\
back
\end{array}$$
This cube is commutative (in the sense just defined, using $r$ and $s$ instead of $p$ and $q$).
\end{lemm}
\begin{proof}
This is straightforward, using the commutative diagrams of Lemma \ref{adjab}, and left to the reader.
\end{proof}

We now use Lemma \ref{adjab} to prove a theorem, which doesn't contain a lot more than the lemma, but is stated in a convenient way for future reference in the applications
we are interested in. 

\begin{theo}\label{adjsquare}
Let $L$, $R$, $L'$, $R'$, $F_1$, $G_1$, $F_2$, $G_2$ be functors whose sources and targets are specified by the diagram 
$$\xymatrix{
\cC_1 \ar@<0.5ex>[d]^{G_1} \ar@<0.5ex>[r]^{L} & \cC_2 \ar@<0.5ex>[l]^{R} \ar@<0.5ex>[d]^{G_2} \\
\cC_1' \ar@<0.5ex>[u]^{F_1} \ar@<0.5ex>[r]^{L'} & \cC_2'. \ar@<0.5ex>[u]^{F_2} \ar@<0.5ex>[l]^{R'} \\
}$$
We will study morphisms of functors $f_L$, $f'_L$, $g_L$, $g'_L$, $f_R$, $f'_R$, $g_R$ and $g'_R$ whose sources and targets will be as follows:
$$
\begin{array}{lr}
\xymatrix{
LF_1 \ar@<0.5ex>[r]^{f_L} & F_2 L' \ar@<0.5ex>[l]^{f'_L} 
}
&
\xymatrix{
L'G_1 \ar@<0.5ex>[r]^{g'_L} & G_2 L \ar@<0.5ex>[l]^{g_L}
}
\\
\xymatrix{
F_1 R' \ar@<0.5ex>[r]^{f'_R} & RF_2 \ar@<0.5ex>[l]^{f_R}
}
&
\xymatrix{
G_1 R \ar@<0.5ex>[r]^{g_R} & R' G_2 \ar@<0.5ex>[l]^{g'_R}
} \\
\end{array}
$$
Let us consider the following diagrams, in which the maps and their directions will be the obvious ones induced by the eight maps above and the adjunctions accordingly to the 
different cases discussed below.
$$\xymatrix{
F_2 L' G_1 \ar@{-}[d] \ar@{-}[r] \ar@{}[dr]|{\diag{L}} & L F_1 G_1 \ar@{-}[d] \\
F_2 G_2 L \ar@{-}[r] & L
}
\hspace{15ex}
\xymatrix{
L' \ar@{-}[d] \ar@{-}[r] \ar@{}[dr]|{\diag{L'}} & G_2 F_2 L' \ar@{-}[d] \\
L' G_1 F_1 \ar@{-}[r] & G_2 L F_1
}
$$
$$\xymatrix{
F_1 R' G_2 \ar@{-}[d] \ar@{-}[r] \ar@{}[dr]|{\diag{R}} & R F_2 G_2 \ar@{-}[d] \\
F_1 G_1 R \ar@{-}[r] & R 
}
\hspace{15ex}
\xymatrix{
R' \ar@{-}[d] \ar@{-}[r] \ar@{}[dr]|{\diag{R'}} & G_1 F_1 R' \ar@{-}[d] \\
R' G_2 F_2 \ar@{-}[r] & G_1 R F_2
}
$$
$$\xymatrix{
G_1 \ar[d] \ar[r] \ar@{}[dr]|{\diag{G_1}} & G_1 R L \ar@{-}[d] \\
R'L'G_1 \ar@{-}[r] & R'G_2 L 
}
\hspace{15ex}
\xymatrix{
L'G_1 R \ar@{-}[d] \ar@{-}[r] \ar@{}[dr]|{\diag{G_2}} & G_2 L R \ar[d] \\
L'R' G_2 \ar[r] & G_2
}
$$
$$\xymatrix{
F_1 \ar[d] \ar[r] \ar@{}[dr]|{\diag{F_1}} & F_1 R' L' \ar@{-}[d] \\
RLF_1 \ar@{-}[r] & RF_2 L'
}
\hspace{15ex}
\xymatrix{
LF_1 R' \ar@{-}[d] \ar@{-}[r] \ar@{}[dr]|{\diag{F_2}} & F_2 L' R' \ar[d] \\
LR F_2 \ar[r] & F_2
}
$$
Then
\begin{itemize}
\item[1.] Let $(G_i,F_i)$, $i=1,2$ be adjoint couples. Let $g_L$ (resp. $f_L$) be given, then there is a unique $f_L$ (resp. $g_L$) such that $\diag{L}$ and $\diag{L'}$ are commutative. Let $g_R$ (resp. $f_R$) be given, then there is a unique $f_R$ (resp. $g_R$) such that $\diag{R}$ and $\diag{R'}$ are commutative.
\item[1'.] Let $(F_i,G_i)$, $i=1,2$ be adjoint couples. Let $g'_L$ (resp. $f'_L$) be given, then there is a unique $f'_L$ (resp. $g'_L$) such that $\diag{L}$ and $\diag{L'}$ are commutative. Let $g'_R$ (resp. $f'_R$) be given, then there is a unique $f'_R$ (resp. $g'_R$) such that $\diag{R}$ and $\diag{R'}$ are commutative.
\item[2.] Let $(L,R)$ and $(L',R')$ be adjoint couples. Let $f_L$ (resp. $f'_R$) be given, then there is a unique $f'_R$ (resp. $f_L$) such that $\diag{F_1}$ and $\diag{F_2}$ are commutative. Let $g'_L$ (resp. $g_R$) be given, then there is a unique $g_R$ (resp. $g'_L$) such that $\diag{G_1}$ and $\diag{G_2}$ are commutative.
\item[3.] Assuming $(G_i,F_i)$, $i=1,2$, $(L,R)$ and $(L',R')$ are adjoint couples, and $g_L$, $g'_L=g_L^{-1}$ are given (resp. $f_R$ and $f'_R=f_R^{-1}$). By 1 and 2, we obtain $f_L$ and $g_R$ (resp. $g_R$ and $f_L$). We then may construct $f_R$ and $f'_R$ (resp. $g_L$ and $g'_L$) which are inverse to each other.  
\item[3'.] Assuming $(F_i,G_i)$, $i=1,2$, $(L,R)$ and $(L',R')$ are adjoint couples, $f_L$ and $f'_L=f_L^{-1}$ are given (resp. $g_R$ and $g'_R=g_R^{-1}$). By 1' and 2, we obtain $g'_L$ and $f'_R$ (resp. $f'_R$ and $g'_L$). We then may construct $g'_R$ and $g_R$ (resp. $f'_L$ and $f_L$) which are inverse to each other.  
\end{itemize}
\end{theo}
\begin{proof}
Points 1, 1' and 2 are obvious translations of the previous lemma. We only prove Point 3, since 3' is dual to it. Let $(L,R,\unit,\counit)$, $(L',R',\unit',\counit')$ and $(G_i,F_i,\unit_i,\counit_i)$, $i=1,2$, be the adjoint couples. Using 1 and 2, we first obtain $f_L$ and $g_R$, as well as the commutative diagrams $\diag{L}$, $\diag{L'}$ (both involving $g_L=(g'_L)^{-1}$), $\diag{G_1}$ and $\diag{G_2}$ (both involving $g'_L=(g_L)^{-1}$). The morphisms of functors $f'_R$ and $f_R$ are respectively defined by the compositions
$$\xymatrix{F_1 R' \ar[r]^-{\unit F_1 R'} & RL F_1 R' \ar[r]^{R f_L R'} & RF_2L'R' \ar[r]^-{RF_2 \counit'} & RF_2}$$
and
$$\xymatrix{RF_2 \ar[r]^-{\unit_1 RF_2} & F_1 G_1 RF_2 \ar[r]^{F_1 g_R F_2} & F_1 R' G_2 F_2 \ar[r]^-{F_1 R' \counit_2} & F_1 R'}.$$
We compute $f_R \circ f'_R$ as the upper right composition of the following commutative diagram
$$
\xymatrix@R=3ex{
F_1 R' \ar[d]_{\unit_1} \ar[r]^{\unit} \ar@{}[dr]|{\diag{mf}} & RLF_1 R' \ar[d] \ar[r]^{f_L} \ar@{}[dr]|{\diag{mf}} & RF_2 L'R' \ar[d] \ar[r]^{\counit'} \ar@{}[dr]|{\diag{mf}} & RF_2 \ar[d]^{\unit_1} \\
F_1 G_1 F_1 R' \ar[r] \ar@{}[dr]|{\diag{G_1}} \ar@/_/[ddr]_{\unit'} & F_1 G_1 RL F_1 R' \ar[d] \ar[r] \ar@{}[dr]|{\diag{mf}} & F_1 G_1 R F_2 L' R' \ar[d] \ar[r] \ar@{}[dr]|{\diag{mf}} & F_1 G_1 R F_2 \ar[d]^{g_R} \\
 & F_1 R' G_2 L F_1 R' \ar[d]|(0.35){(g'_L)^{-1}=g_L} \ar[r] \ar@{}[dr]|{\diag{L'}} & F_1 R' G_2 F_2 L' R' \ar[d] \ar[r] \ar@{}[dr]|{\diag{mf}} & F_1 R' G_2 F_2 \ar[d]^{\counit_2} \\
  & F_1 R' L' G_1 F_1 R' \ar[r]_{\counit_1} & F_1 R' L' R' \ar[r]_{\counit'} & F_1 R' 
}
$$
The lower left composition in the above diagram is the identity because it appears as the upper right composition of the commutative diagram
$$
\xymatrix@R=4ex{F_1 R' \ar[r]^{\unit_1} \ar@{=}[dr] & F_1 G_1 F_1 R' \ar[d] \ar[r]^{\unit'} \ar@{}[dr]|{\diag{mf}} & F_1 R' L' G_1 F_1 R' \ar[d]^{\counit_1} \\
 \ar@{}@<1ex>[ur]_(.75){\diag{adj}} & F_1 R' \ar[r] \ar@{=}[dr] & F_1 R' L' R' \ar[d]^{\counit'} \\
 & \ar@{}@<1ex>[ur]_(.75){\diag{adj}} & F_1 R' 
}.
$$
The composition $f'_R \circ f_R=id$ is proved in a similar way, involving the diagrams $\diag{L}$ and $\diag{G_2}$.
\end{proof}

The reader has certainly noticed that there is a statement $2'$ 
which we didn't spell out because we don't need it.

\subsection{Bifunctors and adjunctions} \label{BifandAdj}

We have to deal with couples of bifunctors that give adjoint couples of usual functors when one of the entries in the bifunctors is fixed. We need to explain how these adjunctions are functorial in this entry. The standard example for that is the classical adjunction between tensor product and internal Hom. 

\begin{defi} \label{defiAdjBif}
Let $\cX$, $\cC$, $\cC'$ be three categories, and let $L: \cX \times \cC' \to \cC$ and $R: \cX^o \times \cC \to \cC'$ be bifunctors. We say that $(L,R)$ form an adjoint couple of bifunctors (abbreviated as ACB) 
from $\cC'$ to $\cC$ with parameter in $\cX$ if we are given adjoint couples $(L(X,-),R(X,-),\unit_X,\counit_X)$ for every $X$ and if furthermore $\unit$ and $\counit$ are generalized transformations in the sense of \cite{Eilenberg66b}, \ie for any morphism $f:A \to B$ in $\cX$, the diagrams
$$\xymatrix{
L(A,R(B,C)) \ar[r]^-{L(f,id)} \ar[d]_{L(id,R(f,id))} \ar@{}[dr]|{\diag{gen}} & L(B,R(B,C)) \ar[d]^{\counit_B} \\
L(A,R(A,C)) \ar[r]^-{\counit_A} & C
}
\xymatrix{
C' \ar[r]^-{\unit_A} \ar[d]_{\unit_B} \ar@{}[dr]|{\diag{gen}} & R(A,L(A,C')) \ar[d]^{R(id,L(f,id))} \\
R(B,L(B,C')) \ar[r]^{R(f,id)} & R(A,L(B,C'))
}$$
commute. 
We sometimes use the notation $(L(*,-),R(*,-))$, where the $*$ is the entry in $\cX$ and
write
$$\xymatrix{\cC \ar@<1ex>[rr]^{L} \ar@{}[rr]|{\cX} & & \cC' \ar@<1ex>[ll]^{R}}$$
in diagrams.
\end{defi}

\begin{exam}
Let $\cC=\cC'=\cX$ be the category of modules over a commutative ring. 
The tensor product (with the variables switched) and the internal Hom form an ACB, with the usual unit and counit.
\end{exam}

\begin{lemm} \label{changeParam}
Let $F:\cX' \to \cX$ be a functor, and $(L,R)$ be an ACB with parameter in $\cX$.
Then $(L(F(*),-),R(F^o(*),-))$ is again an ACB in the obvious way. 
\end{lemm}

\begin{lemm} \label{AdjComposeBifFonc}
Let $(L,R)$ be an ACB from $\cC'$ to $\cC$ with parameter in $\cX$ and let $(F,G)$ be an adjoint couple from $\cD$ to $\cC'$, then $(L(*,F(-)),GR(*,-))$ is an ACB with unit and counit defined in the obvious and natural way.
\end{lemm}
\begin{proof}
Left to the reader.
\end{proof}

We now give a version of Lemma \ref{adjab} for ACBs. 

\begin{lemm}\label{adjabbif}
For $i=1$ or $2$, let $(J_i,K_i,\unit_i,\counit_i)$ be an ACB from $\cC_i'$ to $\cC_i$ with parameter in $\cX$, let $H:\cC_1 \to \cC_2$ and $H':\cC'_1 \to \cC'_2$ be functors. Let $a: J_2 (*,H'(-)) \trafo H J_1(*,-)$ (resp. $b: H' K_1(*,-) \trafo K_2(*, H(-))$) be a morphism of bifunctors. Then there exists a unique morphism of bifunctors $b: H' K_1(*,-) \trafo K_2(*, H(-))$ (resp. $a: J_2 (*,H'(-)) \trafo H J_1(*,-)$) such that the diagrams
$$
\xymatrix{
J_2 (X,H' K_1(X,-))  \ar[d]_{a K_1} \ar[r]^{J_2 b} \ar@{}[dr]|{\diag{H}} & J_2(X,K_2(X, H(-))) \ar[d]^{\eta_2 H} \\
H J_1(X, K_1(X,-)) \ar[r]_{H \counit_1} & H 
}$$
and
$$ 
\xymatrix{
H'  \ar[d]_{\unit_2 H'} \ar[r]^{H'\unit_1} \ar@{}[dr]|{\diag{H'}} & H' K_1(X, J_1(X,-)) \ar[d]^{b J_1}  \\
K_2(X, J_2(X,H'(-))) \ar[r]_{K_2 a} & K_2(X, H J_1(X,-))
}
$$ 
are commutative for every morphism $X \in \cX$. 
\end{lemm}
\begin{proof}
We only consider the case when $a$ is given. For every $X \in \cX$, we apply Lemma \ref{adjab} to obtain $b: H'K_1(X,-) \to K_2(X,H(-))$. It remains to prove that this is functorial in $X$. This is an easy exercise on commutative diagrams which we leave to the reader (it involves that both $\unit_1$ and $\counit_2$ are generalized transformations as in Definition \ref{defiAdjBif}).
\end{proof}

\begin{lemm} \label{isomorphACRB}
Let $(L,R,\unit,\counit)$ and $(L',R',\unit',\counit')$ be ACBs, and let $l:L \trafo L'$ (resp. $r:R\trafo R'$) be an isomorphism of bifunctors. Then, there exists a unique isomorphism of bifunctors $r:R \trafo R'$ (resp. $l:L \trafo L'$) such that $\unit'_X = r_X R(X,l) \unit_X$ and $\counit'_X =\counit_X L(X,r^{-1}_{X}) l^{-1}_{X}$ for every $X \in \cX$. In other words, a right (resp. left) bifunctor adjoint is unique up to unique isomorphism.
\end{lemm}
\begin{proof}
We only give the proof in case $l$ is given, the case with $r$ given is similar.
Apply Lemma \ref{adjabbif} with $H$ and $H'$ being identity functors. Starting with $l^{-1}$, we get $r$, and starting with $l$ we get a morphism $r':R'\to R$.
The morphisms $r$ and $r'$ are inverse to each other
as in the usual proof of the unicity of adjunction (which applies to every parameter $X \in \cX$).
\end{proof}

\subsection{Suspended and triangulated categories} \label{SuspTriangCat}

We recall here what we need about triangulated categories. The reason why we make a distinction between suspended categories and triangulated ones is because all the commutative diagrams that we are interested in are just related to the suspension, and not to the exactness of the functors involved. So when we need to prove the commutativity of those diagrams, we forget about the exactness of our functors, and just think of them as suspended functors, in the sense described below.

\begin{defi}
A suspended category is an additive category $\cC$ together with an adjoint couple $(T, T^{-1})$ from $\cC$ to $\cC$ which is an equivalence of category (the unit and counit are isomorphisms).
\end{defi}

\begin{rema} \label{TT-1Id}
We assume furthermore in all what follows that $TT^{-1}$ and $T^{-1}T$ are the identity of $\cC$ and that the unit and counit are also the identity. This assumption is not true in some suspended (triangulated) categories arising in stable homotopy theory.
Nevertheless, it simplifies the exposition which is already sufficiently
technical. When working in an example where this assumption
does not hold, it is of course possible to make the modifications to get this
even more general case. 
\end{rema}

Between suspended categories $(\cC,T_\cC)$ and $(\cD,T_\cD)$, we use suspended functors:

\begin{defi}
A suspended functor $(F,f)$ from $\cC$ to $\cD$ is a functor $F$ together with an isomorphism of functors $f: F T_\cC \trafo T_\cD F$. We sometimes forget about $f$ in the notation.
\end{defi}

Without the assumption in Remark \ref{TT-1Id}, we would need another isomorphism $f':F T^{-1} \trafo T^{-1} F$ and compatibility diagrams analogous to the ones in Lemma \ref{adjab}. Then, we would have to carry those compatibilities in our constructions. 
Again, this would not be a problem, just making things even more tedious.

Suspended functors can be composed in an obvious way, and $(T,id_{T^2})$ and $(T^{-1},id_{Id})$ are suspended endofunctors of $\cC$ that we call $T$ and $T^{-1}$ for short.

\begin{defi} \label{defiHigher}
To a suspended functor $F$, one can associate ``shifted" ones, composing $F$ by $T$ or $T^{-1}$ several times on either sides. The isomorphisms $T^iFT^j \simeq T^kFT^l$ with $i+j=k+l$ constructed using $f$, $f^{-1}$, $T^{-1}T=Id$ and $TT^{-1}=Id$ all coincide, so whenever we use one, we label it $``f"$ without further mention.
\end{defi}

\begin{defi}
The opposite suspended category $\cC^o$ of a suspended category $\cC$ is given the suspension $(T_\cC^{-1})^o$.
\end{defi}

With this convention, we can deal with contravariant suspended functors in two different ways (depending where we put the "op"), and this yields essentially the same thing, using the definition of shifted suspended functors.

\begin{defi} \label{defiMorSusp}
A morphism of suspended functors $h:(F,f) \to (G,g)$ is a morphism of functors $h: F \trafo G$ such that the diagram
$$\xymatrix{ 
FT \ar[d]_{hT} \ar[r]^{f} \ar@{}[dr]|{\diag{sus}}& TF \ar[d]^{Th} \\
GT \ar[r]_{g} & TG
}
$$
is commutative.
\end{defi}

\begin{lemm}
The composition of two morphisms of suspended functors yields a morphism of suspended functors.
\end{lemm}
\begin{proof}
Straightforward.
\end{proof} 

A triangulated category is a suspended category with the choice of some exact triangles, satisfying some axioms. This can be found in text books as \cite{Weibel94}(see also
the nice introduction in \cite[Section 1]{Balmer00}). We include the enriched octahedron axiom in the list of required axioms as it is suitable to deal with Witt groups, as explained in \loccit 

\begin{defi} \label{deltaexact} (see for example \cite[§ 1.1]{Gille03})
Let $(F,f):\cC \to \cD$ be a covariant (resp. contravariant) suspended functor. 
We say that $(F,f)$ is $\delta$-exact ($\delta=\pm 1$) if for any exact triangle
$$A \arr{u} B \arr{v} C \arr{w} TA$$
the triangle
$$FA \arr{Fu} FB \arr{Fv} FC \arr{\delta f_A \circ Fw} TFA$$
respectively
$$FC \arr{Fv} FB \arr{Fu} FA \arr{\delta f_C \circ FT^{-1}w} TFC$$
is exact.
\end{defi}

\begin{rema} \label{signTF}
With this definition, $T$ and $T^{-1}$ are $(-1)$-exact functors, because of the second axiom of triangulated categories, and the composition of exact functors multiplies their signs. Thus, if $F$ is $\delta$-exact, then $T^iFT^j$ is $(-1)^{i+j}\delta$-exact.
\end{rema}

To define morphisms between exact functors $F$ and $G$, the signs $\delta_F$ and $\delta_G$ of the functors have to be taken into account, so that the morphism of functors induces a morphism between the triangles obtained by applying $F$ or $G$ to a triangle and making the sign modifications.
\begin{defi} \label{morphTriangFunc}
We say that $h: F \trafo G$ is a morphism of exact functors if the diagram $\diag{sus}$ in Definition \ref{defiMorSusp} is $\delta_F \delta_G$ commutative.
\end{defi}
On the other hand, we have the following lemma.
\begin{lemm} \label{isoDeltaExact}
Let $h: F \trafo G$ be an isomorphism of suspended functors such that $\diag{sus}$ is $\nu$-commutative. Assume $F$ is $\delta$-exact. Then $G$ is $\delta \nu$-exact.
\end{lemm}
\begin{proof}
For any triangle
$$A \arr{u} B \arr{v} C \arr{w} TA$$
the triangle
$$GA \arr{Gu} GB \arr{Gv} FC \arr{\delta g_A \circ Gw} TGA$$
is easily shown to be isomorphic to
$$FA \arr{Fu} FB \arr{Fv} FC \arr{\nu \delta f_A \circ Fw} TFA$$
\end{proof}

We also need to deal with bifunctors from two suspended categories to another one. These are just suspended functors in each variable, with a compatibility condition. Examples are 
the internal Hom or the tensor product in triangulated categories.

\begin{defi} \label{defiSuspBif}
Let $\cC_1$, $\cC_2$ and $\cD$ be suspended categories. A suspended bifunctor from $\cC_1 \times \cC_2$  to $\cD$ is a triple $(B,b_1,b_2)$ where $B: \cC_1 \times \cC_2 \to \cD$ is a functor and two morphisms of functors $b_1:B(T(-),*) \to TB(-,*)$ and $b_2:B(-,T(*)) \to TB(-,*)$, such that the diagram
$$\xymatrix{
B(TA,TC) \ar[d]_{b_{1,A,TC}} \ar[r]^{b_{2,TA,C}} \ar@{}[rd]|{-1} & TB(TA,C) \ar[d]^{b_{1,A,C}} \\
TB(A,TC) \ar[r]_{b_{2,A,C}} & T^2 B(A,C)
}$$
anti-commutes for every $A$ and $C$.
\end{defi}

\begin{rema} \label{remaHigherBif}
As in Definition \ref{defiHigher}, we have shifted versions of $b_1$ (or $b_2$) which we 
will sometimes label $``b_1"$ (or $``b_2"$) below. It is important to describe precisely in which order morphisms are applied when combining $b_1$ and $b_2$.
\end{rema}

\begin{defi} \label{defiMorSuspBif}
A morphism of suspended bifunctors from a suspended bifunctor $(B,b_1,b_2)$ to a suspended bifunctor $(B',b'_1,b'_2)$ is a morphism of functors $f:B \to B'$ such that the two diagrams
$$\xymatrix{
B(TA,C) \ar[d]_{f_{TA,C}} \ar[r]^{b_{1,A,C}} & TB(A,C) \ar[d]^{Tf_{A,C}} & B(A,TC) \ar[l]_{b_{2,A,C}} \ar[d]^{f_{A,TC}} \\
B'(TA,C) \ar[r]^{b'_{1,A,C}} & TB'(A,C) & B'(A,TC) \ar[l]_{b'_{2,A,C}} \\
}$$
are commutative for every $A$ and $C$.
\end{defi}

By composing with a usual suspended functor to $\cC_1$ or $\cC_2$ or from $\cD$, we get other suspended bifunctors (the verification is easy). But, if we do that several times, using different functors, the order in which the suspended functors to $\cC_1$ or $\cC_2$ are used {\it does} matter. For example, as with usual suspended functors, it is possible to define shifted versions by composing with the suspensions in each category as mentioned in Remark \ref{remaHigherBif}. This can be useful. Unfortunately, according to the order in which we do this (if we mix functors to $\cC_1$ and to $\cC_2$), we don't get the same isomorphism of functors, even though we get the same functors in the pair. One has to be careful about that.

\subsection{Suspended adjunctions} \label{SuspAdj}

As with usual functors, there is a notion of adjunction well suited for suspended functors.

\begin{defi}
A suspended adjoint couple $(L,R)$ is a adjoint couple in the usual sense in which $L$ and $R$ are suspended functors and the unit and counit are morphisms of suspended functors.
\end{defi}

\begin{defi} \label{shiftedAdjoint}
When $(L,R)$ is an adjoint couple of suspended functors, using Lemma \ref{composeadj} we obtain shifted versions $(T^iLT^j,T^{-j}RT^{-i})$. Using Definition \ref{defiHigher}, we obtain isomorphisms exchanging the $T$'s. Applying Lemma \ref{isoadjoints} to them, we get an adjoint couple of suspended bifunctors $(T^iLT^j,T^{-i}RT^{-j})$.
\end{defi}

The following proposition seems to be well-known.

\begin{prop} \label{suspAdjExists}
Let $(L,R)$ be an adjoint couple from $\cC$ to $\cD$ (of usual functors) and let $(L,l)$ be a suspended functor. Then
\begin{enumerate}
\item there is a unique isomorphism of functors $r:RT \to TR$ that turns $(R,r)$ into a suspended functor and $(L,R)$ into a suspended adjoint couple. 
\item if furthermore $\cC$ and $\cD$ are triangulated and $(L,l)$ is $\delta$-exact, then $(R,r)$ is also $\delta$-exact (with the {\it same} $\delta$).
\end{enumerate}
\end{prop}
\begin{proof}
Point 1 is a direct corollary of Point 3 of Theorem \ref{adjsquare}, by taking $L=L'$, $R=R'$, $F_1=T_\cC$, $G_1=T^{-1}_\cC$, $F_2=T_\cD$, $G_2=T^{-1}_\cD$, $g_L=(g'_L)^{-1}=T^{-1}lT^{-1}$. This gives $f_R=r$. The commutative diagrams $\diag{F_1}$ and $\diag{F_2}$ 
exactly tell us that the unit and counit are suspended morphisms of functors with this choice of $r$.
To prove Point 2, we have to show that the pair $(R,r)$ is exact. Let 
$$A \arr{u} B \arr{v} C \arr{w} TA$$
be an exact triangle. We want to prove that the triangle
$$RA \arr{u} RB \arr{v} RC \arr{r_A \circ Rw} TRA$$
is exact. We first complete $RA \arr{u} RB$ as an exact triangle
$$RA \arr{u} RB \arr{v'} C' \arr{w'} TRA$$
and we prove that this triangle is in fact isomorphic to the previous one. 
To do so, 
one completes the incomplete morphism of triangles
$$\xymatrix@R=3ex{
LRA \ar[r]^{LRu} \ar[d] & LRB \ar[r]^{LRv} \ar[d] & LC' \ar[rr]^{f_{RA}\circ LRw} 
\ar@{.>}[d]^(.4){h} & & TLRA \ar[d] \\
A \ar[r]_{u} & B \ar[r]_{v} & C \ar[rr]_{w} & & TA \\
}.$$
Looking at the adjoint diagram, we see that $ad(h):C' \to RC$ is an isomorphism by the five lemma for triangulated categories.
\end{proof}

\begin{theo} \label{adjsquareT}
Lemma \ref{adjab}, Lemma \ref{cube} and Theorem \ref{adjsquare} hold when we replace every functor by a suspended functor, every adjoint couple by a suspended adjoint couple and every morphism of functor by a morphism of suspended functors.
\end{theo}
\begin{proof}
The same proofs hold, since they only rely on operations and properties of functors and morphism of functors, such as composition or commutative diagrams, that exist and behave the same way in the suspended case.
\end{proof}

We now adapt the notion of an adjoint couple of bifunctors
to suspended categories.

\begin{defi} \label{defiSACRB}

Let $(L,R)$ be an ACB from $\cC'$ to $\cC$ with parameter in
$\cX$, where $\cC, \cC'$ and $\cX$ are suspended categories.
Assume moreover that  $(L,l_1,l_2)$ and $(R,r_1,r_2)$ are suspended bifunctors. 
We say that $(L,R)$ is a suspended adjoint couple of bifunctors if 
\begin{enumerate}
\item $((L(X,-),l_2),(R(X,-),r_2))$ is a suspended adjoint couple 
of functors for every parameter $X$, 
\item the following diagrams commute: 
$$\xymatrix{
TC \ar[r]^-{\unit_{X,TC}} \ar[d]_{T\unit_{TX,C}} & R(X,L(X,TC)) \ar[r]^-{``r_1"} & TR(TX,L(X,TC)) \ar[d]^{TR(TX,l_2)} \\
TR(TX,L(TX,C)) \ar[rr]^-{TR(X,l_1)} & & TR(TX,TL(X,C)) 
}$$
$$\xymatrix{
TL(X,R(TX,C)) \ar[rr]^-{l_1^{-1}} \ar[d]_{l_2^{-1}} & & L(TX,R(TX,C)) \ar[d]^{\counit_{TX,C}} \\
L(X,TR(TX,C)) \ar[r]^-{``r_1"^{-1}} & L(X,R(X,C)) \ar[r]^{\counit_{X,C}} & C
}$$
\end{enumerate}
\end{defi}

\begin{rema}\label{explainSACRB}
Note that (2) ensures the compatibility of the suspension functor on the parameter 
with the other ones.
\end{rema}

\begin{lemm} \label{changeParamSus}
Let $(F,f):\cX' \to \cX$ be a suspended functor, and $(L,R)$ a suspended ACB with parameter in $\cX$.
Then $(L(F(*),-),R(F^o(*),-))$ is again a suspended ACB in the obvious way. 
\end{lemm}
\begin{proof}
Left to the reader.
\end{proof}

The following proposition is an analogue of Proposition \ref{suspAdjExists} for suspended bifunctors.
\begin{prop} \label{suspACRBExists}
Let $(L,R)$ be an ACB such that $(L,l_1,l_2)$ (resp. $(R,r_1,r_2)$) is a suspended bifunctor. Then 
\begin{enumerate}
\item there exist unique $r_1,r_2$ (resp. $l_1,l_2$) such that $(R,r_1,r_2)$ (resp. $(L,l_1,l_2)$) is a suspended bifunctor and $(L,R)$ is a suspended ACB,
\item if $(L(X,-),l_2)$ is $\delta$-exact for some object $X$, then so is $(R(X,-),r_2)$.
\end{enumerate}
\end{prop}
\begin{proof}
The existence and uniqueness of $r_2$ follows (for every parameter $X$) from Point 1 of Proposition \ref{suspAdjExists}. That $r_1$ is also natural in the first
variable follows from a large diagram of commutative squares 
of type $\diag{mf}$ and $\diag{gen}$. This yields in particular morphisms of suspended functors $\unit_X$ and $\counit_X$ for every $X$. For the existence and uniqueness of $r_1$, we use Lemma \ref{adjabbif} applied to $H'=T$, $H=Id$, $J_1=L(*,-)$, $K_1=R(*,-)$, $J_2=L(T^{-1}(*),-)$, $K_2=R(T^{-1}(*),-)$ and $a=``l_1"``l_2"$ (the order is important). We obtain two diagrams which are easily seen to be equivalent to the ones required for $(L,R)$ to be a suspended ACB (Point 2 of Definition \ref{defiSACRB}). The anticommutativity required by Definition \ref{defiSuspBif} for $r_2$ and $r_1$ may be proved using the anticommutativity of $l_1$ and $l_2$ and the fact that $\unit_X$ and $\counit_{T^{-1}X}$ are morphisms of suspended functors. Point 2 is proved in the same way as Point 2 of Proposition \ref{suspAdjExists}.
\end{proof}

We need a version of Lemma \ref{adjabbif} for suspended bifunctors.

\begin{lemm} \label{adjabbifT}
Lemma \ref{adjabbif} holds when all the functors, bifunctors and adjunctions become suspended ones.
\end{lemm}
\begin{proof}
The proof of Lemma \ref{adjabbif} works since it only involves compositions and commutative diagrams that exist in the suspended case. \\
\end{proof}

Finally, there is also a version of Theorem \ref{adjsquare} for (suspended) bifunctors, which is our main tool for the applications. For this reason, we state it 
in full detail.

\begin{theo} \label{theoGenAdj}
Let $L$, $R$, $L'$, $R'$, $F_1$, $G_1$, $F_2$, $G_2$ be (suspended) functors whose sources and targets are specified by the diagram
(recall the notation of Definition \ref{defiAdjBif}), 
$$\xymatrix{
\cC_1 \ar@<1.5ex>[d]^{G_1} \ar@{}[d]|{\cX} \ar@<0.5ex>[rr]^{L} & & \cC_2 \ar@<1.5ex>[d]^{G_2} \ar@{}[d]|{\cX} \ar@<0.5ex>[ll]^{R} \\
\cC_1' \ar@<1.5ex>[u]^{F_1} \ar@<0.5ex>[rr]^{L'} & & \cC_2' \ar@<1.5ex>[u]^{F_2} \ar@<0.5ex>[ll]^{R'} \\
}
$$
and let $f_L$, $f'_L$, $g_L$, $g'_L$, $f_R$, $f'_R$, $g_R$ and $g'_R$ be morphisms of bifunctors (resp. suspended bifunctors) whose sources and targets will be as follows:
$$
\begin{array}{lr}
\xymatrix{
LF_1(*,-) \ar@<0.5ex>[r]^{f_L} & F_2 (*,L'(-)) \ar@<0.5ex>[l]^{f'_L} 
}
&
\xymatrix{
L'G_1(*,-) \ar@<0.5ex>[r]^{g'_L} & G_2 (*,L(-)) \ar@<0.5ex>[l]^{g_L}
}
\\
\xymatrix{
F_1(*,R'(-)) \ar@<0.5ex>[r]^{f'_R} & RF_2(*,-) \ar@<0.5ex>[l]^{f_R}
}
&
\xymatrix{
G_1(*,R(-)) \ar@<0.5ex>[r]^{g_R} & R' G_2(*,-) \ar@<0.5ex>[l]^{g'_R}
} \\
\end{array}
$$
Let us consider the following diagrams, in which the maps and their directions will be the only obvious ones in all the cases discussed below.
{\small
$$\xymatrix{
F_2(X, L'G_1(X,-)) \ar@{-}[d] \ar@{-}[r] \ar@{}[dr]|{\diag{L}} & L F_1(X,G_1(X,-)) \\
F_2(X, G_2(X, L(-))  & L \ar[l] \ar[u] 
}
\xymatrix{
G_1(C,-) \ar[d] \ar[r] \ar@{}[dr]|{\diag{G_1}} & G_1(C, R L(-)) \ar@{-}[d] \\
R'L'G_1(C,-) \ar@{-}[r] & R'G_2(C, L(-)) 
}
$$
$$\xymatrix{
F_1(X,R' G_2(X,-)) \ar@{-}[d] \ar@{-}[r] \ar@{}[dr]|{\diag{R}} & R F_2(X, G_2(X,-)) \\
F_1(X, G_1(X, R(-))) & R \ar[l] \ar[u] 
}
\xymatrix{
F_1(C,-) \ar[d] \ar[r] \ar@{}[dr]|{\diag{F_1}} & F_1(C, R' L'(-)) \ar@{-}[d] \\
RLF_1(C,-) \ar@{-}[r] & RF_2(C, L'(-))
}
$$
$$\xymatrix{
L' \ar@{}[dr]|{\diag{L'}} & G_2(X, F_2(X, L'(-))) \ar[l] \ar@{-}[d] \\
L' G_1(X, F_1(X,-)) \ar[u] \ar@{-}[r] & G_2(X, L F_1(X,-))
}
\xymatrix{
L'G_1(C, R(-)) \ar@{-}[d] \ar@{-}[r] \ar@{}[dr]|{\diag{G_2}} & G_2(C, L R(-)) \ar[d] \\
L'R' G_2(C,-) \ar[r] & G_2(C,-)
}
$$
$$\xymatrix{
R' \ar@{}[dr]|{\diag{R'}} & G_1(X, F_1(X, R'(-))) \ar[l] \ar@{-}[d] \\
R' G_2(X, F_2(X,-)) \ar[u] \ar@{-}[r] & G_1(X, R F_2(X,-))
}
\xymatrix{
LF_1(C, R'(-)) \ar@{-}[d] \ar@{-}[r] \ar@{}[dr]|{\diag{F_2}} & F_2(C, L' R'(-)) \ar[d] \\
LR F_2(C,-) \ar[r] & F_2(C,-)
}
$$
}
Then
\begin{itemize}
\item[1.] Let $(G_i,F_i)$, $i=1,2$ be ACBs (resp. suspended ACBs). Let $g_L$ (resp. $f_L$) be given, then there is a unique (suspended) $f_L$ (resp. $g_L$) such that $\diag{L}$ and $\diag{L'}$ are commutative for any $X$. Let $g_R$ (resp. $f_R$) be given, then there is a unique (suspended) $f_R$ (resp. $g_R$) such that $\diag{R}$ and $\diag{R'}$ are commutative for any $X$.
\item[2.] Let $(L,R)$ and $(L',R')$ be adjoint couples (resp. suspended adjoint couples). Let $f_L$ (resp. $f'_R$) be given, then there is a unique (suspended) $f'_R$ (resp. $f_L$) such that $\diag{F_1}$ and $\diag{F_2}$ are commutative. Let $g'_L$ (resp. $g_R$) be given, then there is a unique (suspended) $g_R$ (resp. $g'_L$) such that $\diag{G_1}$ and $\diag{G_2}$ are commutative.
\item[3.] Assuming $(G_i,F_i)$, $i=1,2$, are ACBs (resp. suspended ACBs), $(L,R)$ and
$(L',R')$ are (suspended) adjoint couples, and $g_L$ and $g'_L=g_L^{-1}$ are given (resp. $f_R$ and $f'_R=f_R^{-1}$). By 1 and 2, we obtain $f_L$ and $g_R$ (resp. $g_R$ and $f_L$).
We then may construct $f_R$ and $f'_R$ (resp. $g_L$ and $g'_L$) which are inverse to each other.  
\end{itemize}
\end{theo}
\begin{proof}
The proof is the same as for Theorem \ref{adjsquare}, but using Lemma \ref{adjabbif} (or \ref{adjabbifT}) instead of Lemma \ref{adjab} for Points 1 and 2.
\end{proof}
\begin{rema}
We didn't state the analogues of Points 1' and 2' of Theorem \ref{adjsquare} in this context because we don't need them.
\end{rema}

\section{Dualities} \label{Dualities}

We now introduce our main subject of interest: duality. As before, we state everything for the suspended (or triangulated) case and the usual case in a uniform way. 

\subsection{Categories with duality}

\begin{defi}\label{catwithduality}
A category with duality is a triple $(\cC,D,\bid)$ where $\cC$ is a category with an adjoint couple $(D,D^o,\bid,\bid^o)$ from $\cC$ to $\cC^o$. When $\bid$ is an isomorphism, $D$ is an equivalence of categories and we say that the duality is {\it strong}. A suspended (resp. triangulated) category with duality is defined in the same way, but $(D,d)$ is a suspended (resp. $\delta$-exact) functor 
on a suspended (resp. triangulated) category $\cC$, and the adjunction is suspended.
\end{defi}

\begin{rema}
Observe that the standard condition $(D^o \bid^o) \circ (\bid^o D)=id_{D^o}$
is satisfied by the definition of an adjunction.
The only difference between our definition and Balmer's definition \cite[Def. 2.2]{Balmer00} is that we don't require the isomorphism $d:DT \to T^{-1}D$ to be an equality in the suspended or triangulated case. Assuming that duality and suspension
strictly commute is as bad as assuming that the internal Hom 
and the suspension strictly
commute, or (by adjunction) that the tensor product and the suspension strictly commute.
This is definitely a too strong condition when checking strict commutativity
of diagrams in some derived category. Dropping all signs in this setting
when defining these isomorphisms just by saying ``take the canonical ones''
may even lead to contradictions as the results of the Appendix show.
When $d=id$, we say that the duality is {\it strict}.
\end{rema}

\begin{defi} \label{defiSuspTriCat}
Let $(\cC,D,\bid)$ be a triangulated category with duality for which $(D,d)$ is $\delta$-exact. By Definition \ref{shiftedAdjoint} we get a shifted adjoint couple $T(D,D^o)=(TD,TD^o,\bid',(\bid')^o)$. 
We define the suspension of $(\cC,D,\bid)$ as $T(\cC,D,\bid)=(\cC,TD,-\delta \bid')$. Note that $(TD,Td)$ is $(-\delta)$-exact because $T$ is $(-1)$-exact.
\end{defi}

\begin{rema}
This is the definition of \cite[Definition 2.8]{Balmer00} adapted to cover the non strict case, and the next one generalizes \cite[Definition 2.13]{Balmer00} to the non
strict case.
\end{rema}

\begin{defi}\label{Wgraded}
For any triangulated category with strong duality $(\cC,D,\bid)$,
we define its i-th Witt group $W^i$ by
$\W^i(\cC,D,\bid):=W(T^i(\cC,D,\bid))$
(extending \cite[2.4 and Definitions 2.12 and 2.13]{Balmer00}
in the obvious way). If $D$ and $\bid$ are understood,  
we sometimes also write $\W^i(\cC)$ or $\W^i(\cC,D))$ for short.
\end{defi}

\begin{rema}
In concrete terms, this means that the condition of
\loccit\ for an element in $W^1(\cC)$ represented 
by some $\phi$ to be symmetric is that
$(TdT^{-1}) \circ \phi = (D^oT\phi) \circ \bid$ whereas in the 
strict case the $(TdT^{-1})$ may be omitted.
\end{rema}  

\begin{rema}
To define a Witt group, the duality has to be strong (that is $\bid$ has to be an 
isomorphism), but this is not necessary to prove all sorts of commutative diagrams, 
so there is no reason to assume it here in general.
\end{rema}

\subsection{Duality preserving functors and morphisms}

\begin{defi} \label{DualPresFunc_defi}
A duality preserving functor from a (suspended, triangulated) category with duality $(\cC_1,D_1,\bid_1)$ to another one $(\cC_2,D_2,\bid_2)$ is a pair $\pair{F,f}$ where $F$ is a (suspended, $\delta$-exact) functor from $\cC_1$ to $\cC_2$, and $f:FD_1^o \trafo D_2^o F^o$ is a morphism (in $\cC_2$) of (suspended, $\delta$-exact) functors where $f$ and $f^o$ are mates as in Lemma \ref{adjab} (resp. Theorem \ref{adjsquareT}) when setting $J_1=D_1$, $K_1=D_1^o$, $J_2=D_2$, $K_2=D_2^o$, $H=F^o$, $H'=F$, $a=f^o$ and $b=f$. In other words, the diagram $\diag{H}$ (equivalent to $\diag{H'}$) of Lemma \ref{adjab} must commute:
$$\xymatrix{
F \ar[d]_{\bid_2 F} \ar[r]^{F \bid_1}\ar@{}[dr]|{\diag{P}} & F D_1^o D_1 \ar[d]^{f D_1}  \\
D_2^o D_2 F  \ar[r]_{D_2^o f^o} & D_2^o F^o D_1  
}$$
We sometimes denote the functor simply by $\pair{F}$ if $f$ is understood. When $f$ is an isomorphism, we say the functor is strong duality preserving.
\end{defi}

When the dualities are both strict and strong and the functor
is strongly duality preserving, then this coincides with the usual definition (see for example \cite[Definition 2.6]{Gille02} where Point 2 corresponds to the fact that $f$ is a morphism o suspended functors in our definition and is only used in the suspended case).
Duality preserving functors are composed in the obvious way by setting $\pair{F',f}\pair{F,f} = \pair{F'F,f'F^o \circ F' f}$.

\begin{defi} \label{MorphDualPresFunc_defi}
A morphism of (suspended, exact) duality preserving functors from $\pair{F,f}$ to $\pair{F',f'}$ (with same source and target) is a morphism of the underlying (suspended, exact) functors $\rho:F \trafo F'$ such that the diagram
$$\xymatrix{
F D_1^o \ar[d]_{\rho D_1^o} \ar[r]^{f} \ar@{}[dr]|{\diag{M}} & D_2^o F^o \ar[d]^{D_2^o \rho^o} \\
F' D_1^o \ar[r]_{f'} & D_2^o (F')^o  
}$$
commutes. We say that such a morphism is strong if the underlying morphism of functors is an isomorphism.
\end{defi}

Composing two morphisms between duality preserving functors obviously gives another one, and composition preserves being strong.

The proofs of following two propositions are
straightforward
(see also \cite[Theorem 2.7]{Gille02} for 
a proof of the first one in the strict case).

\begin{prop}\label{dualPresIsTransfer_prop}
A $1$-exact strong duality preserving functor $\pair{F,f}$ 
between triangulated categories with strong dualities induces a morphism on Witt groups by sending an element represented by a form $\psi: A \to D_1(A)$ to the class of the form $f_A \circ F(\psi)$.
\end{prop}

\begin{prop} \label{samemorphismWitt_prop}
A duality preserving isomorphism between $1$-exact strong duality preserving functors ensures that they induce the same morphism on Witt groups.
\end{prop}

\begin{defi}
\begin{enumerate}
\item Let $(\cC,D,\bid)$ be a triangulated category with duality and let $\cA$ be a full triangulated subcategory of $\cC$ preserved by $D$. Then, $(\cA,D|_{\cA},\bid|_{\cA})$ is trivially a triangulated category with duality, and we say that the duality of $\cC$ {\it restricts} to $\cA$.
\item Let $\pair{F,f}:(\cC_1,D_1,\bid_1) \to (\cC_2,D_2,\bid_2)$ be a 
duality preserving functor between triangulated categories
with dualities as in Definition \ref{DualPresFunc_defi}. 
Assume that there are full triangulated subcategories $\cA_i \subset \cC_i$, $i=1,2$ such that $D_i$ restricts to $\cA_i$ and such that $F|_{\cA_1}$ factors through $\cA_2$. Then we say that the duality preserving pair $\pair{F,f}$ restricts to the subcategories 
$\cA_1$ and $\cA_2$. 
\item Let $\rho$ be a morphism between two such restricting functors $\pair{F,f}$ and $\pair{F',f'}$, as in Definition \ref{MorphDualPresFunc_defi}, then the restriction of $\rho$ automatically defines a morphism of exact duality preserving functors between the restricted functors. 
\end{enumerate}
\end{defi}

%
\begin{lemm}\label{restricteddualmor}       
Let $\pair{F,f}:(\cC_1,D_1,\bid_1) \to (\cC_2,D_2,\bid_2)$ be a 
duality preserving functor between triangulated categories
with duality that restricts to the subcategories 
$\cA_1$ and $\cA_2$. 
\begin{enumerate}
\item Assume that the restricted dualities
on $\cA_1$ and $\cA_2$ are strong and that $f|_{F(\cA_1^o)}$
is an isomorphism. Then the restriction of the duality preserving
pair $\pair{F,f}$ to $\cA_1$ and $\cA_2$ is a 
a strongly duality preserving functor between
triangulated categories with strong duality
and therefore induces a morphism 
$$\W(\cA_1,D_1|_{\cA_1},\bid_1|_{\cA_1}) \to \W(\cA_2,D_2|_{\cA_2},\bid_2|_{\cA_2})$$
on Witt groups. 
\item If a morphism $\rho$ between two such duality preserving functors is strong when restricted to $\cA_1$, then they induce the same morphisms on Witt groups.
\end{enumerate}
\end{lemm}
\begin{proof} 
One has to check that certain diagrams in $\cA_2$ are commutative.
This follows as they are already commutative in $\cC_2$ by assumption. 
Now Propositions \ref{dualPresIsTransfer_prop} and \ref{samemorphismWitt_prop} prove the claims.
\end{proof}

%
\section{Consequences of the closed monoidal structure} \label{consClosedMon}

We now recall a few notions on tensor products and internal Hom functors (denoted by $\HHom{-}{*}$) and prove very basic facts related to the suspension. 
A category satisfying the axioms of this section deserves 
to be called a 
`` (suspended or triangulated)
symmetric monoidal closed category''. Then we prove that the dualities defined using the internal Hom on such a category are naturally equipped with the necessary data to define (triangulated) categories with dualities. 

\subsection{Tensor product and internal Hom} \label{TensProHom}

Let $(\cC,\TTens)$ be a symmetric monoidal category (see \cite[Chapter VII]{MacLane98})
with an internal Hom $\HHom{-}{-}$ adjoint to the tensor product. More
precisely, we assume that $(-\TTens *, \HHom{*}{-})$ is an ACB. We denote by
$\cc=\cc^{-1}$ the symmetry isomorphism and call this datum
a ``symmetric monoidal closed category''.

When talking about a ``suspended symmetric monoidal closed
category'', we assume that we have a suspended bifunctor 
$(-\TTens *,\tpp_1,\tpp_2)$ (see Definition \ref{defiSuspBif}) such that the diagram
$$\xymatrix@R=3ex{
(TA \TTens B) \TTens C \ar[r] \ar[d]_{\tpp_1 \TTens id} \ar@{}[ddr]|{\diag{assoc}} & TA \TTens (B \TTens C) \ar[dd]^{\tpp_1} \\
T(A\TTens B) \TTens C \ar[d]_{\tpp_1} & \\
T((A \TTens B) \TTens C) \ar[r] & T(A \TTens (B \TTens C)) \\
}$$
commutes, as well as the two similar ones in which the suspension starts on one of the other variables. We also assume that the diagram
$$\xymatrix{
T(-)\TTens * \ar[d]_{\tpp_1} \ar[r]^{\cc} \ar@{}[dr]|{\diag{\cc}} & {* \TTens T(-)} \ar[d]^{\tpp_2} \\
T(-\TTens *) \ar[r]_{T\cc} & T(* \TTens -)
}$$
commutes.
By Proposition \ref{suspACRBExists}, we get morphisms 
$$\thh_1: \HHom{T^{-1}(*)}{-} \to T \HHom{*}{-} \hspace{10ex} \thh_2: \HHom{*}{T(-)} \to T\HHom{*}{-}$$
that make $(\HHom{*}{-},\thh_1,\thh_2)$ a suspended bifunctor and $(-\TTens *, \HHom{*}{-})$ a suspended ACB (Definition \ref{defiSACRB}). Using $\cc$, we obtain a new suspended ACB $(* \TTens -, \HHom{*}{-})$ from the previous one.

If $\cC$ is triangulated, we furthermore assume that $\TTens$ is exact in both variables
(by symmetry it suffices to check this for one of them). 
By Proposition \ref{suspACRBExists}, $\HHom{*}{-}$ is suspended in both variables and automatically exact in
the second variable. We assume furthermore that it is exact in the first
variable,
and say that we have a ``triangulated closed symmetric monoidal category''.
 
The morphisms 
$$\ev^l_{A,K}: \HHom{A}{K}\TTens A \to K \hspace{10ex} \coev^l_{A,K}: K \to \HHom{A}{K \TTens A}$$
respectively
$$\ev^r_{A,K}: A \TTens \HHom{A}{K} \to K \hspace{10ex} \coev^r_{A,K}: K \to \HHom{A}{A \TTens K}$$
induced by the counit and the unit of the (suspended) ACB $(- \TTens *, \HHom{*}{-})$ (resp. $(* \TTens -, \HHom{*}{-})$) are called the left (resp. right) evaluation and coevaluation.

\begin{lemm} \label{diagEv}
The following diagrams are commutative.
$$\xymatrix@R=3ex@C=3ex{
(T\HHom{A}{K}) \TTens A \ar[r]^{\tpp_1} \ar@{}[dr]|{\diagram \label{evl1}} & T(\HHom{A}{K} \TTens A) \ar[d]^{T\ev^l_{A,K}} \\
\HHom{A}{TK} \TTens A \ar[u]^{\thh_2 \TTens id} \ar[r]_-{\ev^l_{A,TK}} & TK
}
\hspace{3ex}
\xymatrix@R=3ex@C=3ex{
A \TTens (T\HHom{A}{K}) \ar[r]^{\tpp_2} \ar@{}[dr]|{\diagram \label{evr1}} & T(A \TTens \HHom{A}{K}) \ar[d]^{T\ev^r_{A,K}} \\
A \TTens \HHom{A}{TK} \ar[u]^{id \TTens \thh_2} \ar[r]_-{\ev^r_{A,TK}} & TK
}$$
$$\xymatrix@R=3ex@C=3ex{
TK \ar[r]^-{\coev^l_{A,TK}} \ar[d]_{T\coev^l_{A,K}} \ar@{}[dr]|{\diagram \label{coevl1}} & \HHom{A}{TK \TTens A} \ar[d]^{\tpp_1} \\
T \HHom{A}{K \TTens A} \ar[r]_-{\thh_2^{-1}} & \HHom{A}{T(A \TTens K)}
}
\hspace{3ex}
\xymatrix@R=3ex@C=3ex{
TK \ar[r]^-{\coev^r_{A,TK}} \ar[d]_{T\coev^r_{A,K}} \ar@{}[dr]|{\diagram \label{coevr1}} & \HHom{A}{A \TTens TK} \ar[d]^{\tpp_2} \\
T \HHom{A}{A \TTens K} \ar[r]_-{\thh_2^{-1}} & \HHom{A}{T(K \TTens A)}
}$$
$$\xymatrix@R=3ex@C=3ex{
(T^{-1}\HHom{A}{K}) \TTens TA \ar[d]_{T^{-1}\thh_{1,TA,K}\TTens id} \ar[r]^{\tpp_2} \ar@{}[drr]|{\diagram \label{evl2}} & T(T^{-1}\HHom{A}{K} \TTens A) \ar[r]^-{\tpp_1^{-1}} & \HHom{A}{K} \TTens A \ar[d]^{\ev^l_{A,K}} \\
\HHom{TA}{K} \TTens TA \ar[rr]_{\ev^l_{TA,K}} & & K
}$$
$$\xymatrix@R=3ex@C=3ex{
(TA  \TTens T^{-1}\HHom{A}{K}) \ar[d]_{id \TTens T^{-1}\thh_{1,TA,K}} \ar[r]^{\tpp_1} \ar@{}[drr]|{\diagram \label{evr2}} & T(A \TTens T^{-1}\HHom{A}{K}) \ar[r]^-{\tpp_2^{-1}} & A \TTens \HHom{A}{K} \ar[d]^{\ev^r_{A,K}} \\
TA \TTens \HHom{TA}{K} \ar[rr]_{\ev^r_{TA,K}} & & K
}$$
$$\xymatrix@R=3ex@C=3ex{
K \ar[rr]^-{\coev^l_{TA,K}} \ar[d]_{\coev^l_{A,K}} \ar@{}[drr]|{\diagram \label{coevl2}} & & \HHom{TA}{K \TTens TA} \ar[d]^{T^{-1}\thh_{1,TA,K\TTens TA}^{-1}} \\ 
\HHom{A}{K \TTens A} \ar[r]^-*!/^1ex/{\labelstyle T^{-1}\thh_2^{-1}} & T^{-1}\HHom{A}{T(K \TTens A)} \ar[r]^-{\tpp_2^{-1}} & T^{-1} \HHom{A}{K \TTens TA}
}$$
$$\xymatrix@R=3ex@C=3ex{
K \ar[rr]^-{\coev^r_{TA,K}} \ar[d]_{\coev^r_{A,K}} \ar@{}[drr]|{\diagram \label{coevr2}} & & \HHom{TA}{TA \TTens K} \ar[d]^{T^{-1}\thh_{1,TA,TA \TTens K}^{-1}} \\ 
\HHom{A}{A \TTens K} \ar[r]^-*!/^1ex/{\labelstyle T^{-1}\thh_2^{-1}} & T^{-1}\HHom{A}{T(A \TTens K)} \ar[r]^-{\tpp_1^{-1}} & T^{-1} \HHom{A}{TA \TTens K}
}$$
\end{lemm}
\begin{proof}
This is a straightforward consequence of Point 1 of Definition \ref{defiSACRB} for the first four diagrams and of Point 2 of Definition \ref{defiSACRB} for the other four. 
\end{proof}

\subsection{Bidual isomorphism} \label{BidualIsomorphism}

We still assume that $(\cC,\TTens)$ is a monoidal category with an internal 
Hom as in the previous section.
We now show that the functor $D_K=\HHom{-}{K}$ naturally defines a duality on the category $\cC$, and that in the suspended case that the dualities $D_{TK}$ and $TD_K$ are naturally isomorphic. 

To form the adjoint couple $(D_K,D_K^o,\bid_K,\bid_K^o)$, we define the bidual morphism of functors $\bid_K: Id \to D_K^o D_K$ as the image of the right evaluation by the adjunction $(-\TTens*,\HHom{*}{-})$ isomorphism
$$\xymatrix{
\Hom(A \TTens \HHom{A}{K},K) \ar[r]^{\sim} & \Hom(A,\HHom{\HHom{A}{K}}{K})
}.$$

It is functorial in $A$ and defines a morphism of functors from $Id$ to $D_K^o D_K$. Note that its definition uses the adjunction $(-\TTens*,\HHom{*}{-})$ and the right evaluation, which is not the counit of this adjunction but of the one obtained from it by using 
$\cc$; so the fact that the monoidal category is symmetric is essential, here. One 
cannot proceed with only one of these adjunctions. In the suspended case, $D_K$ becomes
a suspended functor via $T^{-1} \thh_{1,-,K}^{-1} T: D_K T \to T^{-1} D_K$.

\begin{prop} \label{susBifBid_prop}
In the suspended (or triangulated) case, $\bid_K$ is a morphism of suspended (or exact) 
functors.
\end{prop}
\begin{proof}
First note that in the exact case, whatever the sign of $D_K$ is, $D_K^o D_K$ is $1$-exact, so there is no sign involved in the diagram
$$\xymatrix@R=3ex@C=12ex{
TA \ar[d]_{T\bid_K} \ar[r]^{\bid_K T} & \HHom{\HHom{TA}{K}}{K} \ar[d]^{T^{-1}\thh_{1,TA,K}} \\
T\HHom{\HHom{A}{K}}{K} \ar[r]_{\thh_{1,\HHom{A}{K},K}^{-1}}
& \HHom{T^{-1}\HHom{A}{K}}{K}
}$$
that we have to check (see Definitions \ref{morphTriangFunc} and \ref{defiMorSusp}). It is obtained by (suspended) adjunction from Diagram \diag{\ref{evr2}} in Lemma \ref{diagEv}
\end{proof}

\begin{defi} \label{defiDualizing}
We say that $K$ is a dualizing object when $\bid_K$ is an isomorphism of (suspended) functors. 
\end{defi}

\begin{prop} \label{HomHomSuspAdj_prop}
The functors $\HHom{-}{*}:\cC \times \cC^o \to \cC^o$ and $\HHom{-}{*}^o: \cC^o \times \cC \to \cC$ form a (suspended) ACB with unit $\bid$ and counit $\bid^o$. 
\end{prop}
\begin{proof}
We first have to prove
that $(D_K,D_K^o,\bid_K,\bid_K^o)$ is an adjoint couple in the usual sense. We already know that $\bid_K$ is a suspended morphism. Consider the following diagram, in which all vertical maps are isomorphisms. We use the notation $f^\sharp: \Hom(F',G) \to \Hom(F,G)$ and $f_\sharp: \Hom(G,F) \to \Hom(G,F')$ for the maps induced by $f: F \to F'$.
The unlabeled morphisms are just adjunction bijections,
and we set $\bid_{A,K}:=(\bid_K)_A$.
$$\xymatrix@R=3ex@C=8ex{
\Hom(\HHom{A}{K},\HHom{A}{K}) & & \Hom(\HHom{A}{K}, \HHom{\HHom{\HHom{A}{K}}{K}}{K}) \ar[ll]_-{(\HHom{\bid_{A,K}}{Id_K})_\sharp} \\
\Hom(\HHom{A}{K}\TTens A,K) \ar[u] & & \Hom(\HHom{A}{K} \TTens \HHom{\HHom{A}{K}}{K},K) \ar[ll]^-{(Id_{\HHom{A}{K}}\TTens \bid_{A,K})^\sharp} \ar[u] \\
\Hom(A\TTens \HHom{A}{K},K) \ar[u]^{\cc^\sharp_{A,\HHom{A}{K}}} \ar[d] & & \Hom(\HHom{\HHom{A}{K}}{K} \TTens \HHom{A}{K},K) \ar[ll]^-{(\bid_{A,K} \TTens Id_{\HHom{A}{K}})^\sharp} \ar[u]_{\cc^\sharp_{\HHom{\HHom{A}{K}}{K},\HHom{A}{K}}} \ar[d] \\
\Hom(A, \HHom{\HHom{A}{K}}{K}) & & \Hom(\HHom{\HHom{A}{K}}{K}, \HHom{\HHom{A}{K}}{K}) \ar[ll]^-{(\bid_{A,K})^\sharp} \\
}$$
The diagram commutes by functoriality of $\cc$ and the adjunction bijections.
Now $Id_{\HHom{\HHom{A}{K}}{K}}$ in the lower right set is sent to $\bid_{\HHom{A}{K},K}$ in the upper right set, which is in turn sent to $\HHom{\bid_{A,K}}{K} \circ \bid_{\HHom{A}{K},K}$ in the upper left set. But $Id_{\HHom{\HHom{A}{K}}{K}}$ is also sent to $\bid_{A,K}$ in the lower left set, which is sent to $Id_{\HHom{A}{K}}$ in the upper left set by definition of $\bid_{A,K}$. This proves the two required formulas (see Definition 
\ref{defadjcouple})
for the composition of the unit and the counit in the adjoint couple 
(which are identical in this case). We leave to the reader the easy fact that $\bid_{A,K}$ is a generalized transformation (in $K$) (just using that the unit of the adjunction $(-\TTens *,\HHom{*}{-})$ is one).
The adjoint couple is then a suspended adjoint couple by Proposition \ref{susBifBid_prop}. This proves Point 1 of Definition \ref{defiSACRB}. Point 2 is proved using diagrams \diag{\ref{evr2}} and \diag{\ref{coevl2}}.
\end{proof}

\begin{coro} \label{tripleCat_coro}
The triple $(\cC,D_K,\bid_K)$ is a (suspended) category with duality (see Definition \ref{catwithduality}). When $\cC$ is triangulated closed, $D_K$ is exact, and $(\cC,D_K,\bid_K)$ is a triangulated category with duality which we often denote by $\cC_K$ for short. When $K$ is a dualizing object, the duality is strong. 
\end{coro}
\begin{prop} \label{DTKTDK}
The functors $D_K$ and $D_{TK}$ are exact. The isomorphism $\thh_2:\HHom{-}{T(*)} \to T\HHom{-}{*}$ defines a suspended duality preserving functor $\pair{Id_\cC,\thh_{2,-,K}}$ from $\cC_{TK}$ to $T(\cC,D_K,\bid_K)$. This functor is an isomorphism of 
triangulated categories with duality and therefore induces an isomorphism on 
Witt groups.
\end{prop}
\begin{proof}
We know by \ref{signTF} that $TD_K$ is $(-1)$-exact. Diagram $\diag{sus}$ should therefore be anti-commutative (see Definition \ref{morphTriangFunc}). It follows from the fact that $(\HHom{-}{*},\thh_1,\thh_2)$ is a suspended bifunctor. 
A square obtained by adjunction from $\diag{\ref{evr1}}$ in Lemma \ref{diagEv} 
then implies that $\thh_2$ defines a duality preserving functor 
(see Definition \ref{DualPresFunc_defi}). 
\end{proof}

We conclude this section by a trivial lemma for future reference.

\begin{lemm} \label{isoKisoWitt}
Let $\iota: K \to M$ be a morphism. Then $I_\iota=\pair{Id,\tilde{\iota}}$, where $\tilde{\iota}: D_K \to D_M$ is induced by $\iota$, is a duality preserving functor. 
This respects composition: if $\kappa:M \to N$ is another morphism,
then $I_{\kappa \iota}=I_\kappa I_\iota$. Let $\iota$ be an isomorphism, then if $K$ is dualizing and $D_K$ $\delta$-exact, the same is true for $M$ and $D_M$, and $I_\iota$ induces an isomorphism on Witt groups denoted by $I_\iota^W$.
\end{lemm}

\section{Functors between closed monoidal categories} \label{FunctClosedMon}

Assume from now on that all categories $\cC$ (maybe with an index) are symmetric monoidal
and equipped with an internal Hom, satisfying the set-up of the previous section.
We say that a functor $f^*: \cC_1 \to \cC_2$ is a symmetric monoidal (suspended, exact in both variables) functor when it comes equipped with an isomorphism of (suspended) bifunctors $\fp: f^*(-) \TTens f^*(*) \to f^*(- \TTens *)$ and, when a unit $\one$ for the tensor product is considered, an
isomorphism $f^*(\one) \simeq \one$ making the standard diagrams commutative
(see \cite[section XI.2]{MacLane98} for the details where such functors are
called {\it strongly monoidal}) \\

We will consider the following assumptions (used in the definition of some morphisms of functors): 

\noindent \assum{f}{The functor $f^*: \cC_1 \to \cC_2$ is symmetric monoidal (suspended, exact in both variables).} \label{assufmonoidal}
\assum{f}{We have a functor $f_*: \cC_2 \to \cC_1$ that fits into an adjoint couple $(f^*,f_*,\unit_*^*,\counit_*^*)$.} \label{assuAdjf*f*}
\assum{f}{We have a functor $f^!: \cC_1 \to \cC_2$ that fits into an adjoint couple $(f_*,f^!,\unit_*^!,\counit_*^!)$.} \label{assuAdjf*f!}
\assum{f}{The morphism $\q:f_*(-)\TTens * \to f_* (- \TTens f^*(*))$ from Proposition \ref{projFormMorphExists} is an isomorphism (the ``projection formula'' isomorphism).} \label{assuProjFormIso}
\assum{f,g}{The morphism of functors $\eps$ of Section \ref{BaseChange} is an isomorphism.} \label{epsIso}

In the following, we will define several natural transformations and establish commutative diagrams involving them. Since there are so many of them, for the convenience of the reader, we include a Table \ref{tableMorphFunct}, which displays (from the left to the right): the name of the natural 
transformation, its source and target functor, the necessary assumptions 
to define the natural transformation and where it is defined. 
\begin{table}
\begin{tabular}{|c|c|c|c|c|c|}
\hline
$\fp$ & $f^*(-) \TTens f^*(*) \to f^*(- \TTens *)$ & \assumption{\ref{assufmonoidal}}{f} & Section \ref{FunctClosedMon} \\
\hline
$\fh$ & $f^* \HHom{-}{*} \to \HHom{f^*(-)}{f^*(*)}$ & \assumption{\ref{assufmonoidal}}{f} & Prop \ref{defifh} \\
\hline
$\fg$ & $f_*(-) \TTens f_*(*) \to f_*(- \TTens *)$ & \assumption{\ref{assufmonoidal}}{f} \assumption{\ref{assuAdjf*f*}}{f} & Prop. \ref{fg_defi} \\
\hline
$\ff$ & $f_* \HHom{-}{*} \to \HHom{f_*(-)}{f_*(*)}$ & \assumption{\ref{assufmonoidal}}{f} \assumption{\ref{assuAdjf*f*}}{f} & Prop. \ref{defiffg} \\
\hline
$\q$ & $f_*(-) \TTens * \to f_*(- \TTens f^*(*))$ & \assumption{\ref{assufmonoidal}}{f} \assumption{\ref{assuAdjf*f*}}{f} & Prop. \ref{projFormMorphExists} \\
\hline
$\qh$ & $\HHom{-}{f_*(*)} \to f_* \HHom{f^*(-)}{*}$ & \assumption{\ref{assufmonoidal}}{f} \assumption{\ref{assuAdjf*f*}}{f} & Prop. \ref{projFormMorphExists} \\
\hline
$\rr$ & $f_* \HHom{*}{f^!(-)} \to \HHom{f_*(*)}{-}$ & 
\shortstack{\assumption{\ref{assufmonoidal}}{f} 
\assumption{\ref{assuAdjf*f*}}{f} \assumption{\ref{assuAdjf*f!}}{f} } & Thm. \ref{PushForward0_theo} \\
\hline
$\ssp$ & $f^!(- \TTens *) \to f^!(-) \TTens f^*(*)$ & \assumption{\ref{assufmonoidal}}{f} \assumption{\ref{assuAdjf*f*}}{f} \assumption{\ref{assuAdjf*f!}}{f} \assumption{\ref{assuProjFormIso}}{f} & Prop. \ref{projFormIso} \\
\hline
$\sh$ & $f^! \HHom{*}{-} \to \HHom{f^*(*)}{f^!(*)}$ & \assumption{\ref{assufmonoidal}}{f} \assumption{\ref{assuAdjf*f*}}{f} \assumption{\ref{assuAdjf*f!}}{f} \assumption{\ref{assuProjFormIso}}{f} & Prop. \ref{projFormIso} \\
\hline
$\dd_{K,M}$ & $D_K \TTens D_M \to D_{K \TTens M}$ & & Def. \ref{defidd} \\
\hline
$\xi$ & $\bar{g}^* f^* \to \bar{f}^* g^*$ & \shortstack{\assumption{\ref{assufmonoidal}}{f} \assumption{\ref{assufmonoidal}}{\bar{f}} \assumption{\ref{assufmonoidal}}{g} \assumption{\ref{assufmonoidal}}{\bar{g}} \\ \assumption{\ref{assuAdjf*f*}}{g} \assumption{\ref{assuAdjf*f*}}{\bar{g}} \assumption{\ref{assuAdjf*f!}}{g} \assumption{\ref{assuAdjf*f!}}{\bar{g}}} & Section \ref{BaseChange} \\
\hline
$\eps$ & $f^* g_* \to \bar{g}_* \bar{f}^*$ & \shortstack{\assumption{\ref{assufmonoidal}}{f} \assumption{\ref{assufmonoidal}}{\bar{f}} \assumption{\ref{assufmonoidal}}{g} \assumption{\ref{assufmonoidal}}{\bar{g}} \\ \assumption{\ref{assuAdjf*f*}}{g} \assumption{\ref{assuAdjf*f*}}{\bar{g}} \assumption{\ref{assuAdjf*f!}}{g} \assumption{\ref{assuAdjf*f!}}{\bar{g}}} & Section \ref{BaseChange} \\
\hline
$\gam$ & $\bar{f}^* \bar{g}^! \to g^! f^*$ & \shortstack{\assumption{\ref{assufmonoidal}}{f} \assumption{\ref{assufmonoidal}}{\bar{f}} \assumption{\ref{assufmonoidal}}{g} \assumption{\ref{assufmonoidal}}{\bar{g}} \assumption{\ref{assuAdjf*f*}}{g} \\ \assumption{\ref{assuAdjf*f*}}{\bar{g}} \assumption{\ref{assuAdjf*f!}}{g} \assumption{\ref{assuAdjf*f!}}{\bar{g}} \assumption{\ref{epsIso}}{f,g}} & Section \ref{BaseChange} \\
\hline
\end{tabular}
\caption{Morphisms of functors} \label{tableMorphFunct}
\end{table}

\subsection{The monoidal functor $f^*$}

In this section, we obtain duality preserving functors and morphisms related to a monoidal functor $f^*$.

\begin{prop} \label{defifh}
Under Assumption \assumption{\ref{assufmonoidal}}{f}, there is a unique morphism 
$$\fh: f^*\HHom{*}{-} \to \HHom{f^*(*)}{f^*(-)}$$
of (suspended) bifunctors such that the diagrams
$$\xymatrix{
\HHom{f^* X}{f^*(-\TTens X)} \ar@{}[dr]|{\diagram \label{Lapp1}} & f^*\HHom{X}{-\TTens X} \ar[l]_-{\fh} \\
 \HHom{f^* X}{f^*(-)\TTens f^*X} \ar[u]^{\fp} & f^*(-) \ar[l] \ar[u]
}
\hspace{-2ex}
\xymatrix{
f^*(-) \ar@{}[dr]|{\diagram \label{L'app1}} & \HHom{f^* X}{f^*(-)} \TTens f^*X \ar[l] \\
f^*(\HHom{X}{-}\TTens X) \ar[u] & f^*\HHom{X}{-} \TTens f^*X \ar[l]_{\fp} \ar[u]_{\fh}
}
$$
commute for every $X \in \cC_1$. For any $A$ and $B$, $\fh$ is given by the composition
$$\xymatrix@C=4ex{f^* \HHom{A}{B} \ar[r]^-{\coev^l} & \HHom{f^*A}{f^*\HHom{A}{B}\TTens f^*A} \ar[r]^-{\fp} & \HHom{f^*A}{f^*(\HHom{A}{B} \TTens A)} \ar[r]^-{\ev^l} & \HHom{f^*A}{f^*B}.
}$$
\end{prop}
\begin{proof}
Apply Lemma \ref{adjabbif} to $H=H'=f^*$, $J_1= - \TTens *$, $K_1= \HHom{*}{-}$, $J_2=(- \TTens f^*(*))$, $K_2=\HHom{f^*(*)}{-}$ and $a=\fp$ (recall that in an ACB, the variable denoted $*$ is the parameter, as explained in Definition \ref{defiAdjBif}).
Then define $\fh=b$.
\end{proof}

\begin{theo} (existence of the pull-back) \label{PullBack0_theo} 
Under Assumption \assumption{\ref{assufmonoidal}}{f}, the morphism 
$$\fh_K: f^*D_K^o \to D_{f^*K}^o(f^*)^o$$ 
defines a duality preserving functor $\pair{f^*,\fh_K}$ of (suspended, triangulated) categories with duality from $(\cC_1)_K$ to $(\cC_2)_{f^*K}$.
\end{theo}
\begin{coro} (Pull-back for Witt groups) \label{PullBack0Witt_coro}
When the dualities and the duality preserving functor are strong (\ie $\bid_K$, $\bid_{f^* K}$ and $\fh_K$ are isomorphisms), $\pair{f^*,\fh_K}$ induces a morphism of Witt groups $f^*_W:W^*(\cC_1,K) \to W^*(\cC_2,f^*K)$
by Proposition \ref{dualPresIsTransfer_prop}.
\end{coro}
\begin{proof}[Proof of Theorem \ref{PullBack0_theo}]
We need to show that the diagram
$$\xymatrix@!C=20ex{
f^* \ar[d]_{\bid_K f^*} \ar[r]^{f^* \bid_K} & f^* D_K^o D_K \ar[d]^{\fh_K D_K} \\
D_{f^* K}^o D_{f^* K} f^* \ar[r]^{D_{f^*K}^o \fh_K^o} & D_{f^* K}^o (f^*)^o D_K
}$$
commutes. This follows from the commutative diagram (for all $A$ in $\cC_1$)
$$\xymatrix@R=3ex@C=1ex{
f^* A \ar[d]_{\coev^l} \ar[r]^-{\coev^l} \ar@{}[dr]|{\diag{\ref{Lapp1}}} & f^* \HHom{\HHom{A}{K}}{A \TTens \HHom{A}{K}} \ar[d]_{\fh} \ar[r]^-{\ev^r} \ar@{}[dr]|{\diag{mf}} & f^* \HHom{\HHom{A}{K}}{K} \ar[d]^{\fh} \\
\HHom{f^*\HHom{A}{K}}{f^* A \TTens f^* \HHom{A}{K}} \ar[r]^-{\fp} \ar@{}[dr]|{\diag{\ref{L'app1}'}} \ar[d]_{\fh} & \HHom{f^*\HHom{A}{K} }{f^* (A \TTens \HHom{A}{K})} \ar[r] & \HHom{f^* \HHom{A}{K}}{f^* K} \\
\HHom{f^*\HHom{A}{K}}{f^* A \TTens \HHom{f^*A}{f^*K}} \ar@/_/[rru]_{\ev^r}  & & \HHom{\HHom{f^* A}{f^* K}}{f^* K} \ar[u]_{\fh} \\
\HHom{\HHom{f^* A}{f^* K}}{f^* A \TTens \HHom{f^* A}{f^* K}} \ar[u]^{\fh} \ar@/_/[rru]_{\ev^r} \ar@{}[ur]|{\diag{mf}} & & \\
}$$
where $\diag{\ref{L'app1}'}$ is obtained from $\diag{\ref{L'app1}}$ by using the compatibility of $\fp$ with $\cc$. By functoriality of $\coev$, the counit of the adjunction of bifunctors $(- \TTens *,\HHom{-}{*})$, we can complete the left vertical part of the diagram as a commutative square by a morphism $\coev^l$ from $f^* A$ to the bottom entry. The outer part of this bigger diagram is therefore the one we are looking for. 
In the suspended (or triangulated) case, $\fh$ is a morphism of suspended functors by Proposition \ref{defifh}. 
\end{proof}

\subsection{Adjunctions $(f^*,f_*)$ and $(f_*,f^!)$ and the projection morphism}

In this section, we will assume that we have adjoint couples $(f^*,f_*)$ and $(f_*,f^!)$, and obtain the projection morphism $f_*(-)\TTens * \to f_*(-\TTens f^*(*))$ as well as several related commutative diagrams. We also construct the morphism $\rr$ which turns $f_*$ into a duality preserving functor (Theorem \ref{PushForward0_theo}).

\begin{prop} \label{fg_defi}
Assume \assumption{\ref{assufmonoidal}}{f} and \assumption{\ref{assuAdjf*f*}}{f}. Then
there is a unique morphism of (suspended) bifunctors
$$\fg: f_*(-) \TTens f_*(-) \to f_* (-\TTens -)$$
such that the diagrams
$$\xymatrix{
- \TTens X \ar[d] \ar[r] \ar@{}[dr]|{\diagram \label{Happ0}} & f_* f^*(-) \TTens f_* f^* X \ar[d]^{\fg} \\ 
f_* f^* (- \TTens X) \ar[r]^{\fp^{-1}} & f_*(f^*(-)\TTens f^* X) 
}
\xymatrix{
f^*(f_*(-) \TTens f_*X) \ar[d]_{\fg} \ar[r]_{\fp^{-1}} \ar@{}[dr]|{\diagram \label{H'app0}} & f^* f_* (-) \TTens f^* f_* X \ar[d] \\
f^* f_*(-\TTens X) \ar[r] & - \TTens X 
}
$$
commute for every $X \in \cC_1$. For any $A$ and $B$, $\fg$ is given by the composition
$$\xymatrix{
f_*A \TTens f_*B \ar[r]^-{\unit_*^*} & f_* f^* (f_*A \TTens f_*B) \ar[r]^-{\fp^{-1}} & f_* (f^* f_*A \TTens f^* f_*B) \ar[r]^-{\counit_*^* \TTens \counit_*^*} & f_* (A \TTens B).
}$$
\end{prop}
\begin{proof}
Apply Lemma \ref{adjab} (resp. Theorem \ref{adjsquareT}) to $H=H'=(-\TTens-)$, $J_1=f^*\times f^*$ $K_1= f_* \times f_*$, $J_2= f^*$, $K_2=f_*$, and $a=\fp^{-1}$ to obtain a unique $\fg=b$ satisfying \diag{\ref{Happ0}} and \diag{\ref{H'app0}} and given by the above composition. In the suspended case, $\fg$ is a morphism of suspended bifunctors because it is given by a composition of morphisms of suspended functors and bifunctors.
\end{proof}

\begin{prop} \label{defiffg} Assume \assumption{\ref{assufmonoidal}}{f} and \assumption{\ref{assuAdjf*f*}}{f}.
Then, there is a unique morphism 
$$\ff: f_*\HHom{*}{-} \to  \HHom{f_*(*)}{f_*(-)}$$
of (suspended) bifunctors such that the diagrams
$$\xymatrix{
\HHom{f_*X}{f_*(-\TTens X)} \ar@{}[dr]|{\diagram \label{Happ1}} & f_*\HHom{X}{-\TTens X} \ar[l]_{\ff} \\
 \HHom{f_* X}{f_*(-)\TTens f_*X} \ar[u]^{\fg} & f_*(-) \ar[l] \ar[u]
}
\xymatrix{
f_*(-) \ar@{}[dr]|{\diagram \label{H'app1}} & \HHom{f_* X}{f_*(-)} \TTens f_*X \ar[l] \\
f_*(\HHom{X}{-}\TTens X) \ar[u] & f_*\HHom{X}{-} \TTens f_*X \ar[l]_{\fg} \ar[u]_{\ff \TTens id}
}$$
commute for every $X \in \cC_1$. For any $A$ and $B$, $\ff$ is given by the composition
$$\xymatrix@C=4ex{
f_*\HHom{A}{B} \ar[r]^-{\coev^l} & \HHom{f_*A}{f_*\HHom{A}{B} \TTens f_*A} \ar[r]^-{\fg} & \HHom{f_*A}{f_*(\HHom{A}{B} \TTens A)} \ar[r]^-{\ev^l} & \HHom{f_*A}{f_*B}.
}$$
\end{prop}
\begin{proof}
Apply Lemma \ref{adjabbif} (resp. Lemma \ref{adjabbifT}) to $J_1=(-\TTens *)$, $K_1=\HHom{*}{-}$, $J_2=(-\TTens f_*(*))$, $K_2=\HHom{f_*(*)}{-}$, $H=H'=f_*$ and $a=\fg$. 
\end{proof}

\begin{lemm}
The diagram
$$\xymatrix{
f_*A \TTens f_* B \ar[d]_{\cc(f_* \TTens f_*)} \ar[r]^{\fg} \ar@{}[dr]|{\diagram \label{fgc}} & f_* (A \TTens B) \ar[d]^{f_* \cc} \\
f_*B \TTens f_* A \ar[r]^{\fg} & f_* (B \TTens A)
}$$
is commutative.
\end{lemm}
\begin{proof}
Apply Lemma \ref{cube} to the cube
$$\xymatrix@!C=3ex@!R=2ex{
\cC_2 \times \cC_2 \ar[rr]^{\TTens} \ar[dr]^{x} & & \cC_2 \ar[dr]^{Id} & \\
 & \cC_2 \times \cC_2 \ar@2[ur]^{c} \ar[rr]^{\TTens} & & \cC_2 \\
\cC_1 \times \cC_1 \ar[uu]^{f^* \times f^*} \ar[dr]_{x} & & & \\
 & \cC_1 \times \cC_1 \ar@{=}[luuu]|{\rule[-0.5ex]{0ex}{2ex} id} \ar[uu]_(.40){f^* \times f^*} \ar[rr]_{\TTens} & & \cC_1 \ar@2[lluu]_{\fp^{-1}} \ar[uu]_{f^*}
} 
\hspace{5ex}
\xymatrix@!C=3ex@!R=2ex{
\cC_2 \times \cC_2 \ar[rr]^{\TTens} & & \cC_2 \ar[dr]^{Id} & \\
 & & & \cC_2 \\
\cC_1 \times \cC_1 \ar[uu]^{f^*\times f^*} \ar[dr]_{x} \ar[rr] & & \cC_1 \ar@2[lluu]_(.60){\fp^{-1}} \ar[uu]^{f^*} \ar[dr]_{Id} & \\
 & \cC_1 \times \cC_1 \ar@2[ur]^{c} \ar[rr]_{\TTens} & & \cC_1 \ar@{=}[luuu] \ar[uu]_{f^*}
} 
$$
where $x$ is the functor exchanging the components. Note that the morphism of functors $f_*(- \TTens *) \to f_*(-) \TTens f_*(*)$ obtained on the front and back squares indeed coincides with $\fg$ by construction.
\end{proof}

\begin{prop}
In the suspended case, under Assumption \assumption{\ref{assuAdjf*f*}}{f}
there is a unique way of turning $f_*$ into a suspended functor such that $(f^*,f_*)$ is a suspended adjoint couple. If further Assumption \assumption{\ref{assuAdjf*f!}}{f}
holds, then there is a unique way of turning $f^!$ into a suspended functor such that $(f_*,f^ !)$ is a suspended adjoint couple.
\end{prop}
\begin{proof}
Both results follow directly from Proposition \ref{suspAdjExists}. 
\end{proof}

\begin{prop} \label{projFormMorphExists}
Under assumptions \assumption{\ref{assufmonoidal}}{f} and \assumption{\ref{assuAdjf*f*}}{f}, there is a unique morphism 
$$\q:f_*(-)\TTens * \to f_* (- \TTens f^*(*))$$
and a unique isomorphism
$$\qh: \HHom{*}{f_*(-)} \to f_* \HHom{f^* (*)}{-}$$
of (suspended) bifunctors such that the diagrams
$$\xymatrix{
\HHom{X}{f_*(-\TTens f^* X)} \ar@{}[dr]|{\diagram \label{Rapp1}} & f_*\HHom{f^* X}{-\TTens f^* X} \ar[l]^{\qh^{-1}} \\
\HHom{X}{f_*(-)\TTens X} \ar[u]^{\q} & f_*  \ar[l] \ar[u]
}\hspace{-3ex}
\xymatrix{
f^* \ar@{}[dr]|{\diagram \label{R'app1}} & \HHom{X}{f_*(-)} \TTens X \ar[l] \\
f_*(\HHom{f^*X}{-} \TTens f^* X) \ar[u] & f_*\HHom{f^*X}{-} \TTens X \ar[l]_{\q} \ar[u]_{\qh^{-1}} 
}$$
$$\xymatrix{
- \TTens X \ar[d] \ar[r] \ar@{}[dr]|{\diagram \label{G1app1}} & f_* f^*(-) \TTens X \ar[d]^{\q} \\ 
f_* f^* (- \TTens X) \ar[r]^{\fp^{-1}} & f_*(f^*(-)\TTens f^* X) 
}
\xymatrix{
f^*(f_*(-) \TTens X) \ar[d]_{\q} \ar[r]_{\fp^{-1}} \ar@{}[dr]|{\diagram \label{G2app1}} & f^* f_* (-) \TTens f^* X \ar[d] \\
f^* f_*(-\TTens f^* X) \ar[r] & - \TTens f^* X 
}
$$
$$\xymatrix{
\HHom{X}{-} \ar[d] \ar[r] \ar@{}[dr]|{\diagram \label{F1app1}} & \HHom{X}{f_* f^*(-)} \ar[d]^{\qh} \\ 
f_* f^* \HHom{X}{-} \ar[r]^-{\fh} & f_* \HHom{f^* X}{f^*(-)}
}
\xymatrix{
f^* \HHom{X}{f_*(-)} \ar[d]_{\qh} \ar[r]_-{\fh} \ar@{}[dr]|{\diagram \label{F2app1}} & \HHom{f^* X}{f^* f_* (-)} \ar[d] \\
f^* f_* \HHom{f^* X}{-} \ar[r] & \HHom{f^* X}{-}
}
$$
commute for any $X \in \cC_1$. For any $A$ and $B$, $\q$ is given by the composition
$$\xymatrix@C=4ex{
f_*A \TTens B \ar[r]^-{\unit^*_*} & f_* f^* (f_* A \TTens B) \ar[r]^-{\fp} & f_*(f^* f_* A \TTens f^*B) \ar[r]^-{\counit^*_*} & f_* (A\TTens f^* B)
}$$
$\qh$ by
$$\xymatrix@C=4ex{
\HHom{A}{f_*B} \ar[r]^-{\unit^*_*} & f_* f^* \HHom{A}{f_*B} \ar[r]^-{\fh} & f_* \HHom{f^*A}{f^*f_*B} \ar[r]^-{\counit^*_*} & f_* \HHom{f^*A}{B}
}$$
and $\qh^{-1}$ by
$$\xymatrix@C=4ex{
f_* \HHom{f^*A}{B} \ar[r]^-{\coev^l} & \HHom{A}{f_*\HHom{f^*A}{B} \TTens A} \ar[r]^-{\q} & \HHom{A}{f_*(\HHom{f^*A}{B} \TTens f^*A)} \ar[r]^-{\ev^l} & \HHom{A}{f_*B}.
}$$
\end{prop}
\begin{proof}
Apply Point 3 of Theorem \ref{theoGenAdj} with $L=L'=f^*$, $R=R'=f_*$, $G_1=-\TTens *$, $G_2= - \TTens f^*(*)$, $F_1= \HHom{*}{-}$, $F_2=\HHom{f^*(*)}{-}$, $g_L=(g'_L)^{-1}=\fp$.
Then define $\q=g_R$ and $\qh=f_R'$.
\end{proof}

\begin{lemm} \label{fgq}
The composition 
$$\xymatrix{f_*(-)\TTens (-) \ar[r]^-{id \TTens \unit_*^*} & f_*(-)\TTens f_* f^*(-) \ar[r]^{\fg} & f_*(- \TTens f^*(-))}$$
coincides with $\q$.
\end{lemm}
\begin{proof}
This follows from the commutative diagram (for any $A$ and $B$ in $\cC_2$)
$$\xymatrix@R=3ex{
f_* A \TTens B \ar[d] \ar@/_12ex/[ddd]_{\q} \ar[r] \ar@{}[dr]|{\diag{mf}} & f_* A \TTens f_* f^* B \ar[d] \ar@/^3ex/[dddr]^{\fg}\\
f_* f^*(f_* A \TTens B) \ar[d]_{\fp^{-1}} \ar[r] \ar@{}[dr]|{\diag{mf}} & f_* f^*(f_* A \TTens f_* f^* B) \ar[d]^{\fp^{-1}} \\
f_* (f^* f_* A \TTens f^* B) \ar[d] \ar[r] \ar@{}[dr]|{\diag{mf}} & f_* (f^* f_* A \TTens f^* f_* f^* B) \ar[d] \\
f_* (A \TTens f^* B) \ar[r] \ar@/_6ex/[rr]_{Id}^{\diag{adj}}  & {f_* (A \TTens f^* f_* f^* B)} \ar[r] & f_* (A \TTens f^* B) \\
}$$
in which the curved maps are indeed $\fg$ and $\q$ by construction.
\end{proof}
\begin{lemm}
The composition
$$\xymatrix@C=7ex{f_*(-) \TTens f_*(-) \ar[r]^{\q} & f_*(- \TTens f^* f_* (-)) \ar[r]^-{f_*(id \TTens \counit_*^*)} & f_*(- \TTens -)}$$
coincides with $\fg$.
\end{lemm}
\begin{proof}
This follows directly from the construction of $\q$ and diagram \diag{\ref{Happ0}}.
\end{proof}
\begin{lemm} \label{ffqh}
The composition 
$$\xymatrix{f_*\HHom{f^*(-)}{*} \ar[r]^-{\ff} & \HHom{f_*f^*(-)}{f_*(*)} \ar[r]^-{\unit_*^*} & \HHom{-}{f_*(*)} }$$
coincides with $\qh$.
\end{lemm}
\begin{proof}
This follows from the commutative diagram (for any $A$ in $\cC_1$ and $B$ in $\cC_2$)
$$\xymatrix@R=4ex{
f_* \HHom{f^* A}{B} \ar[d]_{\coev^l} \ar@<-3ex>@/_15ex/[ddd]_{\ff} \ar[r]^{\coev^l} \ar@<-3ex>@{}[r]|{\diag{gen}} & \HHom{A}{f_*\HHom{f^*A}{B}\TTens A} \ar[d]^{\unit_*^*} \ar@<6ex>@/^9ex/[dd]^{\q} \\
\HHom{f_*f^* A}{f_* \HHom{f^*A}{B}\TTens f_* f^* A} \ar[d]_{\fg} \ar[r]^{\unit_*^*} \ar@{}[dr]|{\diag{mf}} & \HHom{A}{f_*\HHom{f^*A}{B}\TTens f_*f^*A} \ar[d]^{\fg} \\
\HHom{f_*f^* A}{f_*(\HHom{f^*A}{B}\TTens f^* A)} \ar[d]_{\ev^l} \ar[r] \ar@<-3ex>@{}[r]|{\diag{mf}} & \HHom{A}{f_*(\HHom{f^*A}{B}\TTens f^*A)} \ar[d]^{\ev^l} \\
\HHom{f_*f^* A}{f_*B} \ar[r]^{\unit_*^*} & \HHom{A}{f_*B} \\
}$$
in which the left curved arrow is $\ff$ by construction, the right one is $\q$ by Lemma \ref{fgq} and the composition from the top left corner to the bottom right 
one along the upper right corner is then $\qh$ by construction.
\end{proof}

\begin{theo}(existence of the push-forward) \label{PushForward0_theo} 
Under assumptions \assumption{\ref{assufmonoidal}}{f}, \assumption{\ref{assuAdjf*f*}}{f}, \assumption{\ref{assuAdjf*f!}}{f}, let $\rr$ be the (suspended, exact) bifunctor defined by the composition 
$$\xymatrix{f_*\HHom{-}{f^!(-)} \ar[r]^-{\ff} & \HHom{f_*(-)}{f_*f^!(-)} \ar[r]^-{\counit_*^!} & \HHom{f_*(-)}{-}}.$$
Then $\pair{f_*,\rr_K}$ is a (suspended, exact) duality preserving functor from $(\cC_2,D_{f^! K}, \bid_{f^! K})$ to $(\cC_1,D_K,\bid_K)$.
\end{theo}
\begin{proof}
We need to prove that the diagram \diag{P} in Definition \ref{DualPresFunc_defi} is commutative. We consider the commutative diagram
$$\xymatrix@R=3ex@C=0.5ex{
f_* A \ar[r]^-{\coev^l} \ar[d]_{\coev^l} \ar@{}[dr]|{\diag{\ref{Happ1}}} & f_* \HHom{\HHom{A}{f^!K}}{A \TTens \HHom{A}{f^!K}} \ar[r]^-{\ev^r} \ar[d]_{\ff} & f_*\HHom{\HHom{A}{f^!K}}{f^! K} \ar@/^12ex/[ddl]^{\ff} \ar@{}[dddl]^{\diag{mf}} \\
\HHom{f_*\HHom{A}{f^! K}}{f_* A \TTens f_* \HHom{A}{f^! K}} \ar[r] \ar[d]_{id \TTens \ff} \ar@{}[dr]|{\diag{\ref{H'app1}'}}  & \HHom{f_*\HHom{A}{f^! K}}{f_*(A \TTens \HHom{A}{f^! K})} \ar[d]_{\ev^r} \\
\HHom{f_*\HHom{A}{f^! K}}{f_* A \TTens \HHom{f_* A}{f_* f^! K}} \ar[r]^-{\ev^r} \ar[d]_{
id \TTens \counit_*^!} \ar@{}[dr]|{\diag{mf}} & \HHom{f_*\HHom{A}{f^! K}}{f_*f^! K} \ar[d]_{\counit_*^!} \\
\HHom{f_*\HHom{A}{f^! K}}{f_* A \TTens \HHom{f_* A}{K}} \ar[r]^-{\ev^r} & \HHom{f_*\HHom{A}{f^! K}}{K}
}$$
where \diag{\ref{H'app1}'} is obtained from \diag{\ref{H'app1}} by using $\cc$. The top row is $f_* \bid_{f^! K}$ and the composition from the top right corner to the bottom one is $\rr$, both by definition. The result follows if we prove that the composition from the top left corner to the bottom right one, going through the bottom left corner, is $\bid_K f_*$ followed by $D_K^o \rr^o$. This follows from the commutative diagram
$$\xymatrix@R=4ex{
f_* A \ar[r]^-{\coev^l} \ar[d]_{\coev^l} \ar@{}[dr]|{\diag{gen}} & \HHom{f_*\HHom{A}{f^! K}}{f_* A \TTens f_* \HHom{A}{f^! K}} \ar[d]^{id \TTens \rr} \\ 
\HHom{\HHom{f_* A}{K}}{f_* A \TTens \HHom{f_* A}{K}} \ar[r]^-{\rr} \ar[d]_{\ev^r} \ar@{}[dr]|{\diag{mf}} & \HHom{f_*\HHom{A}{f^! K}}{f_* A \TTens \HHom{f_* A}{K}} \ar[d]^{\ev^r} \\
\HHom{\HHom{f_* A}{K}}{K} \ar[r]^-{\rr} & \HHom{f_*\HHom{A}{f^! K}}{K}
}$$
in which the left vertical composition is $\bid_K f_*$ and the composition along the upper right corner is the composition along the lower left corner of the previous diagram.
\end{proof}

\subsection{When the projection morphism is invertible}

In this section, we prove that the projection morphism is an isomorphism if and only if the morphism $\rr$ is. In that case, we obtain a new morphism $\ssp$ that will be used in Section \ref{ProjForm_sec} to state a projection formula.
-
\begin{prop} \label{projFormIso}
Under assumptions \assumption{\ref{assufmonoidal}}{f}, \assumption{\ref{assuAdjf*f*}}{f} and \assumption{\ref{assuAdjf*f!}}{f}, there is a unique morphism
$$\sh': \HHom{f^*(*)}{f^!(-)} \to f^! \HHom{*}{-}$$
of (suspended) bifunctors such that the diagrams
$$\xymatrix{
\HHom{f^*X}{-} \ar[r] \ar[d] \ar@{}[dr]|{\diagram \label{F1app2}} & \HHom{f^*X}{f^! f_*(-)} \ar[d]^{\sh'} \\
f^! f_* \HHom{f^*X}{-} \ar[r]^{\qh^{-1}} & f^! \HHom{X}{f_*(-)} 
}
\xymatrix{
f_* \HHom{f^*X}{f^!(-)} \ar[d]_{\sh'} \ar[r]_{\qh^{-1}} \ar@{}[dr]|{\diagram \label{F2app2}} & \HHom{X}{f_*f^!(-)} \ar[d] \\
f_* f^! \HHom{X}{-} \ar[r] & \HHom{X}{-}
}$$
commute for any $X \in \cC_1$. The morphism of bifunctors $\sh'$ is invertible if and only if $\q$ is invertible (Assumption \assumption{\ref{assuProjFormIso}}{f}). In that case, we denote by $\sh$ the inverse of $\sh'$, and there is a unique morphism of (suspended) bifunctors
$$\ssp: f^!(-)\TTens f^*(*) \to f^!(- \TTens *)$$
such that the diagrams
$$\xymatrix{
\HHom{f^*X}{f^!(-\TTens X)} \ar@{}[dr]|{\diagram \label{Rapp2}} & f^! \HHom{X}{- \TTens X} \ar[l]^-{\sh} \\
\HHom{f^*X}{f^!(-)\TTens f^*X} \ar[u]^{\ssp} & f^! \ar[l] \ar[u]
}
\xymatrix{
f^! \ar@{}[dr]|{\diagram \label{R'app2}} & \HHom{f^* X}{f^!(-)} \TTens f^*X \ar[l] \\
f^! (\HHom{X}{-} \TTens X) \ar[u] & f^!\HHom{X}{-} \TTens f^*X \ar[l]_{\ssp} \ar[u]_{\sh}
}$$
$$\xymatrix{
-\TTens f^* X \ar[d] \ar[r] \ar@{}[dr]|{\diagram \label{G1app2}} & f^! f_* (-) \TTens f^* X \ar[d]^{\ssp} \\
f^! f_* (-\TTens f^*X) \ar[r]^{\q^{-1}} & f^!(f_*(-) \TTens X) 
}
\xymatrix{
f_*(f^!(-)\TTens f^* X) \ar[d]_{\ssp} \ar[r]_{\q^{-1}} \ar@{}[dr]|{\diagram \label{G2app2}} & f_* f^!(-) \TTens X \ar[d] \\
f_* f^! (- \TTens X) \ar[r] & - \TTens X 
}
$$
commute for any $X \in \cC_1$. For any $A$ and $B$, $\sh'$ is given by the composition
$$\xymatrix@C=4ex{
\HHom{f^*A}{f^!B} \ar[r]^-{\unit^!_*} & f^! f_* \HHom{f^*A}{f^!B} \ar[r]^-{\qh^{-1}} & f^! \HHom{A}{f_* f^! B} \ar[r]^-{\counit^!_*} & f^! \HHom{A}{B}
}$$
and $\ssp$ by
$$\xymatrix@C=4ex{
f^!A \TTens f^* B \ar[r]^-{\unit^!_*} & f^!f_*(f^!A \TTens f^*B) \ar[r]^-{\q^{-1}} & f^!(f_*f^!A \TTens B) \ar[r]^-{\counit^!_*} & f^!(A \TTens B)
}$$
or equivalently by
$$\xymatrix@C=4ex{
f^!A \TTens f^* B \ar[r]^-{\coev^l} & f^! \HHom{B}{A \TTens B} \TTens f^* B \ar[r]^-{\sh} & \HHom{f^* B}{f^!(A \TTens B)} \TTens f^* B \ar[r]^-{\ev^l} & f^!(A \TTens B).
}$$
\end{prop} 
\begin{proof}
Apply Point 2 of Theorem \ref{theoGenAdj} with $L=L'=f_*$, $R=R'=f^!$, $G_1=-\TTens f^*(*)$, $G_2= - \TTens *$, $F_1= \HHom{f^*(*)}{-}$, $F_2=\HHom{*}{-}$ and $f_L=\qh^{-1}$ to obtain $\sh'=f'_R$ and diagrams \diag{\ref{F1app2}} and \diag{\ref{F2app2}}. Then Point 3 of the same Theorem tells us that $\sh'=f'_R$ is invertible if and only if $\q=g_L$ is (Note that $\q$ and $\qh$ indeed correspond to each other through Point 1 of the Theorem by construction). In that case, we can use Point 2 to obtain $\ssp=g_R$ from $\pi^{-1}=g'_L$ and diagrams \diag{\ref{Rapp2}}, \diag{\ref{R'app2}}, \diag{\ref{G1app2}} and \diag{\ref{G2app2}}.
\end{proof}

In particular, under Assumption \assumption{\ref{assuProjFormIso}}{f}, this gives isomorphisms $\sh_K: f^! D_K^o \to D_{f^!K}^o (f^*)^o$ (in $\cC_2$) for each $K$.

\begin{lemm} \label{ffr}
The composition 
$$\xymatrix{f_*\HHom{-}{f^!(-)} \ar[r]^-{\counit^*_*} & f_* \HHom{f^*f_*(-)}{f^!(-)} \ar[r]^-{\sh'} & f_*f^! \HHom{f_*(-)}{-} \ar[r]^{\counit^!_*} & \HHom{f_*(-)}{-}}$$
coincides with $\rr$.
\end{lemm}
\begin{proof}
This follows from the commutative diagram (for every $A$ in $\cC_2$ and $K$ in $\cC_1$)
$$\xymatrix@R=3ex{
f_*\HHom{A}{f^!K} \ar[d]_{\ff} \ar[r] \ar@{}[dr]|{\diag{mf}} & f_* \HHom{f^* f_* A}{f^!K} \ar[d]^{\ff} \ar[r]_{\sh'} \ar@{}[ddr]|{\diag{\ref{F2app2}}} & f_* f^! \HHom{f_*A}{K} \ar[dd]^{\counit^!_*} \\
\HHom{f_*A}{f_*f^!K} \ar[r] \ar@<2ex>@{}[dr]|(.80){\diag{adj}} \ar@{=}[dr] & \HHom{f_*f^*f_* A}{f_*f^!K} \ar[d]^{\unit_*^*} & \\
 & \HHom{f_* A}{f_* f^! K} \ar[r]^{\counit_*^!} & \HHom{f_* A}{K}
}$$
identifying the vertical composition in the middle with $\qh^{-1}$ 
by Lemma \ref{ffqh} to recognize $\diag{\ref{F2app2}}$. 
\end{proof}

\begin{prop} \label{adjpairsf_*_prop}
Under assumptions \assumption{\ref{assufmonoidal}}{f}, \assumption{\ref{assuAdjf*f*}}{f}, \assumption{\ref{assuAdjf*f!}}{f}, $\rr$ is an isomorphism if and only if \assumption{\ref{assuProjFormIso}}{f} holds.
\end{prop}

\begin{coro} (Push-forward for Witt groups) \label{PushForward0Witt_coro}
When the dualities $D_{f^!K}$ and $D_K$ are strong and
Assumption \assumption{\ref{assuProjFormIso}}{f} is
satisfied, then $\pair{f_*,\rr_K}$ is a strong duality preserving functor
and thus induces a morphism of Witt groups
$$f_*^W:W^*(\cC_2,f^! K) \to W^*(\cC_1,K).$$
\end{coro}
\begin{proof}
This follows from Definition \ref{Wgraded},
Theorem \ref{PushForward0_theo}, Proposition
\ref{adjpairsf_*_prop} and Proposition \ref{dualPresIsTransfer_prop}.
\end{proof}

\begin{proof}[Proof of Proposition \ref{adjpairsf_*_prop}]
Let us first recall that $(\HHom{-}{*},\HHom{-}{*}^o,\bid,\bid^o)$ is a (suspended) ACB from $\cC_1$ to $\cC_1^o$ with parameter in $\cC_1^o$ by Proposition \ref{HomHomSuspAdj_prop}, and thus $(\HHom{-}{(f^!)^o(*)},\HHom{-}{f^!(*)}^o,\bid,\bid^o)$ is a (suspended) ACB from $\cC_2$ to $\cC_2^o$ with parameter in $\cC_1^o$ by Lemma \ref{changeParamSus}. We then apply Theorem \ref{theoGenAdj} with the square
$$\xymatrix@C=10ex{
\cC_1 \ar@<1.5ex>[d]^{\HHom{-}{*}} \ar@{}[d]|{\cC_1^o} \ar@<0.5ex>[rr]^{f^*} & & \cC_2 \ar@<1.5ex>[d]^{\HHom{-}{(f^!)^o(*)}} \ar@{}[d]|{\cC_1^o} \ar@<0.5ex>[ll]^{f_*} \\
\cC_1^o \ar@<1.5ex>[u]^{\HHom{-}{*}^o} \ar@<0.5ex>[rr]^{(f^!)^o} & & \cC_2^o \ar@<1.5ex>[u]^{\HHom{-}{f^!(*)}^o} \ar@<0.5ex>[ll]^{f_*^o} \\
}
$$ 
starting with $g'_L=(\sh')^o$. By Point 2 of \loccit, we obtain a unique $g_R$ such that diagrams \diag{G1} and \diag{G2} commute. This $g_R$ coincides with $\rr^o$ by Lemma \ref{ffr}. By Point 1 of \loccit, we then obtain a unique $f_R$ such that diagrams \diag{R} and \diag{R'} commute. But $f_R$ has to coincide with $\rr$ by uniqueness, since those commutative diagrams were already proven in Theorem \ref{PushForward0_theo} with $\rr$ and $\rr^o$. By Point 3 of \loccit, $f_R$ is an isomorphism if and only if  $g'_L=(\sh')^o$ is, which is an isomorphism if and only if \assumption{\ref{assuProjFormIso}}{f} holds by Proposition \ref{projFormIso}. We have thus proved that $\rr$ is an isomorphism if and only if \assumption{\ref{assuProjFormIso}}{f} holds. 
\end{proof}

\subsection{Products} 

We now define one more classical morphism related to products for Witt groups. It is easily checked that since $(-\TTens *, \HHom{*}{-})$ forms an ACB from $\cC$ to $\cC$ with parameter in $\cC$, $((-_1 \TTens *_1) \times (-_2 \TTens *_2), \HHom{*_1}{-_1} \times \HHom{*_2}{-_2})$ forms an ACB from $\cC \times \cC$ to $\cC \times \cC$ with parameter in $\cC \times \cC$, and $(- \TTens (*_1 \TTens *_2), \HHom{*_1 \TTens *_2}{-})$ forms an ACB from $\cC$ to $\cC$ with parameter in $\cC \times \cC$. We can thus consider the situation of Lemma \ref{adjabbif} applied to the square
$$\xymatrix@C=30ex{
\cC \times \cC \ar[r]^{\TTens} \ar@<2ex>[d]^{\HHom{*_1}{-_1} \times \HHom{*_2}{-_2}} \ar@{}[d]|{\cC \times \cC} & \cC \ar@<1ex>[d]^{\HHom{*_1 \TTens *_2}{-}} \ar@{}[d]|{\cC} \\
\cC \times \cC \ar[r]_{\TTens} \ar@<2ex>[u]^{(-_1 \TTens *_1)\times (-_2 \TTens *_2)} & \cC \ar@<1ex>[u]^{- \TTens(*_1 \TTens *_2)}
}$$
and to the morphism of bifunctors $\exch: (-_1 \TTens -_2) \TTens (*_1 \TTens *_2) \to (-_1 \TTens *_1) \TTens (-_2 \TTens *_2)$. We thus obtain its mate, a morphism of bifunctors 
$$\dd : \HHom{*_1}{-_1} \TTens \HHom{*_2}{-_2} \to \HHom{*_1 \TTens *_2}{-_1 \TTens -_2}$$
which, for any $A,B,K,M$ is given by the composition
$$\xymatrix@R=2ex@C=8ex{
 \HHom{A}{K} \TTens \HHom{B}{M} \ar[r]^-{\coev^l} & \HHom{A \TTens B}{(\HHom{A}{K}\TTens \HHom{B}{M})\TTens (A \TTens B)} \ar[d]^{\exch} \\
 \HHom{A \TTens B}{K \TTens M} & \HHom{A \TTens B}{(\HHom{A}{K}\TTens A) \TTens (\HHom{B}{M} \TTens B)} \ar[l]_-{\ev^l \TTens \ev^l} 
}$$
and obtain the following commutative diagrams
$$\xymatrix@R=2ex{
(\HHom{X_1}{K} \TTens \HHom{X_2}{M})\TTens (X_1 \TTens X_2) \ar[d]_{\exch} \ar[r]^{\dd} \ar@{}[dr]|{\diagram \label{Happ4}} & \HHom{X_1 \TTens X_2}{K \TTens M} \TTens (X_1 \TTens X_2) \ar[d]^{\ev^l} \\
(\HHom{X_1}{K} \TTens X_1) \TTens (\HHom{X_2}{M} \TTens X_2) \ar[r]_-{\ev^l \TTens \ev^l} & K \TTens M
}$$
$$\xymatrix@R=2ex{
K \TTens M \ar[r]^-{\coev^l \TTens \coev^l} \ar[d]_{\coev^l} \ar@{}[dr]|{\diagram \label{H'app4}} & \HHom{X_1}{K \TTens X_1} \TTens \HHom{X_2}{M \TTens X_2} \ar[d]^{\dd} \\
\HHom{X_1 \TTens X_2}{(K \TTens M)\TTens (X_1 \TTens X_2)} \ar[r]_-{\exch} & \HHom{X_1 \TTens X_2}{(K \TTens X_1) \TTens (M \TTens X_2)} 
}$$ 
as well.
\begin{defi} \label{defidd}
Consider objects $K,M \in \cC$. 
We write $\dd_{K,M}$ for the morphism of functors $D_K(-_1) \TTens D_M(-_2) 
\to D_{K \TTens M}(-_1\TTens -_2)$ defined above where $-_1=K$ and $-_2=M$.
\end{defi}

The proofs of the next two results are left to the reader. They are not difficult although they require large commutative diagrams.

\begin{prop}\label{ddmor}
The morphism $\dd_{K,M}$ is a morphism of suspended bifunctors.
\end{prop}
The proof of the following lemma uses the unit of the tensor product.
\begin{lemm} \label{KMdualKtensMdual}
If $K$ is dualizing, $M$ is dualizing and $\dd_{K,M}$ is an isomorphism, then $K \TTens M$ is dualizing.
If $\fh_{f,K}$, $\fh_{f,M}$ and $\dd_{K,M}$ are isomorphisms, then $\fh_{f,K \TTens M}$ is an isomorphism.
\end{lemm}

In \cite{Gille03}, Gille and Nenashev define two natural products for Witt groups. These products coincide up to a sign. We just choose one of them (the left product, for example), and refer to it as {\it the} product, but everything works fine with the other one too. Let us recall the basic properties of the product, rephrasing \cite{Gille03}
in our terminology.

\begin{theo}\label{GiNe} (\cite[Definition 1.11 and Theorem 2.9]{Gille03}) 
Let $\cC_1$, $\cC_2$ and $\cC_3$ be triangulated categories with dualities $D_1$, $D_2$ and $D_3$. Let $(B,b_1,b_2): \cC_1 \times \cC_2 \to \cC_3$ be a suspended bifunctor 
(see Definition \ref{defiSuspBif}) and $d: B (D_1^o \times D_2^o) \trafo D_3^o B^o$ be an isomorphism of suspended bifunctors (see Definition \ref{defiMorSuspBif}) that makes $\pair{B,d}$ a duality preserving functor (see Definition \ref{DualPresFunc_defi}, here $\cC_1 \times \cC_2$ is endowed with the duality $D_1 \times D_2$). Then $\pair{B,d}$ induces a product
$$\W(\cC_1) \times \W(\cC_2) \to \W(\cC_3).$$
\end{theo}

The following proposition is not stated in \cite{Gille03}, but easily follows from the construction of the product.

\begin{prop}
Let $\sigma:B \to B'$ be an isomorphism of suspended bifunctors that 
is duality preserving. Then $B$ and $B'$ induce the same product on Witt groups.
\end{prop}

Let us now apply this to our context. 

\begin{prop} (product) \label{existProduct}
The morphism of functors $\dd_{K,M}$ turns the functor $\TTens$ into a duality preserving functor $\pair{-_1 \TTens -_2, \dd_{K,M}}$ from $(\cC\times \cC,D_K\times D_M,\bid_K\times \bid_M)$ to $(\cC, D_{K\TTens M},\bid_{K \TTens M})$. 
\end{prop}

\begin{coro} (product for Witt groups)
By Theorem \ref{GiNe}, 
when the dualities and the functor are strong (\ie $\bid_K$, $\bid_M$ and $\dd_{K,M}$ are isomorphisms, and thus $\bid_{K \TTens M}$ as well by Lemma \ref{KMdualKtensMdual}), 
$\pair{-_1\TTens-_2, \dd_{K,M}}$ induces a product
$$\xymatrix{\W (\cC,D_K)\times \W(\cC,D_M) \ar[r]^-{.} & \W(\cC,D_{K \TTens M})}$$
on Witt groups.
\end{coro}
 
\section{Relations between functors}

In this section, we prove the main results on composition, base change, and the projection formula. 

\subsection{Composition} \label{Composition}

This section studies the behavior of pull-backs and push-for\-wards with respect 
to composition.
Let $\cK$ be a category whose objects are (suspended, triangulated) closed categories as in Section \ref{TensProHom}, and whose morphisms are (suspended, exact) functors. Let $\cB$ be another category, and let $(-)^*$ be a contravariant pseudo functor from $\cB$ to $\cK$, \ie a functor, except that instead of having equalities $f^* g^*= (g f)^*$ when $f$ and $g$ are composable, we only have isomorphisms of (suspended) functors $\ea_{g,f}:f^* g^* \trafo (g f)^*$. We also require that $(-)^*$ sends the identity of an object to the identity. We denote $X^*$ by $\cC_X$.
When, moreover, the diagram 
$$\xymatrix{
f^* g^* h^* \ar[d]_{\ea_{h,g}} \ar[r]^{\ea_{g,f}} \ar@{}[dr]|{\diagram \label{f^*pseudoFun}} & (g f)^* h^* \ar[d]^{\ea_{h,gf}} \\
f^* (h g)^* \ar[r]^{\ea_{hg,f}} & (h g f)^*
}$$
is commutative, we say that the pseudo functor is associative. 

\begin{rema}
An example for this setting 
is to take for $\cB$ the category of schemes
(or regular schemes) and $\cC_X=D^b(Vect(X))$ (or $\cC_X=D^b(O_X-Mod)$).
\end{rema}

We require that $(-)^*$ is a monoidal associative pseudo functor, which means that each $f^*$ is symmetric monoidal (Assumption \assumption{\ref{assufmonoidal}}{f}) and that the diagram
$$\xymatrix{
g^* f^* \TTens g^* f^* \ar[d]_{\ea \TTens \ea} \ar[r]^{\fp} \ar@{}[drr]|{\diagram \label{f^*tens}} & g^* (f^* \TTens f^*) \ar[r]^{\fp} & g^* f^*(- \TTens -) \ar[d]^{\ea} \\ 
(fg)^* \TTens (fg)^* \ar[rr]^{\fp} & & (fg)^* (- \TTens -)
}$$
commutes for any two composable $f$ and $g$.

\begin{nota}
Let $X$ be an object in $\cB$ and $K$ an object in $\cC_X$
We denote by $\cC_{X,K}$ the (suspended, triangulated) category with duality $(\cC_X,D_K,\bid_K)$ obtained by Corollary \ref{tripleCat_coro}. When $K$ is dualizing, we denote by $\W^i(X,K)$ the $i$-th shifted Witt group of $\cC_{X,K}$.
\end{nota}

We have already seen in Theorem \ref{PullBack0_theo} that under Assumption \assumption{\ref{assufmonoidal}}{f}, the couple $\pair{f^*,\fh_{K}}$ is a duality preserving functor between (suspended, triangulated) categories with dualities, under Assumption \assumption{\ref{assufmonoidal}}{f}. Recall the $I$-notation of Lemma \ref{isoKisoWitt}.
 
\begin{theo} (composition of pull-backs) \label{compof^*_theo}
For any two composable $f$ and $g$ in $\cB$, the isomorphism of functors $\ea_{g,f}: f^* g^* \trafo (gf)^*$ is a morphism of duality preserving functors from $I_{\ea_{g,f,K}}\pair{f^*,\fh_{f,g^*K}} \pair{g^*,\fh_{g,K}}$ to $\pair{(gf)^*,\fh_{gf,K}}$. 
\end{theo}
\begin{coro} (composition of pull-backs for Witt groups)
By Proposition \ref{samemorphismWitt_prop}, for composable morphisms $f$ and $g$ in $\cB$, the pull-back on Witt groups defined in Corollary \ref{PullBack0Witt_coro} satisfies $(gf)^*_W =I_{\ea_{g,f,K}}^W f^*_W g^*_W$.
\end{coro}

\begin{proof}[Proof of Theorem \ref{compof^*_theo}]
The claim amounts to checking that the diagram
$$\xymatrix@R=3ex{
f^* g^* \HHom{-}{K} \ar[d]_{\ea} \ar[r]^-{\fh} & f^* \HHom{g^*(-)}{g^* K} \ar[r]^-{\fh} & \HHom{f^* g^*(-)}{f^* g^* K} \ar[d]^{\ea} \\
(gf)^* \HHom{-}{K} \ar[r]^-{\fh} & \HHom{(gf)^*(-)}{(gf)^* K} \ar[r]^-{\ea^o} & \HHom{f^* g^* (-)}{(gf)^* K}
}$$
is commutative. Since by construction of $\fh$ (Proposition \ref{defifh}), $\fp$ and $\fh$ are mates, this commutativity follows from Lemma \ref{cube} applied to the cube 
$$\xymatrix@!C=2ex@!R=3ex{
\cC_Z \ar[rr]^{(gf)^*} \ar[dr]^{g^*} & & \cC_X \ar[dr]^{Id} & \\
 & \cC_Y \ar@2[ur]^{\ea} \ar[rr]^{f^*} & & \cC_X \\
 \cC_Z \ar[uu]^{- \TTens K} \ar[dr]_{g^*} & & & \\
 & \cC_Y \ar@2[luuu]_{\fp} \ar[uu]_(.3){- \TTens g^* K} \ar[rr]_{f^*} & & \cC_X \ar@2[lluu]_{\fp} \ar[uu]_{- \TTens f^* g^* K}
}
\hspace{-2ex}
\xymatrix@!C=2ex@!R=3ex{
 \cC_Z \ar[rr]^{(gf)^*} & & \cC_X \ar[dr]^{Id} & \\
 & & & \cC_X \\
 \cC_Z \ar[uu]^{- \TTens K} \ar[dr]_{g^*} \ar[rr]^{(gf)^*} & & \cC_X \ar@2[lluu]^{\fp} \ar[uu]^(.65){- \TTens (gf)^* K} \ar[dr]_{Id} & \\
 & \cC_Y \ar@2[ur]^{\ea} \ar[rr]_{f^*} & & \cC_X \ar@2[luuu]|{\rule[-.5ex]{0ex}{2ex} id \TTens \ea} \ar[uu]_{- \TTens f^* g^* K}
}$$
which is commutative as Diagram \diag{\ref{f^*tens}} is.
\end{proof}
A covariant pseudo functor $(-)_*$ is defined in the same way as a contravariant one, except that we are given isomorphisms in the other direction: $\eb_{g,f}: (gf)_* \to g_* f_*$. 
\begin{defi}
Let $(-)^*$ (resp. $(-)_*$) be a (suspended, exact) contravariant (resp. covariant) pseudo functor from $\cB$ to $\cK$. We say that $((-)^*,(-)_*)$ is an adjoint couple of pseudo functors if $(f^*,f_*)$ is a (suspended, exact) adjoint couple for every $f$ (in particular $(-)^*$ and $(-)_*$ coincide on objects), and the diagrams
$$\begin{array}{c}
\xymatrix{
Id \ar[d] \ar[r] \ar@{}[dr]|{\diagram \label{compAdj1}} & (gf)_* (gf)^* \ar[d]^{\eb} \\
g_*f_*f^* g^* \ar[r]^{\ea} & g_* f_* (gf)^* 
}
\end{array}\hfill
\hspace{5ex}
and
\hspace{5ex}\begin{array}{c} 
\xymatrix{
f^* g^* (gf)_* \ar[d]_{\ea} \ar[r]^{\eb} \ar@{}[dr]|{\diagram \label{compAdj2}} & f^*g^* g_*f_* \ar[d] \\
(gf)^* (gf)_* \ar[r] & Id
}
\end{array}$$
commute for any composable $f$ and $g$. 
\end{defi}
Spelling out the symmetric notion when the left adjoint pseudofunctor is covariant and the right one contravariant is left to the reader. 

As usual, the right (or left) adjoint is unique up to unique isomorphism.

\begin{lemm} \label{adjWeak}
Let $(-)^*$ be a contravariant pseudo-functor. Assume that for any $f^*$, we are given a right (suspended, exact) adjoint $f_*$ which is the identity when $f$ is the identity.
Then there is a unique collection of isomorphisms $\eb_{g,f}:(g f)_* \trafo g_* f_*$ such that 
$((-)^*,(-)_*)$ forms an adjoint couple of pseudo functors.
\end{lemm}
\begin{proof}
Apply Theorem \ref{adjsquare} (or \ref{adjsquareT} in the suspended case) to $(L,R)=(f^*,f_*)$, $(F_1,G_1)=(g^*,g_*)$, $(F_2,G_2)=(Id,Id)$ and $(L',R')=((gf)^*,(gf)_*)$. This gives the required isomorphism $(gf)_* \trafo g_*f_*$ and the diagrams $\diag{\ref{compAdj1}}$ and $\diag{\ref{compAdj2}}$. 
\end{proof}

\begin{lemm} \label{adjWeakAssoc}
The right (resp. left) adjoint of an associative pseudo functor is associative.
\end{lemm}
\begin{proof}
Left to the reader.
\end{proof}

\begin{lemm} \label{adjWeakMon_lemm}
Let $((-)^*,(-)_*)$ be an adjoint pair of pseudo-functors (left contravariant and right covariant). If $(-)^*$ satisfies the commutativity of \diag{\ref{f^*tens}}, then the diagram
$$\xymatrix@R=3ex{
(gf)_* A \TTens (gf)_* B \ar[d]_{\eb \TTens \eb} \ar[rr]^-{\fg} \ar@{}[drr]|{\diagram \label{f_*tens}} & & (gf)_* (A \TTens B) \ar[d]^{\eb} \\ 
g_* f_* A \TTens g_* f_* B \ar[r]^{\fg} & g_* (f_* A \TTens f_*B) \ar[r]^{\fg} & g_* f_* (A \TTens B)
}$$
is commutative for every $A$ and $B$.
\end{lemm}
\begin{proof}
The commutative diagram
$$\xymatrix@R=3ex@C=3ex{
g^*(A \TTens B) \ar[r]^-{\fp^{-1}} \ar[d]_{\unit_*^*} \ar@{}[dr]|{\diag{mf}} & g^* A \TTens g^* B \ar[r]^-{\unit_*^* \TTens \unit_*^*} \ar[d] \ar@{}[dr]|{\diag{\ref{Happ0}}} & f_* f^* g^* A \TTens f_* f^* g^* B \ar[d]^{\fg} \ar@<3ex>@/^15ex/[dd]|{f_* \ea \TTens f_* \ea} \\
f_* f^* g^* (A \TTens B) \ar[r] \ar[d]_{f_* \ea} \ar@{}[drr]|{\diag{\ref{f^*tens}}} & f_* f^* (g^* A \TTens g^* B) \ar[r] & f_* (f^* g^* A \TTens f^* g^* B) \ar[dl] \ar@{}[d]|{\diag{mf}} \\
f_* (gf)^* (A \TTens B) \ar[r]^-{\fp^{-1}} & f_* ((gf)^* A \TTens (gf)^* B) & f_* (gf)^* A \TTens f_* (gf)^* B \ar[l]_-{\fg} \\
}$$
shows that the cube
$$\xymatrix@!C=2ex@!R=3ex{
\cC_X \times \cC_X \ar[rr]^{- \TTens -}  \ar[dr]^*-{\labelstyle f_* \times f_*} & & \cC_X \ar[dr]^{f_*} \\
 & \cC_Y \times \cC_Y \ar@2[ur]_(.6){\fg} \ar[rr]^{- \TTens -} & & \cC_Y \\
 \cC_Z \times \cC_Z \ar[uu]^{(gf)^* \times (gf)^*} \ar[dr]_{Id \times Id} & & & \\
 & \cC_Z \times \cC_Z \ar@2[luuu]|{\stackrel{\labelstyle (\ea \times \ea)\circ}{(\unit_*^* \times \unit^*_*)}} \ar[uu]_(.4){g^* \times g^*} \ar[rr]_{- \TTens -} & & \cC_Z \ar@2[lluu]_{\fp^{-1}} \ar[uu]_{g^*}
}
\xymatrix@!C=2ex@!R=3ex{
\cC_X \times \cC_X \ar[rr]^{- \TTens -} & & \cC_X \ar[dr]^{f_*} & \\
 & & & \cC_Y \\
\cC_Z \times \cC_Z \ar[uu]^{(gf)^* \times (gf)^*} \ar[dr]_{Id \times Id} \ar[rr]^{- \TTens -} & & \cC_Z \ar@2[lluu]^{\fp^{-1}} \ar[uu]^{(gf)^*} \ar[dr]_{Id} & \\
 & \cC_Z \times \cC_Z \ar@2[ur]^{id} \ar[rr]_{- \TTens -} & & \cC_Z \ar@2[luuu]_{\ea} \ar[uu]_{g^*}
}$$
is commutative. We apply Lemma \ref{cube} to it and obtain a new commutative cube which is Diagram \diag{\ref{f_*tens}}. Note that the morphism of functors on the left face of our first cube indeed gives $\eb$ in the resulting cube, by construction of $\eb$.
\end{proof}

Let us now consider the subcategory $\cB'$ of $\cB$ with the same objects, but whose morphisms are only those $f:X \to Y$ satisfying assumptions \assumption{\ref{assufmonoidal}}{f}, \assumption{\ref{assuAdjf*f*}}{f} and \assumption{\ref{assuAdjf*f!}}{f}.
We then choose successive right adjoints $f_*$ and $f^!$ for each morphism $f$ (they are unique up to unique isomorphism, and we choose $(Id_X)_*=Id_X^!=Id_{\cC_X}$ for simplicity), and by Lemma \ref{adjWeak}, using $(-)^*$, we turn them into (suspended, exact) pseudo functors
$$(-)_*: \cB' \to \cK \hspace{10ex} (-)^!: \cB' \to \cK$$ 
with structure morphisms $\eb_{g,f}$ and $\ec_{g,f}$.
They are associative by Lemma \ref{adjWeakAssoc}. 

\begin{theo} \label{compof_*0_theo}
For morphisms $f:X \to Y$ and $g: Y \to Z$ in $\cB'$ and $K$ an object of $\cC_Z$, the isomorphism of functors $\eb_{g,f}: (gf)_* \to g_* f_*$ defined above is a morphism of duality preserving functors from $\pair{(gf)_*,\rr_K} I_{\ec_{g,f,K}}$ to $\pair{g_*,\rr_K} \pair{f_*,\rr_{g^!K}}$. 
\end{theo}

\begin{coro} (composition of push-forwards) \label{compof_*0Witt_coro}
By Proposition \ref{samemorphismWitt_prop}, for composable morphisms $f$ and $g$ in $\cB'$, the push-forward on Witt groups defined in Corollary \ref{PushForward0Witt_coro} satisfies $(gf)^W_* I_{\ec_{g,f,K}}^W=g^W_* f^W_*$.
\end{coro}

\begin{proof}[Proof of Theorem \ref{compof_*0_theo}]
We have to prove that the diagram \diag{M} of Definition \ref{MorphDualPresFunc_defi} which is here
$$\xymatrix{
(gf)_*\HHom{A}{f^!g^! K} \ar[r]^-{\ec} \ar[d]_{\eb} \ar@{}[drr]|{\diagram \label{compo_diag}} & (gf)_* \HHom{A}{(gf)^!K} \ar[r]^-{\rr} & \HHom{(gf)_*A}{K} \\
g_* f_* \HHom{A}{f^! g^! K} \ar[r]^{\rr} & g_* \HHom{f_* A}{g^! K} \ar[r]^{\rr} & \HHom{g_* f_* A}{K} \ar[u]_{\eb^o}
}$$
is commutative (for any $A$ and $B$).
The commutative diagram \diag{\ref{f_*tens}} may be written as the commutative cube
$$\xymatrix@!C=2ex@!R=3ex{
\cC_Z \ar[rr]^{Id} \ar[dr]^{Id} & & \cC_Z \ar[dr]^{Id} & \\
 & \cC_Z \ar@2[ur]^{id} \ar[rr]^{Id} & & \cC_Z \\
\cC_X \ar[uu]^{g_*f_*(-\TTens B)} \ar[dr]_{Id} & & & \\
 & \cC_X \ar@2[luuu]^{\eb} \ar[uu]_(.25)*-{\labelstyle (gf)_*(-\TTens B)} \ar[rr]_{(gf)_*} & & \cC_Z \ar@2[lluu]_{\fg} \ar[uu]|{-\TTens (gf)_*B}
}
\xymatrix@!C=2ex@!R=3ex{
\cC_Z \ar[rr]^{Id} & & \cC_Z \ar[dr]^{Id} & \\
 & & & \cC_Z \\
\cC_X \ar[uu]|{g_*f_*(-\TTens B)} \ar[dr]_{Id} \ar[rr]^{f_*} & & \cC_Y \ar@2[lluu]^{\fg} \ar[uu]^(.75)*-{\labelstyle g_*(- \TTens f^* B)} \ar[dr]^{g_*} & \\
 & \cC_X \ar@2[ur]^{\eb} \ar[rr]_{(gf)_*} & & \cC_Z \ar@2[luuu]|{\fg \circ (id \TTens \eb)} \ar[uu]_{- \TTens (gf)_* B}
}$$
out of which we get a new commutative cube by Lemma \ref{cube}, whose commutativity is equivalent to that of Diagram \diag{\ref{compo_diag}}.
\end{proof}

\subsection{Base change} \label{BaseChange}

The next fundamental theorem that we will prove is a base change formula. In this section, we fix a commutative diagram in $\cB$ with $g$ and $\bar{g}$ in $\cB'$.
$$\xymatrix{
V \ar[d]_{\bar{f}} \ar[r]^{\bar{g}} \ar@{}[dr]|{\diagram \label{cart_diag}} & Y \ar[d]^{f} \\
X \ar[r]_{g} & Z
}$$
Using Lemma \ref{adjab} (or its suspended version \ref{adjsquareT}) with $J_1=g^*$, $K_1=g_*$, $J_2=\bar{g}^*$, $K_2=\bar{g}_*$, $H=\bar{f}^*$, $H'=f^*$ and  the isomorphism of (suspended) functors
$$\xi:\bar{g}^* f^* \to (f \bar{g})^* =(g \bar{f})^* \to \bar{f}^* g^*$$
for $a$, we obtain its mate, a morphism
$$\eps: f^* g_* \to \bar{g}_* \bar{f}^*$$
and two commutative diagrams.
Assuming \assumption{\ref{epsIso}}{f,g} (\ie $\eps$ is an isomorphism)
and applying the same lemma to $J_1=g_*$, $K_1=g^!$, $J_2=\bar{g}_*$, $K_2=\bar{g}^!$, $H=f^*$, $H'=\bar{f}^*$ and $a=\eps^{-1}$, we obtain a morphism
$$\gam: \bar{f}^* g^! \to \bar{g}^! f^*$$
and two commutative diagrams.
$$\xymatrix{
\bar{g}_* \bar{f}^* g^! \ar[r]^-{\gam} \ar[d]_{\eps^{-1}} \ar@{}[dr]|{\diagram \label{Happ3}} & \bar{g}_* \bar{g}^! f^* \ar[d] \\
f^* g_* g^! \ar[r] & f^*
}
\hspace{10ex}
\xymatrix{
\bar{f}^* \ar[r] \ar[d] \ar@{}[dr]|{\diagram \label{H'app3}} & \bar{f}^* g^! g_* \ar[d]^{\gam} \\
\bar{g}^! \bar{g}_* \bar{f}^* \ar[r]^-{\eps^{-1}} & \bar{g}^! f^* g_*
}
$$

\begin{theo}(base change) \label{basechange0_theo}
Let $f$ and $\bar{f}$ be morphisms in $\cB$ and $g$ and $\bar{g}$ be morphisms in $\cB'$ fitting in the commutative diagram \diag{\ref{cart_diag}} such that Assumption \assumption{\ref{epsIso}}{f,g} is satisfied. Then, the isomorphism of functors $\eps: f^* g_* \to \bar{g}_* \bar{f}^*$ from $\pair{f^*,\fh_K} \pair{g_*,\rr_K}$ to $\pair{\bar{g}_*,\rr_{f^*K}} I_{\gam_K} \pair{\bar{f}^*,\fh_{g^!K}}$ is duality preserving. 
\end{theo}

This together with Proposition \ref{samemorphismWitt_prop} 
immediately implie sthe following.

\begin{coro} 
If in the situation of the theorem above $\gam_K$ is an isomorphism, then the pull-backs and push-forwards on Witt groups defined in corollaries \ref{PullBack0Witt_coro} and \ref{PushForward0Witt_coro} satisfy $f^*_W g_*^W = \bar{g}_*^W I_{\gam_K}^W \bar{f}^*_W$.
\end{coro}

\begin{proof}[Proof of Theorem \ref{basechange0_theo}]
We have to prove that the diagram
$$\xymatrix@R=3ex{
f^*g_* \HHom{A}{g^! K} \ar[r]^-{\rr} \ar[d]_{\eps} \ar@{}[drrr]|{\diagram \label{basechange_diag}}& f^* \HHom{g_*A}{K} \ar[rr]^-{\fh} & & \HHom{f^* g_* A}{f^*K} \\
\bar{g}_* \bar{f}^* \HHom{A}{g^! K} \ar[r]^-{\fh} & \bar{g}_* \HHom{\bar{f}^* A}{\bar{f}^*g^!K} \ar[r]^-{\gam} & \bar{g}_* \HHom{\bar{f}^* A}{\bar{g}^! f^* K} \ar[r]^-{\rr} & \HHom{\bar{g}_* \bar{f}^* A}{f^* K} \ar[u]_{\eps^o}
}$$
(corresponding to Diagram \diag{M} of Definition \ref{MorphDualPresFunc_defi}) is commutative.
We first prove the two following lemmas.
\begin{lemm} \label{mates1_lemm}
The composition
$$\xymatrix{\bar{g}_*(\bar{f}^*(-) \TTens \bar{f}^*(*)) \ar[r]^-{\fp} & \bar{g}_* \bar{f}^* (- \TTens *) \ar[r]^-{\eps^{-1}} & f^* g_* (- \TTens *)}$$
and the composition
$$\xymatrix{\bar{f}^*\HHom{*}{g^!(-)} \ar[r]^-{\fh} & \HHom{\bar{f}^*(*)}{\bar{f}^*\bar{g}^!(-)} \ar[r]^-{\gam} & \HHom{\bar{f}^*(*)}{\bar{g}^!f^*(-)}}$$
are mates in Lemma \ref{adjabbif} when $J_1=g_*(-\TTens*)$, $K_1=\HHom{*}{g^!(-)}$, $J_2=\bar{g}_*(- \TTens \bar{f}^*(*))$, $K_2=\HHom{\bar{f}^*(*)}{\bar{g}^!(-)}$, $H=f^*$ and $H'=\bar{f}^*$.
\end{lemm}
\begin{proof}
By uniqueness, this follows from the proof of one of the commutative diagrams in Lemma \ref{adjabbif}. We choose \diag{H} which is as follows.
$$\xymatrix@R=3ex@C=3ex{
\bar{g}_*(\bar{f}^*\HHom{X}{g^!A} \TTens \bar{f}^* X) \ar[r]^-{\fh} \ar[d]_{\fp} \ar@{}[dr]|{\diag{\ref{L'app1}}} & \bar{g}_*(\HHom{\bar{f}^*X}{\bar{f}^*g^!A}\TTens \bar{f}^* X) \ar[d] \ar[r]^-{\gam} \ar@{}[dr]|{\diag{mf}} & \bar{g}_*(\HHom{\bar{f}^*X}{\bar{g}^!f^*A}\TTens \bar{f}^*X) \ar[d] \\
\bar{g}_*\bar{f}^*(\HHom{X}{g^!A} \TTens X) \ar[r] \ar[d]_{\eps^{-1}} \ar@{}[dr]|{\diag{mf}} & \bar{g}_* \bar{f}^* g^! A \ar[r]^-{\gam} \ar[d]_{\eps^{-1}} \ar@{}[dr]|{\diag{\ref{Happ3}}} & \bar{g}_*g^!f^*A \ar[d] \\
f^* g_* (\HHom{X}{g^!A} \TTens X) \ar[r] & f^* g_* g^! A \ar[r] & f^* A
}$$
\end{proof}
\begin{lemm} \label{mates2_lemm}
The composition
$$\xymatrix{\bar{g}_*(-)\TTens f^* g_*(*) \ar[r]^-{id \TTens \eps} & \bar{g}_*(-) \TTens \bar{g}_* \bar{f}^* (*) \ar[r]^-{\fg} & \bar{g}_* (- \TTens \bar{f}^*(*))}$$
and the composition
$$\xymatrix{\bar{g}_* \HHom{\bar{f}^*(*)}{\bar{g}^!(-)} \ar[r]^-{\rr} & \HHom{\bar{g}_*\bar{f}^*(*)}{-} \ar[r]^-{\eps^o} & \HHom{f^*g_*(*)}{-}}$$
are mates in Lemma \ref{adjabbif} when $J_1=\bar{g}_*(-\TTens \bar{f}^*(*))$, $K_1=\HHom{\bar{f}^*(*)}{\bar{g}^!(-)}$, $J_2=(-\TTens f^*g_*(*))$, $K_2=\HHom{f^*g_*(*)}{-}$, $H=Id$ and $H'=\bar{g}_*$. 
\end{lemm}
\begin{proof}
Similar to the proof of the previous Lemma, this follows from the commutative diagram
$$\xymatrix@R=3ex@C=3ex{
\bar{g}_* \HHom{\bar{f}^*X}{\bar{g}^!A} \TTens f^* g_* X \ar[d]_{id \TTens \eps} \ar[r]^-{\rr \TTens id} \ar@{}[dr]|{\diag{mf}} & \HHom{\bar{g}_* \bar{f}^* X}{A} \TTens f^* g_* X \ar[r]^-{\eps^o} \ar[d] \ar@{}[ddr]^{\diag{gen}} & \HHom{f^* g_* X }{A} \TTens f^* g_* X \ar[dd] \\
\bar{g}_* \HHom{\bar{f}^* X}{\bar{g}^! A} \TTens \bar{g}_* \bar{f}^* X \ar[d]_{\fg} \ar[r]_-{\rr \TTens id} & \HHom{\bar{g}_* \bar{f}^* X}{A} \TTens \bar{g}_* \bar{f}^* X \ar[dr]  \ar@{}[d]|{\diag{\ref{H'app1}''}} & \\
\bar{g}_* (\HHom{\bar{f}^* X}{\bar{g}^! A} \TTens \bar{f}^* X) \ar[r] & \bar{g}_* \bar{g}^! A \ar[r] & A
}$$
where \diag{\ref{H'app1}''} is easily obtained from \diag{\ref{H'app1}'} by looking at the definition of $\rr$ in Theorem \ref{PushForward0_theo}.
\end{proof}
The commutative diagram \diag{\ref{f^*tens}} easily implies (by definition of $\xi$) that the diagram
$$\xymatrix@R=3ex{
\bar{g}^*(f^* A \TTens f^* B) \ar[d]_{\fp} \ar[r]^-{\fp^{-1}} & \bar{g}^* f^* A \TTens \bar{g}^* f^* B \ar[r]^-{\xi \TTens \xi} & \bar{f}^* g^* A \TTens \bar{f}^* g^* B \ar[d]^{\fp} \\
\bar{g}^* f^* (A \TTens B) \ar[r]^-{\xi} & \bar{f}^* g^* (A \TTens B) \ar[r]^{\fp^{-1}} & \bar{f}^* (g^* A \TTens g^* B)
}$$
is commutative for any $A$ and $B$. It can be written as the commutative cube
$$\xymatrix@!C=2ex@!R=3ex{
\cC_X \times \cC_X \ar[rr]^{- \TTens -}  \ar[dr]^*-{\labelstyle \bar{f}^* \times \bar{f}^*} & & \cC_X \ar[dr]^{\bar{f}^*} \\
 & \cC_V \times \cC_V \ar@2[ur]_(.6){\fp} \ar[rr]^{- \TTens -} & & \cC_V \\
 \cC_Z \times \cC_Z \ar[uu]^{g^* \times g^*} \ar[dr]_{f^* \times f^*} & & & \\
 & \cC_Y \times \cC_Y \ar@2[luuu]|{\xi \times \xi} \ar[uu]_(.4){\bar{g}^* \times \bar{g}^*} \ar[rr]_{- \TTens -} & & \cC_Y \ar@2[lluu]_{\fp^{-1}} \ar[uu]_{\bar{g}^*}
}
\hspace{5ex}
\xymatrix@!C=2ex@!R=3ex{
\cC_X \times \cC_X \ar[rr]^{- \TTens -} & & \cC_X \ar[dr]^{\bar{f}^*} & \\
 & & & \cC_V \\
\cC_Z \times \cC_Z \ar[uu]^{g^* \times g^*} \ar[dr]_{f^* \times f^*} \ar[rr]^{- \TTens -} & & \cC_Z \ar@2[lluu]^{\fp^{-1}} \ar[uu]^{g^*} \ar[dr]_{f^*} & \\
 & \cC_Y \times \cC_Y \ar@2[ur]^{\fp} \ar[rr]_{- \TTens -} & & \cC_Y \ar@2[luuu]_{\xi} \ar[uu]_{\bar{g}^*}
}$$
out of which Lemma \ref{cube} gives the commutative cube
$$\xymatrix@!C=2ex@!R=3ex{
\cC_X \times \cC_X \ar[dd]_{g_* \times g_*} \ar[rr]^-{- \TTens -}  \ar[dr]^*-{\labelstyle \bar{f}^* \times \bar{f}^*} & & \cC_X \ar[dr]^{\bar{f}^*} \\
 & \cC_V \times \cC_V \ar[dd]^(.4){\bar{g}_* \times \bar{g}_*} \ar@2[ur]_(.6){\fp} \ar[rr]^-{- \TTens -} & & \cC_V \ar[dd]^{\bar{g}_*} \\
 \cC_Z \times \cC_Z \ar[dr]_{f^* \times f^*} \ar@2[ru]|{\rule[-.5ex]{0ex}{2ex} \eps \times \eps} & & & \\
 & \cC_Y \times \cC_Y \ar[rr]_-{- \TTens -} \ar@2[rruu]_{\fg} & & \cC_Y 
}
\hspace{5ex}
\xymatrix@!C=2ex@!R=3ex{
\cC_X \times \cC_X \ar[dd]_{g_* \times g_*} \ar[rr]^-{- \TTens -} & & \cC_X \ar[dr]^{\bar{f}^*} \ar[dd]^{g_*} & \\
 & & & \cC_V \ar[dd]^{\bar{g}_*} \\
\cC_Z \times \cC_Z \ar@2[rruu]^{\fg} \ar[dr]_{f^* \times f^*} \ar[rr]^-{- \TTens -} & & \cC_Z \ar@2[ru]_{\eps} \ar[dr]_{f^*} & \\
 & \cC_Y \times \cC_Y \ar@2[ur]^{\fp} \ar[rr]_-{- \TTens -} & & \cC_Y 
}$$
The commutativity of the last cube implies the commutativity of the following
cube for any $A$. 
$$\xymatrix@!C=2ex@!R=3ex{
\cC_Z \ar[rr]^{f^*} \ar[dr]^{Id} & & \cC_Y \ar[dr]^{Id} & \\
 & \cC_Z \ar@2[ur]^{id} \ar[rr]^{f^*} & & \cC_Y \\
\cC_X \ar[uu]^{g_*(- \TTens A)} \ar[dr]_{g_*} & & & \\
 & \cC_Z \ar@2[luuu]^{\fg} \ar[uu]_(.3){- \TTens g_* A} \ar[rr]_{f^*} & & \cC_Y \ar@2[lluu]_{\fp} \ar[uu]|{- \TTens f^* g_* A}
}
\xymatrix@!C=2ex@!R=3ex{
 \cC_Z \ar[rr]^{f^*} & & \cC_Y \ar[dr]^{Id} & \\
 & & & \cC_Y \\
 \cC_X \ar[uu]|{g_*(- \TTens A)} \ar[dr]_{g_*} \ar[rr]^{\bar{f}^*} & & \cC_V \ar@2[lluu]^{x} \ar[uu]^(.75)*-{\labelstyle \bar{g}_*(- \TTens \bar{f}^* A)} \ar[dr]_{\bar{g}_*} & \\
 & \cC_Z \ar@2[ur]^{\eps} \ar[rr]_{f^*} & & \cC_Y \ar@2[luuu]_{y} \ar[uu]_{- \TTens f^* g_* A}
}$$
where $x$ is the first composition in Lemma \ref{mates1_lemm} and $y$ is the first composition in Lemma \ref{mates2_lemm}. Applying Lemma \ref{cube} to this new cube gives a cube in which the mates of $x$ and $y$ may be described using lemmas \ref{mates1_lemm} and \ref{mates2_lemm}, and whose commutativity is the one of Diagram 
\diag{\ref{basechange_diag}}.
\end{proof}

Let us now establish two more commutative diagrams that will be useful in applications.
In geometric situations they allow us to check whether $\ssp$ is an isomorphism by restricting to open subsets. 
\begin{prop} \label{moreCommDiag}
The diagram
$$\xymatrix@R=3ex@C=10ex{
f^* g_*(-) \TTens f^*(*) \ar[dd]_{\fp} \ar[r]^{\eps} \ar@{}[ddrr]|{\diagram \label{epsqp}} & \bar{g}_* \bar{f}^*(-) \TTens f^*(*) \ar[r]^-{\q} & \bar{g}_* (\bar{f}^*(-) \TTens \bar{g}^* f^*(*)) \ar[d]^{\xi} \\ 
 & & \bar{g}_* (\bar{f}^*(-) \TTens \bar{f}^* g^* (*)) \ar[d]^{\fp} \\
f^* (g_*(-) \TTens *) \ar[r]^-{\q} & f^* g_* (- \TTens g^*(*)) \ar[r]^-{\eps} & \bar{g}_* \bar{f}^* (- \TTens g^*(*))
}$$
is commutative. If furthermore $\eps$ is an isomorphism (Assumption \assumption{\ref{epsIso}}{f,g}), $\gam$ is defined and the diagram
$$\xymatrix@R=3ex@C=10ex{
\bar{f}^*g^!(-) \TTens \bar{f}^*g^*(*) \ar[d]_{\fp} \ar[r]^-{\gam \TTens \xi^{-1}} \ar@{}[drr]|{\diagram \label{spgamm}} & \bar{g}^! f^*(-) \TTens \bar{g}^* f^*(*) \ar[r]^-{\ssp} & \bar{g}^! (f^*(-) \TTens f^*(*)) \ar[d]^{\fp} \\ 
\bar{f}^* (g^!(-) \TTens g^*(*)) \ar[r]^-{\ssp} & \bar{f}^* g^! (- \TTens *) \ar[r]^-{\gam} & \bar{g}^! f^* (- \TTens *)
}$$
is commutative.
\end{prop}
\begin{proof}
To get the first diagram, apply Lemma \ref{cube} to the cube
$$\xymatrix@!C=3ex@!R=2ex{
\cC_X \times \cC_Z \ar[rr]^{\bar{f}^* \times Id} \ar[dr]^*-{\labelstyle Id \TTens g^*} & & \cC_V \times \cC_Z \ar[dr]^*-{\labelstyle Id \TTens \bar{g}^*f^*} & \\
 & \cC_X \ar@2[ur]_*-{\labelstyle \fp^{-1}} \ar[rr]_{\bar{f}^*} & & \cC_V \\
\cC_Z \times \cC_Z \ar[uu]^{g^* \times Id} \ar[dr]_{\TTens} & & & \\
 & \cC_Z \ar@2[luuu]|{\rule[-.5ex]{0ex}{2ex} \fp^{-1}} \ar[uu]_{g^*} \ar[rr]_{f^*} & & \cC_Y \ar@2[lluu]^{\xi} \ar[uu]^{\bar{g}^*}
}
\hspace{5ex}
\xymatrix@!C=3ex@!R=2ex{
\cC_X \times \cC_Z \ar[rr]^{\bar{f}^* \times Id} & & \cC_V \times \cC_Z \ar[dr]^*-{\labelstyle Id \TTens \bar{g}^*f^*} & \\
 & & & \cC_V \\
\cC_Z \times \cC_Z \ar[uu]^{g^*\times Id} \ar[dr]_*-{\labelstyle \TTens} \ar[rr]^{f^* \times Id} & & \cC_Y \times \cC_Z \ar@2[lluu]^{\xi \times id} \ar[uu]^(.6)*-{\labelstyle \bar{g}^* \times Id} \ar[dr]_*-{\labelstyle Id \TTens f^*} & \\
 & \cC_Z \ar@2[ur]^*-{\labelstyle \fp^{-1}} \ar[rr]_{f^*} & & \cC_Y \ar@2[luuu] \ar[uu]_{\bar{g}^*}
}$$
which is commutative by $\diag{\ref{f^*tens}}$ (using $f \bar{g}=g\bar{f}$). This gives a cube 
whose commutativity is the one of
$\diag{\ref{epsqp}}$. Exchanging the top and bottom faces of this new cube (thus reversing the vertical arrows), inverting the 
front ($\eps$), back ($\eps \times Id$) 
and sides ($\q$) morphisms of functors and applying again Lemma \ref{cube} to this new cube 
shows that $\diag{\ref{spgamm}}$ is commutative too. 
\end{proof}

\subsection{Associativity of products} \label{assocProducts}

We now establish a few commutative diagrams implied by the compatibility of $f^*$ with the associativity of the tensor product. For simplicity, we hide in diagrams all associativity morphisms and bracketing concerning the tensor product (as if the tensor product was strictly associative).

Since $f^*$ is monoidal, the diagram (involving $\fp$)
$$\xymatrix@R=3ex{
f^* A \TTens f^* B \TTens f^* C \ar[r] \ar[d] \ar@{}[dr]|{\diagram \label{f^*assoc}} & f^*(A \TTens B) \TTens f^*C \ar[d] \\
f^* A \TTens f^* (B \TTens C) \ar[r] & f^* (A \TTens B \TTens C)
}$$
is commutative.

\begin{lemm}
Under assumptions \assumption{\ref{assufmonoidal}}{f} and \assumption{\ref{assuAdjf*f*}}{f}, the diagram (involving $\fg$)
$$\xymatrix@R=3ex{
f_* A \TTens f_* B \TTens f_* C \ar[r] \ar[d] \ar@{}[dr]|{\diagram \label{f_*assoc_diag}} & f_*(A \TTens B) \TTens f_*C \ar[d] \\
f_* A \TTens f_* (B \TTens C) \ar[r] & f_* (A \TTens B \TTens C)
}$$
is commutative 
\end{lemm}
\begin{proof}
Left to the reader.
\end{proof}

\begin{prop}
Under assumptions \assumption{\ref{assufmonoidal}}{f} and \assumption{\ref{assuAdjf*f*}}{f}, the following diagrams are commutative.
$$\xymatrix@R=4ex@C=3ex{
f_* A \TTens B \TTens C \ar[d]_-{\q} \ar[r]_{\q} \ar@{}[dr]|{\diagram \label{projMorphAssoc_diag}} & f_* (A \TTens f^* B) \TTens C \ar[d]^{\q} \\
f_* (A \TTens f^* (B \TTens C)) \ar[r]^{\fp^{-1}} & f_* (A \TTens f^* B \TTens f^* C) 
}$$
$$\xymatrix@R=4ex@C=3ex{
f_* A \TTens f_* B \TTens C \ar[d]_-{id \TTens \q} \ar[r]_{\fg \TTens id} \ar@{}[dr]|{\diagram \label{projMorphAssoc2_diag}} & f_* (A \TTens B) \TTens C \ar[d]^{\q} \\
f_* (A \TTens f_* (B \TTens f^*C)) \ar[r]^{\fg} & f_* (A \TTens B \TTens f^* C) 
}
$$
\end{prop}
\begin{proof}
Using Lemma \ref{fgq} to decompose $\q$, Diagram \diag{\ref{projMorphAssoc_diag}} is
$$\xymatrix@R=3ex@C=2ex{
f_* A \TTens B \TTens C \ar[r]^-{\unit_*^*} \ar[dd]_{\unit_*^*} \ar@{}[ddr]_(.35){\labelstyle \diag{\ref{Happ0}}} & f_* A \TTens f_*f^* B \TTens C \ar[d]_{\unit_*^*} \ar[r]^-{\fg} \ar@{}[dr]|{\diag{mf}} & f_* (A \TTens f^* B) \TTens C \ar[d]^{\unit_*^*} \\
 & f_* A \TTens f_* f^* B \TTens f_* f^* C \ar[r] \ar[d]_{\fg} \ar@{}[dr]|{\diag{\ref{f_*assoc_diag}}} & f_* (A \TTens f^* B) \TTens f_* f^* C \ar[d]^{\fg} \\
f_* A \TTens f_* f^* (B \TTens C) \ar[r]^-{\fp^{-1}} \ar[dr]_{\fg} & f_* A \TTens f_* (f^* B \TTens f^* C) \ar[r]_-{\fg} \ar@{}[d]|{\diag{mf}} & f_* (A \TTens f^* B \TTens f^* C) \\
 & f_* (A \TTens f^* (B \TTens C)) \ar[ur]_{\fp^{-1}} & 
}$$
and Diagram \diag{\ref{projMorphAssoc2_diag}} is
$$\xymatrix@R=3ex{
f_* A \TTens f_* B \TTens C \ar[r]^-{\unit^*_*} \ar[d]_{\fg} \ar@{}[dr]|{\diag{mf}} & f_* A \TTens f_* B \TTens f_* f^* C \ar[d]^{\fg} \ar[r]^-{\fg} \ar@{}[dr]|{\diag{\ref{f_*assoc_diag}}} & f_* A \TTens f_* (B \TTens f^* C) \ar[d]^{\fg} \\
f_* (A \TTens B) \TTens C \ar[r]^{\unit^*_*} & f_* (A \TTens B) \TTens f_* f^* C \ar[r]^{\fg} & f_* (A \TTens B \TTens f^* C)
}$$
\end{proof}

\subsection{The monoidal functor $f^*$ and products}

We have the following diagram of duality preserving functors
$$\xymatrix@R=3ex{
(\cC_1)_K \times (\cC_1)_M \ar[d]_{\pair{f^*\times f^* , \fh_K \times \fh_M}} \ar[rr]^{\pair{-\TTens *, \dd_{K,M}}} & & (\cC_1)_{K \TTens M} \ar[d]^{\pair{f^*,\fh_{K\TTens M}}} \\
(\cC_2)_{f^*K} \times (\cC_2)_{f^*M} \ar@<1ex>@{}[r]^{\pair{-\TTens *, \dd_{f^* K, f^*M}}} \ar[r] & (\cC_2)_{f^*K \TTens f^*M} \ar[r]^{I_{\fp_{K,M}}} & (\cC_2)_{f^*(K \TTens M)}
}$$
where $I_{\fp}:(\cC_2)_{f^*K \TTens f^*M} \to (\cC_2)_{f^*(K \TTens M)}$ is
the duality preserving functor induced by $\fp$ by Lemma \ref{isoKisoWitt}).

\begin{prop} \label{fpDualPres} (the pull-back respects the product)
Under Assumption \assumption{\ref{assufmonoidal}}{f}, 
the isomorphism of suspended bifunctors $\fp: f^*(-) \TTens f^*(*) \to f^*(-\TTens*)$ is an isomorphism of duality preserving functors between the two duality preserving functors defined by the compositions above. 
\end{prop}

\begin{coro}
The pull-back and the product defined in Corollary \ref{PullBack0Witt_coro} and Proposition \ref{existProduct} satisfy
$$f^*_W(x.y)=I_{\fp_{K,M}}^W(f^*_W(x).f^*_W(y))$$ 
for all $x \in \W((\cC_1)_K)$ and $y \in \W((\cC_1)_M)$
by Proposition \ref{samemorphismWitt_prop} provided all necessary 
assumptions about being strong are satisfied.
\end{coro}

\begin{proof}[Proof of Proposition \ref{fpDualPres}]
We start with a lemma.
\begin{lemm} \label{mates3_lemm}
The composition
$$\xymatrix{
f^*(-) \TTens (f^* A \TTens f^*B) \ar[r]^-{id \TTens \fp} & f^*(-) \TTens f^*(A \TTens B) \ar[r]^-{\fp} & f^*(- \TTens (A \TTens B))
}$$
and the composition
$$\xymatrix{
f^* \HHom{A \TTens B}{-} \ar[r]^-{\fh} & \HHom{f^*(A \TTens B)}{f^*(-)} \ar[r]^-{\fp^o} & \HHom{f^* A \TTens f^* B}{-}
}$$
are mates in Lemma \ref{adjab} when $J_1=- \TTens (A \TTens B)$, $K_1=\HHom{A \TTens B}{-}$, $J_2=- \TTens (f^* A \TTens f^* B)$, $K_2=\HHom{f^* A \TTens f^* B}{-}$, $H=f^*$ and $H'=f^*$.
\end{lemm}
\begin{proof}
By uniqueness, it suffices to establish
the commutativity of diagram \diag{H} which is 
$$\xymatrix@R=3ex{
f^* \HHom{A \TTens B}{-} \TTens (f^* A \TTens f^* B) \ar[r]^-{\fh \TTens id} \ar[d]_{id \TTens \fp} \ar@{}[dr]|{\diag{mf}} & \HHom{f^* (A \TTens B)}{f^*(-)} \TTens (f^* A \TTens f^* B) \ar@<8ex>@/^10ex/[dd]^-{\fp^o} \ar[d]^{id \TTens \fp} \ar@{}[ddd]^{\diag{gen}}  \\
f^* \HHom{A \TTens B}{-} \TTens f^*( A \TTens B) \ar[r] \ar[d]_{\fp} \ar@{}[ddr]|(.25){\diag{\ref{L'app1}}} & \HHom{f^* (A \TTens B)}{f^*(-)} \TTens f^*(A \TTens B) \ar@<-1ex>[ddl]_{\coev^l} & \\
f^* (\HHom{A \TTens B}{-} \TTens (A \TTens B)) \ar[d]_-{\coev^l} & \HHom{f^* A \TTens f^* B}{f^*(-)} \TTens (f^* A \TTens f^* B) \ar[dl]^{\coev^l} \\  
 f^* (-) & 
}$$
\end{proof}
Now, according to Definition \ref{MorphDualPresFunc_defi}, we need to prove that the following diagram is commutative for any $A,B,K$ and $M$.
$$\xymatrix@R=3ex@C=2ex{
f^* \HHom{A}{K} \TTens f^* \HHom{B}{M} \ar[dd]_{\fp} \ar[r]^-{\fh \TTens \fh} \ar@{}[ddrr]|{\diagram \label{f^*dualpres_diag}} & \HHom{f^* A}{f^* K} \TTens \HHom{f^* B}{f^* M} \ar[r]^-{\dd} & \HHom{f^* A \TTens f^* B}{f^* K \TTens f^* M} \ar[d]^-{\fp} \\
 & & \HHom{f^* A \TTens f^* B}{f^* (K \TTens M)} \\
f^*( \HHom{A}{K} \TTens \HHom{B}{M}) \ar[r]^-{\dd} & f^* \HHom{A \TTens B}{K \TTens M} \ar[r]^-{\fh} & \HHom{f^* (A \TTens B)}{f^* (K \TTens M)} \ar[u]^{\fp^o}
}$$
We apply Lemma \ref{cube} to the cube
$$\xymatrix@!C=2ex@!R=3ex{
\cC_1 \times \cC_1 \ar[rr]^{- \TTens -} \ar[dr]^*-{\labelstyle f^* \times f^*} & & \cC_1 \ar[dr]^{f^*} & \\
 & \cC_2 \times \cC_2 \ar@2[ur]^{\fp} \ar[rr]^{- \TTens -} & & \cC_2 \\
 \cC_1 \times \cC_1 \ar[uu]^{\stackrel{\labelstyle (- \TTens A)}{\times(- \TTens B)}} \ar[dr]_{f^* \times f^*} & & & \\
 & \cC_2 \times \cC_2 \ar@2[luuu]|{\rule[-.5ex]{0ex}{2ex} \fp \times \fp} \ar[uu]_(.3){\stackrel{\labelstyle (- \TTens f^* A)}{\times (- \TTens f^* B)}} \ar[rr]_{- \TTens -} & & \cC_2 \ar@2[lluu]_{\exch} \ar[uu]|(.7){- \TTens (f^* A \TTens f^* B)}
}
\hspace{-2ex}
\xymatrix@!C=2ex@!R=3ex{
\cC_1 \times \cC_1 \ar[rr]^{- \TTens -} & & \cC_1 \ar[dr]^{f^*} & \\
 & & & \cC_2 \\
\cC_1 \times \cC_1 \ar[uu]|{\stackrel{\labelstyle (- \TTens A)}{\times (- \TTens B)}} \ar[dr]_{f^* \times f^*} \ar[rr]^{- \TTens -} & & \cC_1 \ar@2[lluu]^{\exch} \ar[uu]^(.7){- \TTens (A \TTens B)} \ar[dr]_{f^*} & \\
 & \cC_2 \times \cC_2 \ar@2[ur]^{\fp} \ar[rr]_{- \TTens -} & & \cC_2 \ar@2[luuu]|{\rule[-.5ex]{0ex}{2ex} \fp \circ (id \TTens \fp)} \ar[uu]_{- \TTens (f^* A \TTens f^*B)}
}$$
which is easily shown to be commutative using the compatibility of $f^*$ with the symmetry and the associativity of the tensor product (diagrams \diag{\ref{fgc}} and \diag{\ref{f^*assoc}}). This yields a new cube whose commutativity is the one of \diag{\ref{f^*dualpres_diag}} using Lemma \ref{mates3_lemm} to recognize the right face.
\end{proof}

\subsection{Projection formula} \label{ProjForm_sec}

The last theorem we want to prove is a projection formula.

\begin{theo} (projection formula) \label{projform0_theo}
Let $\cC_1$ and $\cC_2$ be closed categories and let $f^*$ be a monoidal functor (Assumption \assumption{\ref{assufmonoidal}}{f}) from $\cC_1$ to $\cC_2$ satisfying Assumptions \assumption{\ref{assuAdjf*f*}}{f}, \assumption{\ref{assuAdjf*f!}}{f} and \assumption{\ref{assuProjFormIso}}{f}. Let $K$ and $M$ be objects of $\cC_1$. Then, $\q$ is a morphism of duality preserving functors from 
$$\pair{Id \TTens Id, \dd_{K,M}} \pair{f_* \times Id,\rr_K \times id}$$
to 
$$\pair{f_*,\rr_{K \TTens M}} I_{\ssp_{K,M}} \pair{Id \TTens Id, \dd_{f^! K, f^* M}} \pair{Id \times f^*, id \times \fh_M}.$$ 
\end{theo}

\begin{coro} (projection formula for Witt groups) 
By Proposition \ref{samemorphismWitt_prop}, under strongness assumptions necessary for their existence and when $\ssp_{K,M}$ is an isomorphism, the pull-back, push-forward and product on Witt groups defined in corollaries \ref{PullBack0Witt_coro}, \ref{PushForward0Witt_coro} and Proposition \ref{existProduct} satisfy the equality 
$$f_*^W (I_{\ssp_{K,M}}^W(x . f^*_W(y))) = f_*^W (x). y$$
in $\W((\cC_1)_{K \TTens M})$
for all $x \in \W((\cC_2)_{f^!K})$ and $y \in \W((\cC_1)_M)$. 
\end{coro}

\begin{proof}[Proof of Theorem \ref{projform0_theo}]
According to Definition \ref{MorphDualPresFunc_defi}, we have to prove that 
the following diagram is commutative for any $A$, $B$, $K$ and $M$.
$$\xymatrix@R=3ex@C=2ex{
f_*\HHom{A}{f^! K} \TTens \HHom{B}{M} \ar[d]_{\q} \ar[r]^-{\rr} \ar@{}[ddrr]|{\diagram \label{projFormDualPres_diag}} & \HHom{f_*A}{K} \TTens \HHom{B}{M} \ar[r]^-{\dd} & \HHom{f_* A \TTens B}{K \TTens M} \\
f_*(\HHom{A}{f^! K} \TTens f^* \HHom{B}{M}) \ar[d]_{\fh} & & \HHom{f_* (A \TTens f^* B)}{K \TTens M} \ar[u]_{\q^o} \\
f_*(\HHom{A}{f^! K} \TTens \HHom{f^* B}{f^* M}) \ar[r]^-{\dd} &  f_* \HHom{A \TTens f^* B}{f^! K \TTens f^* M} \ar[r]^-{\ssp} & f_* \HHom{A \TTens f^* B}{f^! (K \TTens M)} \ar[u]_{\rr}
}$$
Decomposing $\rr$ as in its definition (see Theorem \ref{PushForward0_theo}) and
using a few \diag{mf} diagrams and Diagram \diag{\ref{G2app2}}, it is easy to reduce 
this to the commutativity of the following diagram (with $N=f^!K$).
$$\xymatrix@R=3ex@C=2ex{
f_* \HHom{A}{N} \TTens \HHom{B}{M} \ar[r]^-{\ff} \ar[d]_{\q} \ar@{}[ddrr]|{\diagram \label{muPiBet_diag}} & \HHom{f_* A}{f_* N} \TTens \HHom{B}{M} \ar[r]^{\dd} & \HHom{f_* A \TTens B}{f_* N \TTens M} \\
f_*(\HHom{A}{N} \TTens f^*\HHom{B}{M}) \ar[d]_{\fh} & & \HHom{f_*(A \TTens f^* B)}{f_* N \TTens M} \ar[u]_{\q^o} \\
f_*(\HHom{A}{N} \TTens \HHom{f^* B}{f^* M}) \ar[r]^-{\dd} & f_* \HHom{A \TTens f^* B}{N \TTens f^* M} \ar[r]^-{\ff} & \HHom{f_*(A \TTens f^* B)}{f_*(N \TTens f^* M)} \ar[u]_{\q^{-1}}
}$$
which requires the following preliminary lemmas.
\begin{lemm} \label{mates4_lemm}
The composition
$$\xymatrix@R=2ex@C=5ex{
f_*(-_1 \TTens -_2) \TTens f_*(A \TTens f^* B) \ar[r]^-{\fg} & f_*((-_1 \TTens -_2) \TTens (A \TTens f^*B)) \ar[d]^{\exch} \\
& f_* ((-_1 \TTens A) \TTens (-_2 \TTens f^*B)) 
}$$
and the composition
$$\xymatrix@C=2ex{
f_*(\HHom{A}{-_1} \TTens \HHom{f^*B}{-_2}) \ar[r]^-{\dd} & f_* \HHom{A \TTens f^* B}{-_1 \TTens -_2} \ar[r]^-{\ff} & \HHom{f_*(A \TTens f^* B)}{f_*(-_1 \TTens -_2)} 
}$$
are mates in Lemma \ref{adjab} when $J_1=(-_1 \TTens A)\times (-_2 \TTens f^*B)$, $K_1=\HHom{A}{-_1} \times \HHom{f^* B}{-_2}$, $J_2=- \TTens f_*(A \TTens f^* B)$, $K_2=\HHom{f_*(A \TTens f^* B)}{-}$, and $H=H'=f_*(-_1\TTens -_2)$.
\end{lemm}
\begin{proof}
By uniqueness, it suffices to prove Diagram \diag{H} which is here
$$\xymatrix@R=4ex@C=-1ex{
f_*(\HHom{A}{-_1}\TTens \HHom{f^* B}{-_2})\TTens f_*(A \TTens f^* B) \ar[r]^-{\dd} \ar[d]_{\fg} \ar@{}[dr]|{\diag{mf}} & f_* \HHom{A \TTens f^* B}{-_1 \TTens -_2} \TTens f_* (A \TTens f^* B) \ar[d]_{\fg} \ar@<9ex>@/^10ex/[dd]^-{\ff} \\
f_*((\HHom{A}{-_1} \TTens \HHom{f^*B}{-_2}) \TTens (A \TTens f^* B)) \ar[r]^-{\dd} \ar[d]_{\exch} \ar@{}[ddr]|(.25){\diag{\ref{Happ4}}} & f_* ( \HHom{A \TTens f^* B}{-_1 \TTens -_2} \TTens (A \TTens f^* B)) \ar[ddl]|(.5){\rule[-2ex]{0ex}{4ex} \rule{8ex}{0ex}}_(.3){\ev^l} \ar@{}[d]|{\diag{\ref{H'app1}}} \\
f_*((\HHom{A}{-_1}\TTens A)\TTens (\HHom{f^* B}{-_2} \TTens f^* B)) \ar[d]_{\ev^l} & \HHom{f_*(A \TTens f^* B)}{f_*(-_1 \TTens -_2)} \TTens f_* (A \TTens f^* B) \ar[dl]^{\ev^l} \\
f_*(-_1 \TTens -_2) & \\
}$$
\end{proof}
\begin{lemm} \label{mates5_lemm}
The composition
$$\xymatrix@C=2ex{
(-_1 \TTens -_2) \TTens f_*(A \TTens f^* B) \ar[r]^-{\q^{-1}} & (-_1 \TTens -_2) \TTens (f_* A \TTens B) \ar[r]^-{\exch} & (-_1 \TTens f_* A) \TTens (-_2 \TTens B)
}$$
and the composition
$$\xymatrix{
\HHom{f_* A}{-_1} \TTens \HHom{B}{-_2} \ar[r]^-{\dd} & \HHom{f_* A \TTens B}{-_1 \TTens -_2} \ar[r]^-{(\q^{-1})^o} & \HHom{f_*(A \TTens f^* B)}{-_1 \TTens -_2}
}$$
are mates in Lemma \ref{adjab} when $J_1=(-_1 \TTens f_* A)\times (-_2 \TTens B)$, $K_1=\HHom{f_* A}{-_1}\times \HHom{B}{-_2}$, $J_2=- \TTens f_* (A \TTens f^* B)$, $K_2=\HHom{f_*(A \TTens f^* B)}{-}$, and $H=H'=\TTens$.
\end{lemm}
\begin{proof}
By uniqueness, it suffices to prove that diagram \diag{H} commutes, which is 
$$\xymatrix@R=4ex@C=-1ex{
(\HHom{f_*A}{-_1}\TTens \HHom{B}{-_2})\TTens f_*(A \TTens f^* B) \ar[r]^-{\dd} \ar[d]_{\q^{-1}} \ar@{}[dr]|{\diag{mf}} & \HHom{f_* A \TTens B}{-_1 \TTens -_2} \TTens f_* (A \TTens f^* B) \ar[d]_{\q^{-1}} \ar@<10ex>@/^8ex/[dd]^-{(\q^{-1})^o} \\
(\HHom{f_*A}{-_1} \TTens \HHom{B}{-_2}) \TTens (f_* A \TTens B) \ar[r]^-{\dd} \ar[d]_{\exch} \ar@{}[ddr]|(.25){\diag{\ref{Happ4}}} & \HHom{f_* A \TTens B}{-_1 \TTens -_2} \TTens (f_* A \TTens B) \ar[ddl]|(.47){\rule[-2ex]{0ex}{4ex} \rule{4ex}{0ex}}_(.3){\ev^l} \ar@{}[d]|{\diag{gen}} \\
(\HHom{f_* A}{-_1}\TTens f_* A)\TTens (\HHom{B}{-_2} \TTens B) \ar[d]_{\ev^l} & \HHom{f_*(A \TTens f^* B)}{-_1 \TTens -_2} \TTens f_* (A \TTens f^* B) \ar[dl]^{\ev^l} \\
(-_1 \TTens -_2) & \\
}$$
\end{proof}
The cube
$$\xymatrix@!C=2ex@!R=3ex{
\cC_2 \times \cC_2 \ar[rr]^{\TTens} \ar[dr]|(.6)*!/_3ex/{\stackrel{\labelstyle (-\TTens f^* M)}{\labelstyle \times (- \TTens f^* B)}} & & \cC_2 \ar[dr]^*!/_1ex/{\labelstyle - \TTens (f^* M \TTens f^* B)} & \\
 & \cC_2 \TTens \cC_2 \ar@2[ur]_{\exch} \ar[rr]^{\TTens} & & \cC_2 \\
 \cC_1 \times \cC_1 \ar[uu]^{f^* \times f^*} \ar[dr]_*!/^1ex/{\labelstyle (-\TTens M)\TTens(-\TTens B)} & & & \\
 & \cC_1 \times \cC_1 \ar@2[luuu]|{\fp^{-1}\times \fp^{-1}} \ar[uu]_{f^* \times f^*} \ar[rr]_{\TTens} & & \cC_1 \ar@2[lluu]_{\fp^{-1}} \ar[uu]_{f^*}
}
\hspace{-7ex}
\xymatrix@!C=2ex@!R=3ex{
\cC_2 \times \cC_2 \ar[rr]^{\TTens} & & \cC_2 \ar[dr]^*!/_1ex/{\labelstyle -\TTens (f^*M\TTens f^*B)} & \\
 & & & \cC_2 \\
\cC_1 \times \cC_1 \ar[uu]^{f^* \times f^*} \ar[dr]_*!/^1ex/{\labelstyle (-\TTens M)\times (- \TTens B)} \ar[rr]^{\TTens} & & \cC_1 \ar@2[lluu]|{\fp^{-1}} \ar[uu]|{f^*} \ar[dr]|(.4)*!/^2ex/{\labelstyle -\TTens (M \TTens B)} & \\
 & \cC_1 \times \cC_1 \ar@2[ur]^{\exch} \ar[rr]_{\TTens} & & \cC_1 \ar@2[luuu]|{\stackrel{\labelstyle (id \TTens \fp^{-1})}{\circ \fp^{-1}}} \ar[uu]_{f^*}
}$$
is commutative by a classical exercise on monoidal functors, symmetry and associativity. Applying Lemma \ref{cube} to it, we get a new cube, in which the morphism of functors on the right face is $(id \TTens \fp^{-1})\circ \q$, simply because $\q$ is the mate of $\fp^{-1}$ (by Proposition \ref{projFormMorphExists}) and $(id \TTens \fp^{-1})$ is just a change in the parameter. The commutativity of the cube thus obtained can be rewritten as the commutativity of the cube
$$\xymatrix@!C=2ex@!R=3ex{
\cC_2 \times \cC_1 \ar[rr]^{f_* \times Id} \ar[dr]^(.4)*!/_1ex/{\labelstyle Id \times f^*} & & \cC_1 \times \cC_1 \ar[dr]^{\TTens} & \\
 & \cC_2 \times \cC_2 \ar@2[ur]^*!/_1ex/{\labelstyle \q^{-1}} \ar[rr]^{f_*(-\TTens -)} & & \cC_1 \\
\cC_2 \times \cC_1 \ar[uu]^{\stackrel{\labelstyle (-\TTens A)}{\times (-\TTens B)}} \ar[dr]_{Id \times f^*} & & & \\
 & \cC_2 \times \cC_2 \ar@2[luuu]|{id \times \fp} \ar[uu]_(.4){\stackrel{\labelstyle (-\TTens A)}{\times (-\TTens B)}} \ar[rr]_{f_*(-\TTens -)} & & \cC_1 \ar@2[lluu]_{x} \ar[uu]|(.7){- \TTens f_*(A \TTens f^*B)}
}
\hspace{-5ex}
\xymatrix@!C=2ex@!R=3ex{
\cC_2 \times \cC_1 \ar[rr]^{f_* \times Id} & & \cC_1 \times \cC_1 \ar[dr]^{\TTens} & \\
 & & &  \cC_1 \\
\cC_2 \times \cC_1 \ar[uu]^(.7)*!/_.5ex/{\stackrel{\labelstyle (-\TTens A)}{\labelstyle \times (- \TTens B)}} \ar[dr]_{Id \times f^*} \ar[rr]^{f_* \times Id} & & \cC_1 \times \cC_1 \ar@2[lluu]|{\fg \times id} \ar[uu]^(.75){\stackrel{\labelstyle (- \TTens f^*A)}{\times (-\TTens B)}} \ar[dr]_{\TTens} & \\
 & \cC_2 \times \cC_2 \ar@2[ur]^{\q^{-1}} \ar[rr]_{f_*(-\TTens-)} & & \cC_1 \ar@2[luuu]_{y} \ar[uu]_{-\TTens f_*(A \TTens f^* B)}
}$$
where $x$ and $y$ are the first compositions in lemmas \ref{mates4_lemm} and \ref{mates5_lemm}. Applying one more time Lemma \ref{cube} to this new cube yields another one whose commutativity is the one of \diag{\ref{muPiBet_diag}}.
\end{proof}

\section{Various reformulations} \label{variousReform_sec}

This section is devoted to reformulations of the main theorems in view of applications. The corollaries on Witt groups show that these reformulations are useful. The reformulations are obtained by changing the duality at the source or at the target using functors $I_\iota$ (see Lemma \ref{isoKisoWitt}) with $\iota$ adapted to the situation. For this reason, deriving those reformulations from the original theorems is very easy and most of the times follows from simple commutative diagrams involving functors of the form $I_\iota$. We therefore leave all proofs of Sections \ref{fofUnitObj_sec}, \ref{reformcorrectdual_sec} and \ref{reformFinalObject_sec} as exercises for the reader. 

\subsection{Reformulations using $f^!$ of unit objects} \label{fofUnitObj_sec}

In this section, we use the unit objects in the monoidal categories to formulate the main theorems in a different way. In the application to Witt groups and algebraic geometry, this will 
relate $f^!$ to $f^*$ using canonical sheaves. Let $\cB''$ denote the subcategory of $\cB'$ in which the morphisms $f$ are such that \assumption{\ref{assuProjFormIso}}{f} is satisfied ($\q$ is an isomorphism).

For any morphism $f: X \to Y$ in $\cB'$, we define 
$$\rel'_f = f^!(\one_Y).$$ 
To every morphism $f$, we associate a number $d_f$ such that $d_{gf}=d_f+d_g$ for composable morphisms $f$ an $g$. For example, for every object $X$ we may choose a number $d_X$ and set $d_f=d_X - d_Y$. We define
$$\rel_f =T^{-d_f} f^!(\one_Y).$$
\begin{rema}
In applications to schemes, $d_f$ can be the relative dimension of a morphism $f$ and $d_X$ can be the relative dimension of a smooth $X$ over a base scheme. In that case, $\rel'_f$ is isomorphic to a shifted line bundle and $\rel_f$ to a line bundle (the canonical sheaf). This explains why we introduce $d_f$ and $\rel_f$. On the other hand, it is always possible to set $d_f=0$ for any morphism $f$, in which case $\rel_f=\rel'_f$, thus statements in terms of $\rel_f$ also apply to $\rel'_f$. 
\end{rema}

For any two composable morphisms $f$ and $g$ in $\cB''$, let us denote by $i'_{g,f}$ the composition
$$\rel'_f \TTens f^* (\rel'_g) =f^! (\one_Y) \TTens f^*(\rel'_g)  \stackrel{\ssp}{\to} f^! (\rel'_g) = f^! g^! (\one_Z) \stackrel{\ec}{\simeq} (gf)^!(\one_Z) = \rel'_{gf}$$
and 
$$i_{g,f}: \rel_f \TTens f^*(\rel_g) \arr{} \rel_{gf}$$
the composition obtained by the same chain of morphisms and
then desuspending. 

\begin{lemm}\label{cocyclecon}
For any composable morphisms 
$\xymatrix@1{X \ar[r]^{f} & Y \ar[r]^{g} & Z \ar[r]^{h} & V}$ in $\cB ''$, 
the diagram of isomorphisms
$$\xymatrix@R=0ex@C=10ex{
 & \rel_f \TTens f^*(\rel_{hg}) \ar[dl]_-{i_{hg,f}} & \rel_f \TTens f^*(\rel_g \TTens g^*(\rel_h)) \ar[l]^-{id \TTens f^* i_{h,g}} \ar[dd]_{\wr}^{(id \TTens a)\circ (id \TTens \fp^{-1})} \\
\rel_{hgf} & & \\
 & \rel_{gf} \TTens (gf)^* (\rel_h) \ar[ul]^-{i_{h,gf}} & \rel_f \TTens f^*(\rel_g) \TTens (gf)^* (\rel_h) \ar[l]_-{i_{g,f} \TTens id} 
}$$
is commutative. In other words, $i_{g,f}$ (as well as $i'_{g,f}$) satisfies a cocycle condition.
\end{lemm}

We can then reformulate the main theorems \ref{PushForward0_theo}, \ref{compof_*0_theo}, \ref{basechange0_theo} and \ref{projform0_theo} and their corollaries for Witt groups as follows. 

\begin{defi} (push-forward) \label{PushForward1_defi}
Let $f:X \to Y$ be a morphism in $\cB''$ and $K$ be an object in $\cC_Y$. We define a duality preserving functor $\cC_{X,\rel'_f \TTens f^* K} \to \cC_{Y,K}$ by the composition $\pair{f_*,\rr_K}I_{\ssp_{\one,K}}$. We denote it simply by $\pair{f_*}$ in the rest of the section ($K$ is always understood).
\end{defi}

\begin{defi} (push-forward for Witt groups) \label{PushForward1Witt_defi}
When the dualities are strong and $\ssp_{\one,K}$ is an isomorphism, $\pair{f_*}$ induces a push-forward 
$$f_*^W:\W^{i+d_f}(X,\rel_f \TTens f^* K) \to \W^{i}(Y,K)$$
on Witt groups (recall Proposition \ref{DTKTDK} to switch from suspending $\rel'_f$ to the suspended duality used in Definition \ref{Wgraded} of shifted Witt groups).
\end{defi}

\begin{theo} (composition of push-forwards) \label{Composition1_theo}
Let $f: X \to Y$ and $g: Y \to Z$ be morphisms in $\cB''$ and $K$ be an object of $\cC_Z$. The morphism $\eb_{g,f}:(gf)_* \to g_* f_*$ (see Section \ref{Composition}) is duality preserving between the compositions of duality preserving functors in the diagram 
$$\xymatrix@R=3ex{
\cC_{X,\rel'_{f}\TTens f^*(\rel'_{g}\TTens g^* K)} \ar[d]_{\pair{f_*}} \ar[r]^{I_{\iota}} & \cC_{X,\rel'_{gf}\TTens (gf)^* K} \ar[d]^{\pair{(gf)_*}} \ar@{=>}[dl]_{\eb_{g,f}} \\
\cC_{Y,\rel'_g \TTens g^* K} \ar[r]_{\pair{g_*}} & \cC_{Z,K}
}$$
where $\iota$ is the composition
$$\xymatrix@R=3ex@C=10ex{
\rel'_{f} \TTens f^*(\rel'_g \TTens g^*K) \ar[r]^-{Id \TTens \fp^{-1} } & \rel'_{f} \TTens f^* \rel'_{g} \TTens f^* g^* K \ar[r]^-{i'_{g,f} \TTens \ea_{g,f}} & \rel'_{gf} \TTens (gf)^* K.
}$$
\end{theo}

\begin{coro} (composition of push-forwards for Witt groups)
For $f$ and $g$ as in the theorem, the push-forwards on Witt groups of Definition \ref{PushForward1Witt_defi} satisfy $g_*^W f_*^W = (gf)_*^W I_{\iota}^W$.
\end{coro}

\begin{theo} (base change)
In the situation of Theorem \ref{basechange0_theo}, with furthermore $g$ and $\bar{g}$ in $\cB''$, the morphism of functors $\eps: g_*f^* \to \bar{f}^* \bar{g}_*$ is duality preserving between the two compositions of duality preserving functors in the diagram
$$\xymatrix@R=3ex{
 & \cC_{V,\bar{f}^* (\rel'_g \TTens g^* K)} \ar[rr]^{I_\iota} & & \cC_{V,\rel'_{\bar{g}}\TTens \bar{g}^*f^* K} \ar[dr]^{\pair{\bar{g}_*}} & \\
\cC_{X,\rel'_g \TTens g^* K} \ar[ur]^{\pair{\bar{f}^*}} \ar[rr]_{\pair{g_*}} & & \cC_{Z,K} \ar[rr]_{\pair{f^*}} \ar@{=>}[u]^{\eps} & & \cC_{Y,f^* K}
}$$
where $\iota$ is defined as the composition
$$\xymatrix@R=3ex{\bar{f}^* (\rel'_g \TTens g^* K) \ar[r]^{\fp^{-1}} & \bar{f}^*\rel'_{g} \TTens \bar{f}^*g^* K \ar[d]_{Id \TTens \xi} \\ 
 & \bar{f}^* \rel'_g \TTens \bar{g}^* f^* K \ar[r]^-{\gam_{\one}\TTens Id} & \bar{g}^! f^* \one \TTens \bar{g}^* f^* K \ar@{}[r]|-{\simeq} &  \rel'_{\bar{g}}\TTens \bar{g}^*f^* K.
}$$
\end{theo}

\begin{coro} (base change for Witt groups)
In the situation of the theorem, when furthermore $\gam_\one$ is an isomorphism, the push-forwards and pull-backs for Witt groups of Definition \ref{PushForward1Witt_defi} and Corollary \ref{PullBack0Witt_coro} satisfy $f^*_W g_*^W = \bar{g}_*^W I_{\iota}^W \bar{f}^*_W$.
\end{coro}

\begin{theo} (projection formula)
In the situation of Theorem \ref{projform0_theo}, $\q$ is a morphism of duality preserving functors between the two compositions in the diagram
$$\xymatrix@R=3ex@C=10ex{
\cC_{\rel'_{f} \TTens f^* K} \times \cC_{f^* M} \ar[r]^-{\pair{\otimes}} & \cC_{\rel'_{f} \TTens f^*K \TTens f^* M} \ar[r]^-{I_{Id \TTens \fp_{K,M}}} & \cC_{\rel'_{f} \TTens f^*(K \TTens M)} \ar[d]^{\pair{f_*}} \\
\cC_{\rel'_{f} \TTens f^* K} \times \cC_{M} \ar[u]^{\pair{Id \times f^*}} \ar[r]^-{\pair{f_* \times Id}} & \cC_{K} \times \cC_{M} \ar@{=>}[u]^{\q} \ar[r]^-{\pair{\TTens}} & \cC_{K\TTens M}
}$$
\end{theo}

\begin{coro} (projection formula for Witt groups)
In the situation of the theorem, the push-forward, pull-back and product on Witt groups of Definition \ref{PushForward1Witt_defi}, Corollary \ref{PullBack0Witt_coro} and Proposition \ref{existProduct} satisfy $f_*^W I_{Id \TTens \fp_{K,M}}^W(x.f^*_W(y))=f_*^W(x).y$ in $\W^{i+j}(Y, K \TTens M)$ for any $x\in \W^{i+d_f}(X,\rel_f \TTens f^* K)$ and $y \in \W^j(Y,M)$.
\end{coro}

\subsection{Reformulations using morphisms to correct the dualities} \label{reformcorrectdual_sec}

We now reformulate the main results using a convenient categorical setting 
which allows us to display the previous resultsin an even nicer way.
This will be useful in applications.
 
We define new categories.

\begin{defi}
Let $\cB^*$ be the category whose objects are pairs $(X,K)$ with $X \in \cB$, $K\in \cC_X$, and whose morphisms from $(X,K)$ to $(Y,L)$ are pairs $(f,\phi)$ where $f: X \to Y$ is a morphism in $\cB$ and $\phi:f^*L \to K$ is a morphism in $\cC_X$. The composition is defined by $(g,\psi)(f,\phi)=(gf, \phi \circ f^*(\psi) \circ(\ea_{g,f})^{-1})$. The identity morphism on $(X,K)$ is $(Id_X,Id_K)$ and the composition is clearly associative.
\end{defi}

\begin{defi}
Let $\cB^!$ be the category whose objects are pairs $(X,K)$ with $X \in \cB'$, $K\in \cC_X$, and whose morphisms from $(X,K)$ to $(Y,M)$ are pairs $(f,\phi)$ where $f: X \to Y$ is a morphism in $\cB'$ and $\phi:K \to f^! M$ is a morphism in $\cC_X$. The composition is defined by $(g,\psi)(f,\phi)=(gf,\ec_{g,f}\circ f^!(\psi) \circ \phi)$. The identity morphism on $(X,K)$ is $(Id_X,Id_K)$ and the composition is clearly associative.
\end{defi}

For applications to Witt groups, we need the objects defining the dualities to be dualizing and the duality preserving functors to be strong. We thus define two more categories.

\begin{defi}
Let $\cB^*_W$ denote the subcategory of $\cB^*$ in which the objects $(X,K)$ are such that $K$ is dualizing ($\bid_K$ is an isomorphism) and the morphisms $(f,\phi):(X,K) \to (Y,L)$ are such that $\phi$ and $\fh_{f,L}$ are isomorphisms. 
\end{defi}

\begin{defi}
Let $\cB^!_W$ denote the subcategory of $\cB^!$ in which the objects $(X,K)$ are such that $K$ is dualizing ($\bid_K$ is an isomorphism) and the morphisms $(f,\phi)$ are such that $\phi$ is an isomorphism.
\end{defi}

Using the categories $\cB^!$ and $\cB^*$ and their subcategories, the main theorems \ref{PullBack0_theo}, \ref{PushForward0_theo}, \ref{compof^*_theo}, \ref{compof_*0_theo}, \ref{basechange0_theo}, \ref{projform0_theo} and their corollaries on Witt groups can be rephrased as follows. 

\begin{defi} (pull-back)
For any morphism $(f,\phi):(X,K) \to (Y,L)$ in $\cB^*$, we define a duality preserving functor $\func{f,\phi}^*: \cC_{Y,L} \to \cC_{X,K}$
by the composition
$$\xymatrix@C=10ex{
\cC_{Y,L} \ar[r]^-{\pair{f^*,\fh_L}} & \cC_{X,f^* L} \ar[r]^-{I_{\phi}} & \cC_{X,K}.
}$$
\end{defi}

\begin{defi} (pull-back on Witt groups) \label{PullBack2Witt_defi}
For any morphism $(f,\phi):(X,K) \to (Y,L)$ in $\cB^*_W$, the duality preserving functor $\func{f,\phi}$ induces a morphism $\func{f,\phi}^*_W: \W^i(Y,L) \to \W^i(X,K)$ on Witt groups.
\end{defi}

\begin{theo} (composition of pull-backs)
Let $(f,\phi)$ and $(g,\psi)$ be composable morphisms in $\cB^*$.
The morphism $\ea_{g,f}:f^* g^* \to (gf)^*$ is a morphism of duality preserving functors as in the following diagram. 
$$\xymatrix@R=3ex{
 & \cC_{Y,L} \ar[dr]^{\func{f,\phi}^*} \ar@{=>}[d]^{\ea_{g,f}} & \\
\cC_{Z,M} \ar[rr]_{\func{(g,\psi)(f,\phi)}^*} \ar[ur]^{\func{g,\psi}^*} &  & \cC_{X,K} 
}$$
\end{theo}

\begin{coro} (composition of pull-backs for Witt groups)
Let $(f,\phi)$ and $(g,\psi)$ be composable morphisms in $\cB^*_W$. Then the pull-backs for Witt groups of Definition \ref{PullBack2Witt_defi} satisfy $\func{(g,\psi)(f,\phi)}^*_W=\func{f,\phi}^*_W \func{g,\psi}^*_W$. In other words, $\W^*$ is a contravariant functor from $\cB^*_W$ to graded abelian groups. 
\end{coro}

\begin{defi} (push-forward)
For any morphism $(f,\phi):(X,K) \to (Y,L)$ in $\cB^!$, we define a duality preserving functor $\func{f,\phi}_*: \cC_{X,K} \to \cC_{Y,L}$ by the composition 
$$\xymatrix@C=10ex{
\cC_{X,K} \ar[r]^-{I_{\phi}} & \cC_{X,f^! L} \ar[r]^-{\pair{f_*,\rr_L}} & \cC_{Y,L}.
}$$
\end{defi}

\begin{defi} (push-forward for Witt groups) \label{PushForward2Witt_defi}
For any morphism $(f,\phi):(X,K) \to (Y,L)$ in $\cB^!_W$, the duality preserving functor $\func{f,\phi}_*$ induces a morphism on Witt groups $\func{f,\phi}^W_*:\W^i(X,K) \to \W^i(X,L)$.
\end{defi}

\begin{theo} (composition of push-forwards)
Let $(f,\phi)$ and $(g,\psi)$ be composable morphisms in $\cB^!$.
The morphism $\eb_{g,f}:(gf)_* \to g_* f_*$ is a morphism of duality preserving functors as in the following diagram. 
$$\xymatrix@R=3ex{
 & \cC_{Y,L} \ar[dr]^{\func{g,\psi}_*} & \\
\cC_{X,K} \ar[rr]_{\func{(g,\psi)(f,\phi)}_*} \ar[ur]^{\func{f,\phi}_*} & \ar@{=>}[u]^(.4){\eb_{g,f}}  & \cC_{Z,M} 
}$$
\end{theo}

\begin{coro} (composition of push-forwards for Witt groups)
Let $(f,\phi)$ and $(g,\psi)$ be composable morphisms in $\cB^!_W$. Then the push-forwards for Witt groups of Definition \ref{PushForward2Witt_defi} satisfy $\func{(g,\psi)(f,\phi)}_*^W=\func{g,\psi}_*^W \func{f,\phi}_*^W$. In other words, $\W^*$ is a covariant functor from $\cB^!_W$ to graded abelian groups. 
\end{coro}

\begin{theo} (base change)
Let $(f,\phi)$ and $(\bar{f},\bar{\phi})$ be morphisms in $\cB^*$, $(g,\psi)$, and $(\bar{g},\bar{\psi})$ be morphisms in $\cB^!$ fitting in the diagram 
$$\xymatrix@R=3ex{
(V,N) \ar[r]^{(\bar{g},\bar{\psi})} \ar[d]_{(\bar{f},\bar{\phi})} & (Y,L) \ar[d]^{(f,\phi)} \\
(X,M) \ar[r]_{(g,\psi)} & (Z,K)
}$$
such that $f \bar{g} = g \bar{f} \in \cB$, such that Assumptions \assumption{\ref{epsIso}}{f,g} is satisfied
and such that the diagram
$$\xymatrix@R=0ex@C=10ex{
 & \bar{f}^* M \ar[dl]_{\bar{\phi}} \ar[r]^{\bar{f}^*(\psi)} & \bar{f}^* g^! K \ar[dd]^{\gam_K} \\ 
N \ar[dr]_{\bar{\psi}} & & \\
 & \bar{g}^! L & \bar{g}^! f^* K \ar[l]^{\bar{g}^!(\phi)} 
}$$
is commutative. Then the morphism $\eps:f^* g_* \to \bar{g}_* \bar{f}^*$ is duality preserving between the compositions of duality preserving functors in the diagram.
$$\xymatrix@R=3ex{
\cC_{V,N} \ar[r]^{\func{\bar{g},\bar{\psi}}_*} & \cC_{Y,L} \\
\cC_{X,M} \ar[u]^{\func{\bar{f},\bar{\phi}}^*} \ar[r]_{\func{g,\psi}_*} & \cC_{Z,K} \ar[u]_{\func{f,\phi}^*} \ar@{=>}[ul]^{\eps}
}$$
\end{theo}

\begin{coro} (base change for Witt groups)
In the situation of the theorem, assuming furthermore that $(f,\phi)$ and $(\bar{f},\bar{\phi})$ are in $\cB^*_W$, $(g,\psi)$ and $(\bar{g},\bar{\psi})$ are in $\cB^!_W$ and $\gam_K$ is an isomorphism, the pull-backs and push-forwards on Witt groups of definitions \ref{PullBack2Witt_defi} and \ref{PushForward2Witt_defi} satisfy
$\func{\bar{g},\bar{\psi}}_*^W \func{\bar{f},\bar{\phi}}^*_W = \func{f,\phi}^*_W \func{g,\psi}_*^W$.
\end{coro}

\begin{theo} (projection formula)
Let $f:X \to Y$ be a morphism in $\cB''$. Let $K,M \in \cC_X$ and $L,N \in \cC_Y$ be objects, $\phi:f^* L \to K$, $\psi: M \to f^! N$ and $\chi: M \TTens K \to f^! (N \TTens L)$ be morphisms such that
$$\xymatrix@R=3ex{
M \TTens K \ar[r]^{\chi} & f^! (N \TTens L) \\
M \TTens f^* L \ar[u]^{Id \TTens \phi} \ar[r]^{\psi \TTens Id} & f^! N \TTens f^* L \ar[u]_{\ssp_{N,L}}
}$$
is commutative. 

Then, the morphism $\q$ is duality preserving between the compositions in the diagram
$$\xymatrix@R=3ex@C=3ex{
 & \cC_{X,M} \times \cC_{X,K} \ar[rr]^-{\pair{\TTens}} & & \cC_{X,M \TTens K} \ar[dr]^-{\func{f,\chi}_*} \\ 
\cC_{X,M} \times \cC_{Y,L} \ar[ur]^{Id \times \func{f,\phi}^*} \ar[rr]^-{\func{f,\psi}_* \times Id} & & \cC_{Y,N} \times \cC_{Y,L} \ar@{=>}[u]^{\q} \ar[rr]^{\pair{\TTens}} & & \cC_{Y, N\TTens L}
}$$
\end{theo}

\begin{coro} (projection formula for Witt groups)
In the situation of the theorem, assuming furthermore that $(f,\phi)$ is in $\cB^*_W$ and $(f,\psi)$ and $(f,\chi)$ are in $\cB^!_W$, the push-forward, pull-back and product for Witt groups of Definition \ref{PushForward2Witt_defi}, Definition \ref{PullBack2Witt_defi} and Proposition \ref{existProduct} satisfy $\func{f,\chi}_*^W (x.\func{f,\phi}^*_W(y))=\func{f,\psi}_*^W(x).y$ in $\W^{i+j}(Y, N \TTens L)$ for any $x\in \W^{i}(X,M)$ and $y \in \W^j(Y,L)$.
\end{coro}

\subsection{Reformulations using a final object} \label{reformFinalObject_sec}

When the category $\cB$ has a final object denoted by $\pt$ that is also a final object for the subcategory $\cB''$, there is a reformulation of the main results using absolute 
(that is relative to the final object) canonical objects rather than relative ones as in section \ref{fofUnitObj_sec}. 
For any object $X\in \cB'$, let $p_X$ denote the unique morphism $X \to \pt$. We define $d_X=d_{p_X}$, $\abs'_X=\rel'_{p_X}$ and $\abs_X=\rel_{p_X}$.

\begin{defi}
Let $\cB_*$ be the category whose objects are pairs $(X,K)$ with $X \in \cB''$, $K\in \cC_X$, and whose morphisms from $(X,K)$ to $(Y,L)$ are pairs $(f,\phi)$ where $f: X \to Y$ is a morphism in $\cB''$ and $\phi:K \to f^*L$ is a morphism in $\cC_X$. The composition is defined by $(g,\psi)(f,\phi)=(gf, \eb_{g,f}) f^*(\psi) \phi $. The identity of $(X,K)$ is $(Id_X,Id_K)$ and the composition is clearly associative.
\end{defi}

\begin{defi} \label{B_*^W_defi}
Let $\cB^W_*$ be the subcategory of $\cB_*$ in which the objects $(X,K)$ are such that $\abs_X \TTens K$ is dualizing and the morphisms $(f,\phi):(X,K) \to (Y,L)$ are such that $\iota$ defined as the composition
$$\xymatrix@R=3ex@C=12ex{\abs'_X \TTens K \ar[r]^-{(\ec_{f,p_Y})^{-1}\TTens \phi} & f^!\abs'_Y \TTens f^* L \ar[r]^{\ssp_{\abs'_Y,L}} & f^!(\abs'_Y \TTens L).
}$$
is an isomorphism.
\end{defi}

\begin{defi} (push-forward) \label{PushForward3_defi}
For any morphism $(f,\phi):(X,K) \to (Y,L)$ in $\cB_*$, we define a duality preserving functor $\func{f,\phi}_*: \cC_{X,\abs'_X \TTens K} \to \cC_{Y,\abs'_Y \TTens L}$ by the composition 
$$\xymatrix@C=10ex{
\cC_{X,\abs'_X\TTens K} \ar[r]^-{I_{\iota}} & \cC_{X,f^!(\abs'_Y \TTens L)} \ar[r]^-{\pair{f_*,\rr_{\abs'_Y \TTens L}}} & \cC_{Y,L}.
}$$
where $\iota$ is defined as in Definition \ref{B_*^W_defi} (but not necessarily an isomorphism). 
\end{defi}

\begin{defi} (push-forward for Witt groups) \label{PushForward3Witt_defi}
Let $(f,\phi)$ be a morphism in $\cB_*^W$. Then the duality preserving functor $\func{f,\phi}_*$ of Definition \ref{PushForward3_defi} induces a morphism
$\func{f,\phi}_*^W : \W^{i+d_X}(X,\abs_X \TTens K) \to \W^{i+d_Y}(Y,\abs_Y \TTens L)$
on Witt groups.
\end{defi}

\begin{theo} (composition of push-forwards)
Let $(f,\phi)$ and $(g,\psi)$ be composable morphisms in $\cB_*$.
Then the morphism $\eb_{g,f}:(gf)_* \to g_* f_*$ is a morphism of duality preserving functors as in the following diagram. 
$$\xymatrix@R=3ex{
 & \cC_{Y,\abs'_Y \TTens L} \ar[dr]^{\func{g,\psi}_*} & \\
\cC_{X,\abs'_X \TTens K} \ar[rr]_{\func{(g,\psi)(f,\phi)}_*} \ar[ur]^{\func{f,\phi}_*} & \ar@{=>}[u]^(.4){\eb_{g,f}} & \cC_{Z,\abs'_Z \TTens M} 
}$$
\end{theo}

\begin{coro} (composition of push-forwards for Witt groups)
Let $(f,\phi)$ and $(g,\psi)$ be composable morphisms in $\cB_*^W$. Then the push-forwards of Definition \ref{PushForward3Witt_defi} satisfy
$\func{(g,\psi)(f,\phi)}_*^W=\func{g,\psi}_*^W \func{f,\phi}_*^W$. In other words, $\W^*$ is a covariant functor from $\cB_*^W$ to graded abelian groups.
\end{coro}

To state a base change theorem, we need the following. Let $(f,\phi)$ and $(\bar{f},\bar{\phi})$ be morphisms in $\cB^*$ and $(g,\psi)$ and $(\bar{g},\bar{\psi})$ be morphisms in $\cB_*$ with sources and targets as on the diagram
$$\xymatrix@R=3ex{
(V,\abs'_V \TTens N) \ar[d]_{(\bar{f},\bar{\phi})} & (V,N) \ar[r]^{(\bar{g},\bar{\psi})} & (Y,L) & (Y,\abs'_Y \TTens L) \ar[d]^{(f,\phi)} \\
(X,\abs'_X \TTens M) & (X,M) \ar[r]_{(g,\psi)} & (Z,K) & (Z,\abs'_Z \TTens K) 
}$$
such that $f \bar{g} = g \bar{f} \in \cB$ and such that Assumption \assumption{\ref{epsIso}}{f,g} is satisfied. The morphism $\phi$ induces a morphism $``\phi"$ defined by the composition 
$$\xymatrix@R=3ex@C=15ex{
\bar{f}^*(\abs'_X \TTens g^* K) \ar@{-->}[d]_{``\phi"} \ar[r]^{\fp^{-1}} & \bar{f}^* \abs'_X \TTens \bar{f}^* g^* K \ar[r]^{\bar{f}^*(\ec_{p_Z,g})^{-1}\TTens Id} & \bar{f}^* g^! \abs'_Z \TTens \bar{f}^* g^* K \ar[d]^{\gam \TTens \xi^{-1}} \\
\bar{g}^! (\abs'_Y \TTens L) & \bar{g}^! f^* (\abs'_Z \TTens K) \ar[l]^{\bar{g}^!(\phi)} & \bar{g}^! f^* \abs'_Z \TTens \bar{g}^* f^* K \ar[l]^{\fp \circ \ssp} }$$
The morphism $\bar{\psi}$ induces a morphism $``\bar{\psi}"$ defined by the composition
$$\xymatrix@C=15ex{
\abs'_V \TTens N \ar[r]^{(\ec_{p_Y,\bar{g}})^{-1}\TTens \bar{\psi}} & \bar{g}^!\abs'_Y \TTens \bar{g}^* L \ar[r]^{\ssp} & \bar{g}^! (\abs'_Y \TTens L)
}$$

\begin{theo} (base change)
Let $(f,\phi)$, $(\bar{f},\bar{\phi})$, $(g,\psi)$ and $(\bar{g},\bar{\psi})$ be as above. 
We assume that the diagram
$$\xymatrix@R=3ex{
\abs'_V \TTens N \ar[r]^{``\bar{\psi}"} & \bar{g}^! (\abs'_Y \TTens L) \\
\bar{f}^* (\abs'_X \TTens M) \ar[u]^{\bar{\phi}} \ar[r]^{\psi} & \bar{f}^* (\abs_X \TTens g^* K) \ar[u]_{``\phi"}
}$$
is commutative.
Then, $\eps: f^* g_* \to \bar{g}_* \bar{f}^*$ is a morphism of duality preserving functors as in the following diagram.
$$\xymatrix@R=3ex{
\cC_{V,\abs'_V \TTens N} \ar[r]^{\func{\bar{g},\bar{\psi}}_*} & \cC_{Y,\abs'_Y \TTens L} \\
\cC_{X,\abs'_X \TTens M} \ar[u]^{\func{\bar{f},\bar{\phi}}^*} \ar[r]_{\func{g,\psi}_*} & \cC_{Z,\abs'_Z \TTens K} \ar[u]_{\func{f,\phi}^*} \ar@{=>}[ul]^{\eps}
}$$
\end{theo}

\begin{coro} (base change for Witt groups)
In the situation of the theorem, assuming furthermore that $(f,\phi)$ and $(\bar{f},\bar{\phi})$ are in $\cB^*_W$ and $(g,\psi)$ and $(\bar{g},\bar{\psi})$ are in $\cB_*^W$ and $\gam_{\abs_Z}$ is an isomorphism. Then the pull-backs and push-forwards on Witt groups of definitions \ref{PullBack2Witt_defi} and \ref{PushForward3Witt_defi} satisfy $\func{\bar{g},\bar{\psi}}_*^W\func{\bar{f},\bar{\phi}}^*_W = \func{f,\phi}^*_W \func{g,\psi}_*^W$.
\end{coro}

\begin{theo} (projection formula)
Let $f:X \to Y$ be a morphism in $\cB''$. Let $K,M \in \cC_X$ and $L,N \in \cC_Y$ be objects, $\phi:f^* L \to K$, $\psi: M \to f^* N$ and $\chi: M \TTens K \to f^* (N \TTens L)$ be morphisms such that
$$\xymatrix@R=3ex{
M \TTens K \ar[r]^{\chi} \ar[d]_{\psi \TTens Id} & f^* (N \TTens L) \ar[d]^{\fp^{-1}} \\
f^* N \TTens K \ar[r]^{Id \TTens \phi} & f^* N \TTens f^* L 
}$$
is commutative. 
Then, the morphism $\q$ is duality preserving between the compositions in the diagram
$$\xymatrix@R=3ex@C=2.5ex{
 & \cC_{X,\abs'_X \TTens M} \times \cC_{X,K} \ar[rr]^-{\pair{\TTens}} & & \cC_{X,\abs'_X \TTens M \TTens K} \ar[dr]^-{\func{f,\chi}_*} \\ 
\cC_{X,\abs'_X \TTens M} \times \cC_{Y,L} \ar[ur]^{Id \times \func{f,\phi}^*} \ar[rr]^-{\func{f,\psi}_* \times Id} & & \cC_{Y,N} \times \cC_{Y,L} \ar@{=>}[u]^{\q} \ar[rr]^{\pair{\TTens}} & & \cC_{Y, N\TTens L}
}$$
\end{theo}

\begin{coro} (projection formula for Witt groups)
In the situation of the theorem, assuming furthermore that $(f,\phi)$ is in $\cB^*_W$ and $(f,\psi)$ and $(f,\chi)$ are in $\cB_*^W$. Then the pull-back, push-forwards and product for Witt groups of definitions \ref{PullBack2Witt_defi}, \ref{PushForward3Witt_defi} and Proposition \ref{existProduct} satisfy $\func{f,\chi}_*^W(x.\func{f,\phi}^*_W(y))=\func{f,\psi}_*^W(x).y$. 
\end{coro}

\appendix

\section{Signs in the category of complexes} \label{SignComplexes}

Let $\cE$ be an exact category $\cE$ admitting infinite countable direct sums and products, with a tensor product $\ttens$ adjoint to an internal Hom (denoted by $\hhom$) in the sense of Definition \ref{defiAdjBif}.  
In this section, we explain how certain signs have to be chosen in order to induce a suspended symmetric monoidal closed structure on the category of chain complexes of $\cE$ (resp. the homotopy category, the derived category). We use homological complexes, as it is the usual convention in the articles about Witt groups, so the differential of a complex is
$$d^A_i : A_i \to A_{i-1}.$$
The suspension functor $T$ is
$$(TA)_n = A_{n-1}$$
and the tensor product and the internal Hom are given by
$$(A \TTens B)_n = \bigoplus_{i+j=n} A_i \ttens B_j$$
and
$$\HHom{A}{B}_n = \prod_{j-i=n} \hhom(A_i , B_j).$$
In table \ref{defSigns}, we give a possible choice of signs for the translation functor, tensor product, the associativity morphism (denoted by $\asso$), the symmetry morphism and the adjunction morphism (denoted by $\ath$), and what it induces on the internal Hom using Proposition \ref{suspACRBExists}. 
In table \ref{compatSigns} further below, 
we state the compatibility that these signs must satisfy 
to ensure that all axioms of suspended symmetric monoidal closed categories 
considered hold.

\begin{table}[h]
\setlength{\extrarowheight}{4pt}
\begin{tabular}{|c|c|c|c|}
\hline
Definition of & Sign & Choice & Locus \\
\hline \hline
$TA$ & $\ep^{T}_i$ & $-1$ & $d_{i+1}^{TA}= \ep^{T}_i d_i A$ \\
\hline
$A \TTens B$ & $\ep^{1 \TTens}_{i,j}$ & $1$ & $\ep^{1 \TTens}_{i,j} d^A_i \ttens id_{B_j}$ \\
 & $\ep^{2 \TTens}_{i,j}$ & $(-1)^i$ & $\ep^{2 \TTens}_{i,j} id_{A_i} \ttens d^B_j$ \\
\hline
$\tpp_{1,A,B}$ & $\ep^{\tpp1}_{i,j}$ & $1$ & $\ep^{\tpp1}_{i,j} id_{A_i\ttens B_j}$ \\
\hline
$\tpp_{2,A,B}$ & $\ep^{\tpp2}_{i,j}$ & $(-1)^i$ & $\ep^{\tpp2}_{i,j} id_{A_i\ttens B_j}$ \\
\hline
$\asso_{A,B,C}$ & $\ep^{\asso}_{i,j,k}$ & $1$ & $\ep^{\asso}_{i,j,k} ((A_i \ttens B_j) \ttens C_k \to A_i \ttens (B_j \ttens C_k))$ \\
\hline
$\cc_{A,B}$ & $\ep^{\cc}_{i,j}$ & $(-1)^{ij}$ & $\ep^{\cc}_{i,j}(A_i\ttens B_j \to B_j\ttens A_i)$ \\
\hline
$\ath_{A,B,C}$ & $\ep^{\ath}_{i,j}$ & $(-1)^{i(i-1)/2}$ & $\ep^{\ath}_{i,j}(\Hom(A_i\ttens B_j, C_{i+j})$ \\
 & & & $\to \Hom(A_i, \hhom (B_j,C_{i+j})))$ \\
\hline
$\HHom{A}{B}$ & $\ep^{1 \hhom}_{i,j}$ & $1$ & $\ep^{1 \hhom}_{i,j} (d^A_{i+1})^\sharp$ \\
& $\ep^{2 \hhom}_{i,j}$ & $(-1)^{i+j+1}$ & $\ep^{2 \hhom}_{i,j} (d^B_{j})_\sharp$ \\
\hline
$\thh_{1,A,B}$ & $\ep^{\thh1}_{i,j}$ & $1$ & $\ep^{\thh1}_{i,j} id_{\hhom(A_i,B_j)}$ \\
\hline
$\thh_{2,A,B}$ & $\ep^{\thh2}_{i,j}$ & $(-1)^{i+j}$ & $\ep^{\thh2}_{i,j} id_{\hhom(A_i,B_j)}$ \\
\hline
\end{tabular}
\caption{Sign definitions} \label{defSigns}
\end{table}
Balmer \cite{Balmer00}, \cite{Balmer01}, Gille and Nenashev ,\cite{Gille02},
\cite{Gille03} always consider strict dualities, that is $\ep^{\thh1}=1$.
The signs chosen in \cite[§2.6]{Balmer01} imply that 
$\ep_{i,0}^{1 \hhom}=1$.
The choices made by \cite[Example 1.4]{Gille03} are $\ep^{1\TTens}_{i,j}=1$ 
and $\ep^{2\TTens}_{i,j}=(-1)^{i}$. In \cite[p. 111]{Gille02} the signs
$\ep^{1\hhom}_{i,j}=1$ and $\ep^{2\hhom}_{i,j}=(-1)^{i+j+1}$ are chosen.
Finally, the sign chosen
for $\bid$ in \cite[p. 112]{Gille02} corresponds via our definition
of $\bid$ (see Section \ref{BidualIsomorphism}) to the equality $\ep^{\ath}_{j-i,i}\ep^{\ath}_{i,j-i}\ep^{\cc}_{j-i,i}
=(-1)^{j(j-1)/2}$. It is possible to choose the signs in a way 
compatible with all these choices and our formalism. It is given in the third column of Table \ref{defSigns}. 
More precisely, we have the following theorem.

\begin{theo}
Let $a,b \in \{+1,-1 \}$. Then
$$\begin{array}{ll}
\ep^{1 \TTens}_{i,j}=1 & \ep^{\tpp1}_{i,j}=a \\  
\ep^{2 \TTens}_{i,j}=(-1)^i & \ep^{\tpp2}_{i,j}=a(-1)^i \\ 
\ep^{1 \hhom}_{i,j}=1 & \ep^{\thh1}_{i,j}=1 \\  
\ep^{2 \hhom}_{i,j}=(-1)^{i+j+1} &  \ep^{\thh2}_{i,j}=a(-1)^{i+j} \\  
\ep^{\ath}_{i,j}=b(-1)^{i(i-1)/2} &  \ep^{\cc}_{i,j}=(-1)^{ij} \\
\ep^{T}_i=-1 & \\  
\end{array}$$
satisfies all equalities of Table \ref{compatSigns}
as well as $\ep^{\ath}_{j-i,i}\ep^{\ath}_{i,j-i}\ep^{\cc}_{j-i,i}
=(-1)^{j(j-1)/2}$. Therefore, for any exact category $\cE$
the category of chain complexes $Ch(\cE)$ and its
bounded variant $Ch_b(\cE)$ may be equipped 
with the entire structure of suspended symmetric monoidal category discussed in Section \ref{TensProHom}.
Moreover, all signs may be chosen in a compatible 
way with all the above sign choices of Balmer, Gille and Nenashev. 
\end{theo}
\begin{proof}
Straightforward.
\end{proof}

\begin{table}[h]
\setlength{\extrarowheight}{4pt}
\begin{tabular}{|c|c|c|c|}
\hline
 &compatibility & reason \\
\hline \hline
1 & $\ep^{1 \TTens}_{i,j} \ep^{1 \TTens}_{i,j-1} \ep^{2 \TTens}_{i,j} \ep^{2 \TTens}_{i-1,j} = -1$ & $A \TTens B$ is a complex \\
\hline 
2 & $\ep^{\TTens 1}_{i,j}\ep^{\TTens 1}_{i,j+k}\ep^{\TTens 1}_{i+j,k}\ep^{\asso}_{i,j,k}\ep^{\asso}_{i-1,j,k}$ & $\asso$ is a morphism \\
3 & $\ep^{\TTens 2}_{i,j}\ep^{\TTens 1}_{j,k}\ep^{\TTens 1}_{i+j,k}\ep^{\TTens 2}_{i,j+k}\ep^{\asso}_{i,j,k}\ep^{\asso}_{i, j-1,k}$ & \\
4 & $\ep^{\TTens 2}_{i+j,k}\ep^{\TTens 2}_{j,k}\ep^{\TTens 2}_{i,j+k}\ep^{\asso}_{i,j,k}\ep^{\asso}_{i,j,k-1}$ & \\
\hline
5 & $\ep^{\asso}_{i,j,k}\ep^{\asso}_{i,j+k,l}\ep^{\asso}_{j,k,l}\ep^{\asso}_{i,j,k+l}\ep^{\asso}_{i+j,k,l}=1$ & the pentagon of \cite[p. 252]{MacLane98} commutes \\
\hline
6 & $\ep^{1 \TTens}_{i,j} \ep^{2 \TTens}_{j,i} \ep^{\cc}_{i,j} \ep^{\cc}_{i-1,j} = 1$ & $\cc_{A,B}$ is a morphism \\
7 & $\ep^{1 \TTens}_{j,i} \ep^{2 \TTens}_{i,j} \ep^{\cc}_{i,j} \ep^{\cc}_{i,j-1} = 1$ & \\
\hline
8 & $\ep^{\cc}_{i,j} \ep^{\cc}_{j,i} = 1$ & $\cc$ is self-inverse \\
\hline
9 & $\ep^{\cc}_{j,k} \ep^{\cc}_{i,k} \ep^{\cc}_{i+j,k}\ep^{\asso}_{i,j,k} \ep^{\asso}_{k,i,j}\ep^{\asso}_{i,k,j}=1$ & the hexagons of \cite[p. 253]{MacLane98} commutes \\
\hline 
10 & $\ep^{T}_i \ep^{T}_{i+j} \ep^{1 \TTens}_{i,j} \ep^{1 \TTens}_{i+1,j} \ep^{\tpp1}_{i,j} \ep^{\tpp1}_{i-1,j} = 1$ & $\tpp_{1,A,B}$ is a morphism \\
11 & $\ep^{T}_{i+j} \ep^{2 \TTens}_{i,j} \ep^{2 \TTens}_{i+1,j} \ep^{\tpp1}_{i,j} \ep^{\tpp1}_{i,j-1} = 1$ & \\
\hline
12 & $\ep^{T}_j \ep^{T}_{i+j} \ep^{2 \TTens}_{i,j} \ep^{2 \TTens}_{i,j+1} \ep^{\tpp2}_{i,j} \ep^{\tpp2}_{i,j-1} = 1$ & $\tpp_{2,A,B}$ is a morphism \\
13 & $\ep^{T}_{i+j} \ep^{1 \TTens}_{i,j} \ep^{1 \TTens}_{i,j+1} \ep^{\tpp2}_{i,j} \ep^{\tpp2}_{i-1,j} = 1$ & \\
\hline
14 & $\ep^{\tpp1}_{i,j} \ep^{\tpp1}_{i,j+1} \ep^{\tpp2}_{i,j} \ep^{\tpp2}_{i+1,j} = -1$ & the square in Definition \ref{defiSuspBif} \\
 & & anti-commutes \\
\hline
15 & $\ep^{\tpp1}_{i,j} \ep^{\tpp1}_{i+j,k} \ep^{\tpp1}_{i,j+k} \ep^{\asso}_{i,j,k} \ep^{\asso}_{i+1,j,k} =1$ & $\diag{assoc}$ et al. commute \\
16 & $\ep^{\tpp2}_{i,j} \ep^{\tpp1}_{i+j,k} \ep^{\tpp2}_{i,j+k} \ep^{\tpp1}_{j,k} \ep^{\asso}_{i,j,k} \ep^{\asso}_{i,j+1,k}=1$ & \\
17 & $\ep^{\tpp2}_{i+j,k} \ep^{\tpp2}_{i,j+k} \ep^{\tpp2}_{j,k} \ep^{\asso}_{i,j,k} \ep^{\asso}_{i,j,k+1}=1$ & \\ 
\hline
18 & $\ep^{\tpp1}_{i,j} \ep^{\tpp2}_{j,i} \ep^{\cc}_{i,j} \ep^{\cc}_{i+1,j} = 1$ & the square $\diag{s}$ commutes \\
\hline
19 & $\ep^{1 \hhom}_{i,j} \ep^{1 \hhom}_{i,j-1} \ep^{2 \hhom}_{i,j} \ep^{2 \hhom}_{i+1,j} = -1$ & $\HHom{A}{B}$ is a complex \\
\hline
20 & $\ep^{1 \TTens}_{i,j} \ep^{2 \TTens}_{i,j} \ep^{1 \hhom}_{j-1,i+j-1} \ep^{\ath}_{i,j-1} \ep^{\ath}_{i-1,j}= -1$ & $\ath$ is well defined \\
21 & $\ep^{1 \TTens}_{i,j} \ep^{2 \hhom}_{j,i+j} \ep^{\ath}_{i-1,j} \ep^{\ath}_{i,j} = 1$ & \\
\hline
\end{tabular} 
\caption{Sign definitions} \label{compatSigns}
\end{table}

These structures trivially pass to the homotopy category.
If the exact category $\cE$ one considers
has enough injective and projective objects, one obtains a 
left derived functor of the tensor product and a right derived 
functor of the internal Hom which are exact in both variables.

\bibliography{MonCatDual}

\providecommand{\bysame}{\leavevmode\hbox to3em{\hrulefill}\thinspace}
\providecommand{\MR}{\relax\ifhmode\unskip\space\fi MR }
\providecommand{\MRhref}[2]{%
  \href{http://www.ams.org/mathscinet-getitem?mr=#1}{#2}
}
\providecommand{\href}[2]{#2}
\begin{thebibliography}{10}

\bibitem{Balmer00}
P.~Balmer, \emph{{T}riangular {W}itt {G}roups {P}art 1: The 12-{T}erm
  {L}ocalization {E}xact {S}equence}, K-Theory \textbf{4} (2000), no.~19,
  311--363.

\bibitem{Balmer01}
\bysame, \emph{Triangular {W}itt groups. {II}. {F}rom usual to derived}, Math.
  Z. \textbf{236} (2001), no.~2, 351--382.

\bibitem{Calmes08b_pre}
B.~Calm\`es and J.~Hornbostel, \emph{{P}ush-forwards for {W}itt groups of
  schemes}, preprint, 2008.

\bibitem{Eilenberg66b}
S.~Eilenberg and G.~M. Kelly, \emph{A generalization of the functional
  calculus}, J. Algebra \textbf{3} (1966), 366--375.

\bibitem{Fausk03}
H.~Fausk, P.~Hu, and J.~P. May, \emph{Isomorphisms between left and right
  adjoints}, Theory Appl. Categ. \textbf{11} (2003), No. 4, 107--131
  (electronic).

\bibitem{Gille02}
S.~Gille, \emph{On {W}itt groups with support}, Math. Annalen \textbf{322}
  (2002), 103--137.

\bibitem{Gille03}
S.~Gille and A.~Nenashev, \emph{Pairings in triangular {W}itt theory}, J.
  Algebra \textbf{261} (2003), no.~2, 292--309.

\bibitem{Kelly71}
G.~M. Kelly and S.~Mac~Lane, \emph{Coherence in closed categories}, Journal of
  Pure and Applied Algebra \textbf{1} (1971), no.~1, 97--140.

\bibitem{MacLane98}
S.~Mac~Lane, \emph{Categories for the working mathematician}, second ed.,
  Graduate Texts in Mathematics, vol.~5, Springer-Verlag, New York, 1998.

\bibitem{Weibel94}
C.~A. Weibel, \emph{An introduction to homological algebra}, Cambridge Studies
  in Advanced Mathematics, vol.~38, Cambridge University Press, Cambridge,
  1994.

\end{thebibliography}

\end{document}